\documentclass[10pt,a4paper]{amsart}
\usepackage{amssymb,amsmath,amsthm}
\usepackage{tikz}
\usepackage{tikz-cd}
\usepackage{enumitem}
\usepackage{amsaddr}
\usepackage{etoolbox}
\usetikzlibrary{matrix}

\usepackage[utf8]{inputenc}
\usepackage{hyperref}

\begin{document}

\newtheorem{defn}{Definition}[section]
\newtheorem{definitions}[defn]{Definitions}
\newtheorem{lem}[defn]{Lemma}
\newtheorem{construction}[defn]{Construction}
\newtheorem{prop}[defn]{Proposition}
\newtheorem*{prop*}{Proposition}
\newtheorem{thm}[defn]{Theorem}
\newtheorem{cor}[defn]{Corollary}
\newtheorem{claim}{Claim}[defn]
\newtheorem*{claim*}{Claim}
\newtheorem{algo}[defn]{Algorithm}
\theoremstyle{remark}
\newtheorem{rem}[defn]{Remark}
\theoremstyle{remark}
\newtheorem{remarks}[defn]{Remarks}
\theoremstyle{remark}
\newtheorem{notation}[defn]{Notation}
\theoremstyle{remark}
\newtheorem{exmp}[defn]{Example}
\theoremstyle{remark}
\newtheorem{examples}[defn]{Examples}
\theoremstyle{remark}
\newtheorem{dgram}[defn]{Diagram}
\theoremstyle{remark}
\newtheorem{fact}[defn]{Fact}
\theoremstyle{remark}
\newtheorem{illust}[defn]{Illustration}
\theoremstyle{remark}
\newtheorem{que}[defn]{Question}
\numberwithin{equation}{section}

\author[Sinha et al.]{Vinit Sinha, Amit Kuber, Annoy Sengupta and Bhargav Kale}
\address{Department of Mathematics and Statistics\\Indian Institute of Technology, Kanpur\\ Uttar Pradesh, India}
\email{vinitsinha20@iitk.ac.in, askuber@iitk.ac.in, annoysgp20@iitk.ac.in, kalebhargav@gmail.com}
\title[Hammocks for non-domestic string algebras]{{Hammocks for non-domestic string algebras}}

\begin{abstract}
We show that the order type of the simplest version of a hammock for string algebras lies in the class of \emph{finite description} linear orders--the smallest class of linear orders containing $\mathbf 0$, $\mathbf 1$, and that is closed under isomorphisms, finite order sum, anti-lexicographic product with $\omega$ and $\omega^*$, and shuffle of finite subsets--using condensation (localization) of linear orders as a tool. We also introduce two finite subsets of the set of bands and use them to describe the location of left $\mathbb N$-strings in the completion of hammocks.
\end{abstract}

\maketitle

\newcommand{\hgt}{\mathrm{ht}}
\newcommand{\suc}{\mathfrak{succ}}
\newcommand{\pred}{\mathfrak{pred}}
\newcommand\A{\mathcal{A}}
\newcommand\B{\bb}
\newcommand\C{\mathcal{C}}
\newcommand\Pp{\mathcal{P}}
\newcommand\D{\mathcal{D}}
\newcommand\Hamm{\hat{H}}
\newcommand\hh{\mathfrak{h}}
\newcommand\HH{\mathcal{H}}
\newcommand\RR{\mathcal{R}}
\newcommand\Red[1]{\mathrm{R}_{#1}}
\newcommand\HRed[1]{\mathrm{HR}_{#1}}
\newcommand\K{\mathcal{K}}
\newcommand\LL{\mathcal{L}}
\newcommand\M{\mathcal{M}}
\newcommand\Q{\mathcal{Q}}
\newcommand\SD{\mathcal{SD}}
\newcommand\MD{\mathcal{MD}}
\newcommand\SMD{\mathcal{SMD}}
\newcommand\T{\mathcal{T}}
\newcommand\TT{\mathfrak T}
\newcommand\ii{\mathcal I}
\newcommand\UU{\mathcal{U}}
\newcommand\VV{\mathcal{V}}
\newcommand\ZZ{\mathcal{Z}}
\newcommand{\N}{\mathbb{N}} 
\newcommand{\R}{\mathbb{R}}
\newcommand{\Z}{\mathbb{Z}}
\newcommand{\bb}{\mathfrak b}
\newcommand{\qq}{\mathfrak q}
\newcommand{\ch}{\circ_H}
\newcommand{\cg}{\circ_G}
\newcommand{\bua}[1]{\mathfrak b^{\alpha}(#1)}
\newcommand{\falpha}{{\mathfrak{f}\alpha}}
\newcommand{\fgamma}{\gamma^{\mathfrak f}}
\newcommand{\fbeta}{{\mathfrak{f}\beta}}
\newcommand{\bub}[1]{\mathfrak b^{\beta}(#1)}
\newcommand{\bla}[1]{\mathfrak b_{\alpha}(#1)}
\newcommand{\blb}[1]{\mathfrak b_{\beta}(#1)}
\newcommand{\lmin}{\lambda^{\mathrm{min}}}
\newcommand{\lmax}{\lambda^{\mathrm{max}}}
\newcommand{\xmin}{\xi^{\mathrm{min}}}
\newcommand{\xmax}{\xi^{\mathrm{max}}}
\newcommand{\lbmin}{\bar\lambda^{\mathrm{min}}}
\newcommand{\lbmax}{\bar\lambda^{\mathrm{max}}}
\newcommand{\ff}{\mathfrak f}
\newcommand{\cc}{\mathfrak c}
\newcommand{\dd}{\mathfrak d}
\newcommand{\sqsf}{\sqsubset^\ff}
\newcommand{\rr}{\mathfrak r}
\newcommand{\pp}{\mathfrak p}
\newcommand{\uu}{\mathfrak u}
\newcommand{\vv}{\mathfrak v}
\newcommand{\ww}{\mathfrak w}
\newcommand{\xx}{\mathfrak x}
\newcommand{\yy}{\mathfrak y}
\newcommand{\zz}{\mathfrak z}
\newcommand{\MM}{\mathfrak M}
\newcommand{\mm}{\mathfrak m}
\newcommand{\sbq}{\mathfrak s}
\newcommand{\tbq}{\mathfrak t}
\newcommand{\Spec}{\mathbf{Spec}}
\newcommand{\Br}{\mathbf{Br}}
\newcommand{\sk}[1]{\{#1\}}
\newcommand{\Prime}{\mathbf{Pr}}
\newcommand{\Parent}{\mathbf{Parent}}
\newcommand{\Uncle}{\mathbf{Uncle}}
\newcommand{\Cousin}{\mathbf{Cousin}}
\newcommand{\Nephew}{\mathbf{Nephew}}
\newcommand{\Sibling}{\mathbf{Sibling}}
\newcommand{\uc}{\mathrm{uc}}
\newcommand{\MCP}{\mathrm{MCP}}
\newcommand{\MSCP}{\mathrm{MSCP}}
\newcommand{\TTT}{\widetilde{\T}}
\newcommand{\la}{l}
\newcommand{\ra}{r}
\newcommand{\lb}{\bar{l}}
\newcommand{\rb}{\bar{r}}
\newcommand{\tBa}{\varepsilon^{\mathrm{Ba}}}
\newcommand{\brac}[2]{\langle #1,#2\rangle}
\newcommand{\braket}[3]{\langle #1\mid #2:#3\rangle}
\newcommand{\fin}{fin}
\newcommand{\inff}{inf}
\newcommand{\Zg}{\mathrm{Zg}(\Lambda)}
\newcommand{\Zgs}{\mathrm{Zg_{str}}(\Lambda)}
\newcommand{\STR}[1]{\mathrm{Str}(#1)}
\newcommand{\dmod}{\mbox{-}\operatorname{mod}}
\newcommand{\HQ}{\mathcal{HQ}^\mathrm{Ba}}
\newcommand{\bHQ}{\overline{\mathcal{HQ}}^\mathrm{Ba}}
\newcommand\Af{\mathcal{A}^{\ff}}
\newcommand\AAf{\bar{\mathcal{A}}^{\ff}}
\newcommand\Hf{\mathcal{H}^{\ff}}
\newcommand\Rf{\mathcal{R}^{\ff}}
\newcommand\Tf{\T^{\ff}}
\newcommand\Uf{\mathcal{U}^{\ff}}
\newcommand\Sf{\mathcal{S}^{\ff}}
\newcommand\Vf{\mathcal{V}^{\ff}}
\newcommand\Zf{\mathcal{Z}^{\ff}}
\newcommand\bVf{\overline{\mathcal{V}}^{\ff}}
\newcommand\bTf{\overline{\mathcal{T}}^{\ff}}
\newcommand{\fmin}{\xi^{\mathrm{fmin}}}
\newcommand{\fmax}{\xi^{\mathrm{fmax}}}
\newcommand{\xif}{\xi^\ff}

\newcommand{\LOfp}{\mathrm{LO}_{\mathrm{fp}}}
\newcommand\Tl{\mathbf{T}}
\newcommand\Tla{\mathbf{T}_{\la}}
\newcommand\Tlb{\mathbf{T}_{\lb}}
\newcommand\Ml{\mathbf{M}}
\newcommand\Mla{\mathbf{M}_{\la}}
\newcommand\Mlb{\mathbf{M}_{\lb}}
\newcommand\OT{\mathcal{O}}
\newcommand\LO{\mathbf{LO}}
\newcommand\rad{\mathrm{rad}_\Lambda} 
\newcommand{\rk}[1]{\mathrm{rk}(#1)}
\newcommand{\wid}[1]{\mathrm{wd}(#1)}
\newcommand{\fork}[1]{\mathrm{Str}_{\text{Fork}}^{\la}(#1)}
\newcommand{\BaB}{\mathsf{Ba(B)}}
\newcommand{\CycB}{\mathsf{Cyc(B)}}
\newcommand{\ExtB}{\mathsf{Ext(B)}}
\newcommand{\EExtB}{\mathsf{EExt(B)}}
\newcommand{\St}{\mathsf{St}(\Lambda)}
\newcommand{\StB}[1]{\mathsf{St}_{#1}(\sB)}
\newcommand{\STB}[1]{\mathsf{St}_{#1}(\xx_0,i;\sB)}
\newcommand{\stb}[3]{\mathsf{St}(#1,#2;#3)}
\newcommand{\BST}[1]{\mathsf{BSt}_{#1}(\xx_0,i;\sB)}
\newcommand{\CSt}[1]{\mathsf{CSt}_{#1}(\sB)}
\newcommand{\CST}[1]{\mathsf{CSt}_{#1}(\xx_0,i;\sB)}
\newcommand{\ASt}[1]{\mathsf{ASt}_{#1}(\sB)}
\newcommand{\AST}[1]{\mathsf{ASt}_{#1}(\xx_0,i;\sB)}
\newcommand{\OSt}[1]{\overline{\mathsf{St}}_{#1}(\sB)}
\newcommand{\OST}[1]{\overline{\mathsf{St}}_{#1}(\xx_0,i;\sB)}
\newcommand{\OsT}[3]{\overline{\mathsf{St}}(#1,#2;#3)}

\newcommand{\lB}{\ell_{\sB}}
\newcommand{\lbB}{\overline{\ell}_{\sB}}
\newcommand{\LB}{l_{\sB}}
\newcommand{\LbB}{\overline{l}_{\sB}}

\newcommand{\BalB}{\mathsf{Ba}_l(\sB)}
\newcommand{\BalbB}{\mathsf{Ba}_{\lb}(\sB)}

\newcommand{\sB}{\mathsf{B}}
\newcommand{\Str}[1]{\mathrm{Str}_{#1}(\xx_0,i;\sB)}
\newcommand{\Strd}[1]{\mathrm{Str}'_{#1}(\xx_0,i;\sB)}
\newcommand{\Strdd}[1]{\mathrm{Str}''_{#1}(\xx_0,i;\sB)}
\newcommand{\Cent}{\mathrm{Cent}(\xx_0,i;\sB)}
\newcommand{\Start}{\mathrm{Start}(\xx_0,i;\sB)}
\newcommand{\End}{\mathrm{End}(\xx_0,i;\sB)}
\newcommand{\braclB}{\brac{1}{\lB}}
\newcommand{\braclbB}{\brac{1}{\lbB}}
\newcommand{\QBa}{\Q^{\mathrm{Ba}}}
\newcommand{\Ba}{\mathcal Q_0^\mathrm{Ba}}
\newcommand{\Cyc}{\mathsf{Cyc}(\Lambda)}

\newcommand{\dLOfd}{\mathrm{dLO_{fd}}}
\newcommand{\dLOfdb}[2]{\mathrm{dLO}_{\mathrm{fd}}^{{#1}{#2}}}
\newcommand{\dLOfpb}[2]{\mathrm{dLO}_{\mathrm{fp}}^{{#1}{#2}}}
\newcommand{\lex}
{\operatorname{<_\mathrm{lex}}}
\newcommand{\LOfd}{\mathrm{LO}_{\mathrm{fd}}}

\newcommand{\Hbar}{\overline{H}_l^i(\xx_0)}
\newcommand{\Hhat}{\widehat{H}_l^i(\xx_0)}
\newcommand{\HOST}[1]{\widehat{\overline{\mathsf{St}}}_{#1}(\xx_0,i;\sB)}
\newcommand{\HHlix}{\widehat{H}_l^i(\xx_0)}
\newcommand{\HHlx}{\widehat{H}_l(\xx_0)}

\newcommand{\com}{\mathcal C}
\newcommand{\gap}{\mathcal G}
\newcommand{\EStB}{\N\mbox-\mathsf{St(B)}}
\newcommand{\ESTB}{\N\mbox-\mathsf{St}(\xx_0,i;\sB)}
\newcommand{\EOSt}{\N\mbox-\overline{\mathsf{St}}(\sB)}
\newcommand{\EOST}{\N\mbox-\overline{\mathsf{St}}(\xx_0,i;\sB)}
\newcommand{\EST}[3]{\N\mbox-\mathsf{St}(#1,#2;#3)}
\newcommand{\esTB}[1]{\N\mbox-\mathsf{St}_{#1}(\xx_0,i;\sB)}
\newcommand{\EOsT}[3]{\N\mbox-\overline{\mathsf{St}}(#1,#2;#3)}
\newcommand{\StN}{\N\mbox-\mathsf{St}(\Lambda)}

\section{Introduction}\label{sec:intro}
Let $\Lambda$ be a string algebra over an algebraically closed field $\mathcal K$. Brenner \cite{brenner1986combinatorial} introduced certain partially ordered sets known as hammocks to study factorization of maps between finite dimensional indecomposable right $\Lambda$-modules. The simplest version of hammocks introduced by Schr\"oer \cite[\S~3]{schroer2000infinite} in the context of string algebras are bounded discrete linear orders--this is the only type of hammock we will deal with in this paper. We compute the order type of a hammock for $\Lambda$ in terms of some standard order types, thus generalizing the main result of Sardar and the second author from \cite{SardarKuberHamforDom} that only dealt with the case when $\Lambda$ is domestic.

The algebra $\Lambda$ is domestic if and only if there are only finitely many bands for it. These bands are vertices of a finitary combinatorial gadget known as the bridge quiver \cite[\S~4]{schroer2000infinite}--its slight modification was used for the explicit computation of the order type in the domestic case. To extend the concept of a (finitary) bridge quiver to the non-domestic setting, a finite subset \cite[Theorem~3.1.6]{GKS20} of the set of bands was introduced in \cite[Definition~3.1.1]{GKS20}, whose elements are called \emph{prime bands}. We partition the set of bands using an equivalence relation in such a way that each equivalence class contains at least one prime band to obtain a finite set $\QBa$ of equivalence classes that is equipped with a natural reachability partial order $\preceq$. We classify the elements of $\QBa$ as domestic or non-domestic depending on whether the equivalence class is finite or infinite. The existence of a non-domestic element in $\QBa$ characterizes the non-domesticity of the algebra $\Lambda$. In this regard, the study of non-domestic string algebras is a combination of domestic and meta-$\bigcup$-cyclic string algebras (\cite[\S~3.4]{GKS20}), where the latter type of algebras are characterized as those with $\QBa$ consisting only of non-domestic elements such that no two distinct elements are $\preceq$-related.

Let $\rad$ denote the radical of the category of finite length right $\Lambda$-modules. Schr\"oer characterized \cite[Theorem~2]{schroer2000infinite} domestic string algebras as those whose radical is nilpotent. In fact, when $\Lambda$ is domestic, he showed that $\rad^{\omega\cdot(n+2)}=0$, where $n$ is the maximum length of a path in its bridge quiver, bypassing the computation of the order types of hammocks. It is conjectured \cite[Conjecture~4.4.1]{GKS20} that the stable rank of a non-domestic string algebra is strictly bounded above by $\omega^2$. The results in this paper, especially those in \S~\ref{sec: main thm},\ref{sec:completionhammock}, will be used in a future work to settle this conjecture in the affirmative.

Yet another characterization of a non-domestic (string) algebra was given by Prest \cite[Proposition~0.6]{prest1998morphisms} in terms of the existence of a \emph{factorizable system} in its radical--such a factorizable system is indexed by a bounded interval in rationals with non-empty interior.

To explain the main result of the paper, we need to set up some order-theoretic notations and conventions, for which we mostly follow Rosenstein \cite{rosenstein1982linear}.  The notations $\N$ and $\N^+$ stand for the sets of natural numbers and positive natural numbers respectively. For $n\in\N$, the notation $\mathbf{n}$ stands for the order type of the finite linear order with $n$ elements. The notation $\omega$ stands for the order type of $\N$, $\omega^*$ for its dual, $\zeta$ for the order type of the set of integers, $\eta$ for the order type of rationals, $\lambda$ for the order type of reals and $\vartheta$ for the order type of irrationals. For linear orders $L_1$ and $L_2$, the notations $L_1+L_2$ and $L_1\cdot L_2$ stand for their order sum and their anti-lexicographic product respectively. The notation $\sum_{i\in(I,<)}L_i$ stands for the order sum of linear orders $L_i$ indexed by a (possibly infinite) linear order $(I,<)$. A linear order $(L,<)$ is said to be \emph{scattered} if there is no embedding of $\eta$ in it.

In a model-theoretic study of linear orders, L\"auchli and Leonard \cite{leonard1968elementary} introduced two classes $\mathcal M_0\subset\mathcal M$ of linear orders (see \cite[Definitions~7.6, 7.19]{rosenstein1982linear}) to understand graded versions of elementary equivalence described via Ehrenfeucht-Fra\"iss\'e games. The class $\mathcal{M}_0$ is a subclass of the class of scattered linear orders whereas each linear order in $\mathcal M\setminus\mathcal M_0$ is not scattered. The class $\mathcal M_0$ appeared in \cite{SardarKuberHamforDom} as the class $\LOfp$ of \emph{finitely presented linear orders}. Its subclass $\dLOfpb{1}{1}$ consisting of bounded discrete finitely presented linear orders was characterized as the class of order types of hammocks for domestic string algebras \cite[Corollary~10.14, Proposition~10.18]{SardarKuberHamforDom}.

The main result of this paper (Theorem \ref{main}) shows that the order type of a hammock for a (non-domestic) string algebra lies in a subclass of the class $\mathcal M$ of L\"{a}uchli and Leonard. We will refer to the orders in $\mathcal M$ as \emph{finite description linear orders}, and thus use a more suggestive notation $\LOfd$ instead of $\mathcal M$; its subclass consisting of bounded discrete orders will be denoted by $\dLOfdb{1}{1}$. Each finite description order is constructed using only finitely many order-theoretic operations on a fixed finite set of linear orders (Definition \ref{defn: LOFD}). However, in contrast to the domestic case, we do not expect that every linear order in $\dLOfdb{1}{1}$ is isomorphic to a hammock for a string algebra (Question \ref{mainconversefail}).

The technique used to prove the main theorem is ``recursive reconstruction'' (Lemma \ref{recdlofdb}) of a hammock--to explain this method better we need the concept of condensation. A \emph{condensation} of a linear order $L$ is a monotone surjective map $c:L\twoheadrightarrow L'$. When a linear order is thought of as a (small) category then each of its condensations is equivalent to its localization with respect to an appropriate choice of weak equivalences. It is possible to reconstruct $L$ from the knowledge of $L'$ and all fibers of the map $c$ as $$L\cong\sum_{x\in L'}c^{-1}(x).$$ 

If $L$ is the hammock under consideration, we choose a suitable $\sB\in\QBa$ to define a split condensation $c_\sB:L\twoheadrightarrow L_\sB$ in such a way that $L_\sB\in\dLOfdb{1}{1}$, each fiber of $c_\sB$ is itself a hammock (Lemma \ref{hammockordersum}) and there are only finitely many distinct order types of fibers. Depending on whether $\sB$ is domestic or non-domestic, the condensed order $L_\sB$ is isomorphic to a finite order sum of copies of $\omega+\omega^*$ or $\omega+\zeta\cdot\eta+\omega^*$ respectively (Corollary \ref{wstrordertype}). Finiteness of $\QBa$ helps to inductively prove that the fibers of $c_\sB$ are indeed in $\dLOfdb{1}{1}$. The ``finite description'' of the order type of hammock needs several other supporting finiteness results sprinkled throughout the paper (Corollaries \ref{finitely many band free strings rel to x0,i}, \ref{OST-STB is finite}, \ref{STB+-1 is finite}, Remark \ref{rem: fin many B equiv classes} and Propositions \ref{rem: STBpm1-Cent is Bbandfree}, \ref{[mix0,y1] is finite}).

We show in Corollary \ref{neighbourcondenscommute} that for an element in $L_\sB$, the condensation of its immediate successor (resp. predecessor) in $L$ is its immediate successor (resp. predecessor) in $L_\sB$. We also identify a subset of $L_\sB$ that is in bijection with its finite condensation (see \cite[\S~4.2]{rosenstein1982linear})--such a subset is finite if and only if $\sB$ is domestic. In case $\sB$ is non-domestic, we further identify its cofinite subset, the elements of which will be called $\sB$-centers, whose order type (as a suborder of $L$) is $\eta$ (Corollary \ref{Centers are dense in a beam}).

The completion of the hammock is obtained by adding to it the so-called left $\N$-strings (Proposition \ref{completion of hammock}). Recall that every left $\N$-string in a domestic string algebra is almost periodic \cite[Proposition~2]{ringalgcom95}; this statement fails in a non-domestic string algebra. Each interval in the hammock isomorphic to $\omega$ or $\omega^*$ contributes to the completion an almost periodic left $\N$-string of the form $^\infty\bb\uu$, where $\uu$ is a string and $\bb$ is a special type of prime band (Definition \ref{defn: la, lb bands}). The remaining left $\N$-strings, which also include some almost periodic left $\N$-strings, occupy irrational locations in $\lambda$ (Proposition \ref{correspondence between diff forms of left N strings and gaps}).

Though Theorem \ref{main} generalizes \cite[Corollary~10.14]{SardarKuberHamforDom}, which computes the order type of hammocks for domestic string algebras, the former employs a recursive algorithm and thus is computationally more complex than the latter.

A finite dimensional $\mathcal K$-algebra that is presented as a bound quiver algebra has only a finite amount of data. We believe that for such algebras if one identifies a finite poset that plays the same role as $\QBa$ for string algebras then the method of ``condensation away from" elements of this poset could be used to recursively reconstruct hammocks. The class $\LOfd$ seems to be the natural class of countably infinite linear orders which admit description using only a finite amount of data.

The rest of the paper is organized as follows. The class $\LOfd$ of finite description linear orders is defined in \S~\ref{sec:linear order}; the highlight of this section is the recursive reconstruction lemma (Lemma \ref{recdlofdb}). After describing the completion of a linear order in \S~\ref{Completion of linear orders}, completions of some orders in $\dLOfdb{1}{1}$ are computed. The notations and conventions for string algebras are set up in \S~\ref{sec:String algebra} and the finiteness of the poset $(\QBa,\preceq)$ is proved in \S~\ref{sec: a finite poset}. Before describing the condensation operator $c_\sB$ in \S~\ref{sec: Condensation}, the condensed hammock is introduced in \S~\ref{sec:finiteness results}. Lemma \ref{hammockordersum} helps to decompose a hammock as an order sum of smaller hammocks, which enables setting up the induction for the computation of the order type. After introducing two special subsets of the set of prime bands in \S~\ref{sec: Neighbours}, the description of the immediate neighbours of strings in the condensed hammock is completed in \S~\ref{sec: Extending neighbours}. The definition and classification of $\sB$-centers into finitely many classes is achieved in \S~\ref{sec: Centers}. Combining all tools gathered thus far, the main result is proved in \S~\ref{sec: main thm}, where the potential impossibility of its converse is also discussed. Finally, in \S~\ref{sec:completionhammock}, based on the description of the completion of the hammock from Proposition \ref{completion of hammock}, the set of left $\N$-strings is classified into three classes and their locations in the completion are described.

\subsection*{Acknowledgements}
The third author thanks the \emph{Council of Scientific and Industrial Research (CSIR)} India - Research Grant No. 09/092(1090)/2021-EMR-I for the financial support. All authors would like to thank Naivedya Amarnani, Dawood Bin Mansoor and Nupur Jain for discussions on the topic. They would also like to thank Suyash Srivastava for careful reading of the manuscript.

\subsection*{Data Availability Statement}
Data sharing not applicable to this article as no datasets were generated or analysed during the current study.

\section{Finite description linear orders}\label{sec:linear order}
We introduced basic notations, conventions and operations on linear orders in \S~\ref{sec:intro}. We need two more finitary operations defined on linear orders. If $L_1$ and $L_2$ are non-empty linear orders, define $L_1\dotplus L_2$ to be the linear order obtained by identifying in $L_1+L_2$ the maximum element of $L_1$ with the minimum element of $L_2$, if they both exist; otherwise setting $L_1\dotplus L_2:=L_1+L_2$.

The other finitary operation is that of the shuffle of a finite set of linear orders, which we recall below from \cite{rosenstein1982linear} for the convenience of the reader. This operation will be used in the construction of finite description linear orders.

Cantor proved that $\eta$ is the only countably infinite dense linear ordering without maximum and minimum elements up to isomorphism (see \cite[Theorem~2.8]{rosenstein1982linear}). The technique used to prove this result is known as the back-and-forth method, which can also be used to prove the following. 
\begin{prop}\label{partition of Q}
\cite[Theorems~7.11, 7.13]{rosenstein1982linear} For each $n\in\omega$, there is a partition of $\eta$ into sets $\{D_i\mid1\leq i\leq n\}$ such that each $D_i$ is dense in $\eta$. Such a partition is unique in the following sense: If $(A,<)$ and $(A',<')$ are countable, unbounded, dense linear orders, $A$ is partitioned into $n$ subsets $\{D_i\mid1\leq i\leq n\}$ each of which is dense in $A$, and $A'$ is partitioned into $n$ subsets $\{D'_i\mid1\leq i\leq n\}$ each of which is dense in $A'$ then there is an order isomorphism $f:(A,<)\to(A',<')$ such that $f(D_i)=D'_i$ for each $1\leq i\leq n$.
\end{prop}
The uniqueness of the partition described in the above result enables us to combine the data of a finite set of linear orders into a single linear order as described below.
\begin{defn}\label{defn: shuffle}
\cite[Definition~7.14]{rosenstein1982linear}
Let $n\in\N$. Suppose $L_1,\cdots,L_n$ is a finite set of linear orders. If $n\in\N^+$, let $\{D_1,\cdots,D_n\}$ be a partition of $\eta$ guaranteed by Proposition \ref{partition of Q}, where each $D_i$ is dense in $\eta$. Define the \emph{shuffle}, denoted $\Xi$, of linear orders $L_1,\cdots,L_n$ as $$\Xi(L_1,\cdots,L_n):=\begin{cases}\mathbf 0&\text{if }n=0,\\\sum_{i\in\eta}L'_i \text{, where $L'_i=L_j$ when $i\in D_j$}&\text{otherwise}.\end{cases}$$
\end{defn}
The shuffle operator is a generalization of anti-lexicographic product with $\eta$, i.e., $\Xi(L_1)\cong L_1\cdot\eta$, and it ignores repetitions, $\mathbf 0$ and permutations, i.e., $$\Xi(L_1,L_1,L_2,\cdots,L_n)\cong\Xi(L_1,L_2,\cdots,L_n)\cong\Xi(L_1,L_2,\cdots,L_n,\mathbf 0)\cong\Xi(L_{\pi(1)},L_{\pi(2)},\cdots,L_{\pi(n)})$$ for a permutation $\pi$ of $\{1,2,\cdots,n\}$.

Since $\eta+\mathbf1+\eta=\eta$, for any $1\leq j\leq n$, we also get
\begin{equation}\label{shuffle property finite}
\Xi(L_1,L_2,\cdots,L_n)+L_j+\Xi(L_1,L_2,\cdots,L_n)\cong\Xi(L_1,L_2,\cdots,L_n).
\end{equation}
If $L_j=L_{j1}+L_{j2}$ then using $\eta\cdot\zeta=\eta$ and the above identity we get
\begin{equation}\label{shuffle property infinite}
(L_{j2}+\Xi(L_1,L_2,\cdots,L_n)+L_{j1})\cdot\zeta\cong\Xi(L_1,L_2,\cdots,L_n).    
\end{equation}

\begin{defn}\label{defn: LOFD}
\cite[Definition~7.19]{rosenstein1982linear} The class $\LOfd$ of \emph{finite description} linear orders is defined as the smallest subclass of linear orders closed under isomorphisms such that
\begin{enumerate}
    \item $\mathbf{0},\mathbf{1}\in\LOfd$;
    \item if $L_1,L_2\in\LOfd$ then $L_1+L_2\in\LOfd$;
    \item if $L\in\LOfd$ then $ L\cdot\omega, L\cdot\omega^*\in\LOfd$;
    \item if $L_1,L_2,\cdots, L_n\in\LOfd$ for $n\in\N^+$ then $\Xi(L_1,L_2,\cdots, L_n)\in\LOfd$.
\end{enumerate}
\end{defn}
The class $\LOfp$ of finitely presented linear orders is the subclass of $\LOfd$ whose definition omits clause $(4)$ in the above. The notation $\dLOfd$ denotes the subclass of $\LOfd$ containing only discrete linear orders. The class $\dLOfd$ of discrete finite description linear orders can be further partitioned into four subclasses, viz. $\dLOfdb{i}{j}$ for $i,j\in\{0,1\}$, where $L\in\dLOfdb{i}{j}$ only if it has $i$ minimum elements and $j$ maximum elements. In particular, $\dLOfdb{1}{1}$ is the class of \emph{bounded discrete finite description linear orders}. The orders $(\omega+\Xi(\underbrace{\zeta,\zeta,\cdots\zeta}_{n\text{ times}})+\omega^*)$ for $n\in\N$ form a simple family of examples of orders in $\dLOfdb{1}{1}$. We similarly partition $\dLOfpb{}{}$ into four subclasses.

We will use the method of recursive reconstruction described in the introduction to construct complex orders in $\dLOfdb11$. An indispensable tool to prove the main result (Theorem \ref{main}) is the following lemma which shows that, under suitable conditions, if $L$ admits a condensation $c:L\twoheadrightarrow (\omega+\Xi(\underbrace{\zeta,\zeta,\cdots\zeta}_{n\text{ times}})+\omega^*)$ with  fibers in $\dLOfdb{1}{1}$ then $L\in\dLOfdb{1}{1}$.
\begin{lem}\label{recdlofdb}
Fix $n\in\N$. Given any $(n+2)$ functions $L_j:\zeta\to\dLOfdb{1}{1}$, for $j\in\{0,1,\cdots,n+1\}$, satisfying
\begin{itemize}
    \item $L_0(-k)=L_{n+1}(k)=\mathbf{0}$ for every $k>0$;
    \item for each $j\in\{0,1,\cdots,n\}$, there exist $s_j\geq 0$ and $p_j>0$ such that $L_j(s_j+p_j+k)\cong L_j(s_j+k)$ for every $k\in\N$;
    \item for each $j\in\{1,\cdots,n+1\}$, there exist $s'_j\leq 0$ and $p'_j>0$ such that $L_j(s'_j-p'_j-k)\cong L_j(s'_j-k)$ for every $k\in\N$,
\end{itemize} we have $$L:=\sum_{k\in\zeta}L_0(k)+\Xi(\sum_{k\in\zeta}L_1(k),\cdots,\sum_{k\in\zeta}L_n(k))+\sum_{k\in\zeta}L_{n+1}(k)\in\dLOfdb{1}{1}.$$
\end{lem}

\begin{proof}
Set
$$H:=\sum_{k\in\zeta}L_0(k)\cong L_0(0)+\cdots+L_0(s_0-1)+(L_0(s_0)+\cdots+L_0(s_0+p_0-1))\cdot\omega,$$
$$R:=\sum_{k\in\zeta}L_{n+1}(k)\cong (L_{n+1}(s'_{n+1}-p'_{n+1}+1)+\cdots+L_{n+1}(s'_{n+1}))\cdot\omega^*+L_{n+1}(s'_{n+1}+1)+\cdots+L_{n+1}(0),$$ and for each $1\leq j\leq n$,
$$M_j:=\sum_{k\in\zeta}L_j(k)\cong(L_j(s'_j-p'_j+1)\cdot\omega^*+\cdots+L_j(s'_j))+L_j(s'_j+1)+\cdots+L_j(s_j-1)+(L_j(s_j)+\cdots+L_j(s_j+p_j-1))\cdot\omega.$$
It is trivially seen that $H\in\dLOfdb{1}{0}, R\in\dLOfdb{0}{1}$ and $M_j\in\dLOfdb{0}{0}$ for each $1\leq j\leq n$. Hence it follows that $L=H+\Xi(M_1,\cdots,M_n)+R\in\dLOfdb{1}{1}$.
\end{proof}

\begin{cor}
Using the notations of the above proposition, if we have $n=0$ and the images of $L_0$ and $L_1$ lie in $\dLOfpb{1}{1}$ then $L\in\dLOfpb{1}{1}$.
\end{cor}

\section{Completions of linear orders}\label{Completion of linear orders}
Recall from \cite[Definition~2.19]{rosenstein1982linear} that a linear order $L$ is \emph{complete} if each of its suborders that is bounded above has a least upper bound. Completeness of a linear order is an order-theoretic property, i.e., it is preserved and reflected by order isomorphisms \cite[Lemma~2.21]{rosenstein1982linear}. A \emph{Dedekind cut} \cite[Definition~2.22]{rosenstein1982linear} of a linear order $L$ is a pair $(X,Y)$ of non-empty intervals of $L$ whose union is $L$ such that each element of $X$ precedes every element of $Y$. A Dedekind cut $(X,Y)$ is called a \emph{gap} in $L$ if $X$ does not have a maximum element and $Y$ does not have a minimum element. Denote the set of all gaps of $L$ by $\gap(L)$. An equivalent criterion \cite[Lemma~2.23]{rosenstein1982linear} for a linear order $L$ to be complete is that $L$ is Dedekind complete, i.e., $\gap(L)=\emptyset$.

A \emph{completion} of $L$ \cite[Definition~2.31]{rosenstein1982linear}, denoted $\com(L)$, is a complete linear order containing $L$ such that no proper suborder of $\com(L)$ containing $L$ is complete. A completion $\com(L)$ of $L$ exists, the construction of one involves ``filling up'' its gaps (see the proof of \cite[Theorem~2.32(1)]{rosenstein1982linear}), and is unique up to order isomorphism \cite[Theorem~2.32(2)]{rosenstein1982linear}. The set $\gap(L)$ being a subset of $\com(L)$ inherits an order structure from $\com(L)$.

In order to identify a gap $(X,Y)$ of a linear order $L$, it suffices to find a cofinal sequence of elements of $X$ and a coinitial sequence of elements of $Y$, where $X'\subseteq X$ is \emph{cofinal} in $X$ if for every $a\in X$, there is $b\in X'$ such that $a\leq b$, and dually, $Y'$ is said to be \emph{coinitial} in $Y$ if for every $a\in Y$, there is $b\in Y'$ such that $b\leq a$.

\begin{exmp}
It is trivial to note that $\com(\omega)\cong\omega,\ \com(\omega^*)\cong\omega^*,\ \com(\zeta)\cong\zeta$. Moreover, reals are constructed as the completion of $\eta$ using cofinal/coinitial sequences, i.e., $\com(\eta)\cong\lambda.$
\end{exmp}
We will use the technique of finding cofinal/coinitial sequences to compute the completion of certain linear orders in Propositions \ref{completion of hammock} and \ref{gaps of hammocks are gaps of OST}.

The main goal of this section is to compute the completions of two classes of order types in $\dLOfdb{1}{1}$ which are important in the context of this paper, namely $(\omega+\omega^*)\cdot\mathbf n$ and $(\omega+\Xi(\zeta)+\omega^*)\cdot\mathbf n\cong(\omega+\zeta\cdot\eta+\omega^*)$. The computation of the completion of the former class of order types is easy.
\begin{exmp}\label{C(OT(domestic OST))}
$\com((\omega+\omega^*)\cdot\mathbf n)\cong(\omega+\mathbf1+\omega^*)\cdot\mathbf n.$
\end{exmp}

Given $n\in\N^+$ and non-empty linear orders $L_1,\cdots,L_n$, recall the construction of the shuffle $\Xi(L_1,\cdots,L_n)$ from Definition \ref{defn: shuffle}. Using those notations, it is easily verified that the following four types of Dedekind cuts are elements in $\gap(\Xi(L_1,\cdots,L_n))$.
\begin{enumerate}[label=G\arabic*.]
    \item $(\sum_{r\in(-\infty,r_0)\cap\mathbb Q}L'_r,\sum_{r\in(r_0,\infty)\cap\mathbb Q}L'_r)$ for $r_0\in\mathbb R\setminus\mathbb Q$;

    \item $(\sum_{r\in(-\infty,r_0)\cap\mathbb Q}L'_r,\sum_{r\in[r_0,\infty)\cap\mathbb Q}L'_r)$ for $r_0\in D_j$ if $L_j$ does not have a minimum;

    \item $(\sum_{r\in(-\infty,r_0]\cap\mathbb Q}L'_r,\sum_{r\in(r_0,\infty)\cap\mathbb Q}L'_r)$ for $r_0\in D_j$ if $L_j$ does not have a maximum;
    
    \item $(\sum_{r\in(-\infty,r_0)\cap\mathbb Q}L'_r+L_j^1,L_j^2+\sum_{r\in(r_0,\infty)\cap\mathbb Q}L'_r)$ for $r_0\in D_j$ if $(L_j^1,L_j^2)\in\gap(L_j)$.
\end{enumerate}

The following result says that in fact these are all the gaps.
\begin{prop}
Given $n\in\N^+$ and non-empty linear orders $L_1,\cdots,L_n$, if $(X,Y)\in\gap(\Xi(L_1,\cdots,L_n))$ then $(X,Y)$ is of one of the four types listed above.
\end{prop}
\begin{proof}
Define a map $\mathrm{proj}:\sum_{r\in\mathbb Q}L'_r\to\mathbb Q$ by $\mathrm{proj}(x)=r$ if $x\in L'_r$. Thus if $(X,Y)\in\gap(\Xi(L_1,\cdots,L_n))$ then $\mathrm{proj}(X)\cup\mathrm{proj}(Y)=\mathbb Q$ and $r_1\leq r_2$ whenever $r_1\in\mathrm{proj}(X)$ and $r_2\in\mathrm{proj}(Y)$. Hence $\mathrm{proj}(X)\cap\mathrm{proj}(Y)$ is either empty or singleton. If $\mathrm{proj}(X)\cap\mathrm{proj}(Y)=\{r_0\}$ for some $r_0\in\mathbb Q$ then $(X\cap L'_{r_0},Y\cap L'_{r_0})$ is a gap in $L'_{r_0}$; this gap is of the form described in G4.

Now assume that $\mathrm{proj}(X)\cap\mathrm{proj}(Y)=\emptyset$. There are three cases.
\begin{itemize}
    \item If $\mathrm{proj}(X)$ does not have a maximum element and $\mathrm{proj}(Y)$ does not have a minimum element then there exists $r_0\in\mathbb R\setminus\mathbb Q$ such that $r_1<r_0<r_2$ for every $r_1\in\mathrm{proj}(X)$ and every $r_2\in\mathrm{proj}(Y)$. This gap is of the form described in G1.

    \item If $\mathrm{proj}(X)$ has a maximum element, say $r_0$, then $(X,Y)$ is a gap if and only if $L'_{r_0}$ does not have a maximum element. This gap is of the form described in G2.

    \item If $\mathrm{proj}(Y)$ has a minimum element, say $r_0$, then $(X,Y)$ is a gap if and only if $L'_{r_0}$ does not have a minimum element. This gap is of the form described in G3.
\end{itemize}
\end{proof}

As a consequence, we have the following result, which computes the completion of $\Xi(L_1,\cdots,L_n)$.

\begin{cor}\label{completion of shuffle}
Given $n\in\N^+$ and non-empty linear orders $L_1,\cdots,L_n$, using notations of Definition \ref{defn: shuffle}, $$\com(\Xi(L_1,\cdots,L_n))\cong\sum_{r\in\mathbb R} T_r,\text{ where }T_r:=
\begin{cases}
    \mathbf 1\dotplus\com(L_j)\dotplus\mathbf 1&\text{ if }r\in D_j\text{ for some }j\in\{1,2,\cdots,n\},\\
    \mathbf 1&\text{ otherwise}.
\end{cases}$$
\end{cor}

Using the standard embedding of $\eta$ in $\lambda$, we compute the completion of a standard order type in $\dLOfdb{1}{1}$.
\begin{cor}\label{completion nondomestic beam}
Suppose $\mathcal O:=\omega+\zeta\cdot\eta+\omega^*$. Then $$\com(\mathcal O)\cong\omega+\mathbf{1}+\left(\sum_{r\in\lambda}T_r\right)+\mathbf{1}+\omega^*,\text{ where }T_r=\begin{cases}
    \mathbf 1+\zeta+\mathbf 1&\text{ if $r\in\eta$},\\
    \mathbf 1&\text{ otherwise}.
\end{cases}$$
\end{cor}
This result will be useful in \S\ref{sec:completionhammock} along with the partition of $\gap(\mathcal O)$ into three subsets given below.
\begin{align*}
\gap^+(\mathcal O)&:=\{x\in\gap(\mathcal O)\mid x\text{ is minimum of }\gap(\mathcal O)\text{ or }x\text{ is maximum of }T_r\text{ when }r\in\eta\},\\
\gap^-(\mathcal O)&:=\{x\in\gap(\mathcal O)\mid x\text{ is maximum of }\gap(\mathcal O)\text{ or }x\text{ is minimum of }T_r\text{ when }r\in\eta\},\\
\gap^0(\mathcal O)&:=\gap(\mathcal O)\setminus(\gap^+(\mathcal O)\cup\gap^-(\mathcal O))=\{T_r\mid r\in\vartheta\}.    
\end{align*}

\section{Fundamentals of string algebras}\label{sec:String algebra}
Fix an algebraically closed field $\K$. A string algebra is a $\mathcal K$-algebra $\Lambda:=\K\Q/\langle\rho\rangle$ presented as a certain quotient of the path algebra of a finite quiver $\mathcal{Q}=(Q_0,Q_1,s,t)$, where $Q_0$ is a finite set of vertices, $Q_1$ is a finite set of arrows, and $s,t:Q_1\to Q_0$ are the source and target functions respectively, by the ideal generated by a set $\rho$ of monomial relations. For technical reasons, we also choose and fix a pair of maps $\sigma,\varepsilon:Q_1\to\{1,-1\}$ satisfying certain conditions. The reader is referred to \cite[\S~2.1]{GKS20} for the definition of a string algebra as well as for notations and conventions associated with certain combinatorial entities called strings and bands. We use the notation $\St$ to denote the set of strings for $\Lambda$ and $\mathsf{Ba}(\Lambda)$ to denote the set of bands up to a cyclic permutation of its syllables. Let $\Ba$ be a fixed set of representatives in $\mathsf{Ba}(\Lambda)$. Call a cyclic permutation of an element in $\Ba$ a cycle. Denote the set of all cycles in $\Lambda$ by $\Cyc$. We also fix $\xx_0\in\St$ and a parity $i\in\{-1,1\}$ and use them whenever required. 

Strings are read from right to left. For example, if $\xx=\alpha_3\alpha_2\alpha_1$ then $\alpha_1$ is the first syllable of $\xx$ and $\alpha_3$ is the last syllable of $\xx$. For strings $\xx$ and $\yy$, we say that $\xx$ is a \emph{left substring} (resp. \emph{proper left substring}) of $\yy$, denoted $\xx \sqsubseteq_l \yy$ (resp. $\xx \sqsubset_l \yy$) if $\yy=\uu\xx$ for some (resp. positive length) string $\uu$. Dually say that $\xx$ is a \emph{right substring} (resp. \emph{proper right substring}) of $\yy$, denoted $\xx \sqsubseteq_r \yy$ (resp. $\xx \sqsubset_r \yy$) if $\yy=\xx\uu$ for some (resp. positive length) string $\uu$.

Suppose $\xx\in\St$ and $|\xx|>0$. Its \emph{sign}, denoted $\theta(\xx)\in\{1,-1\}$, is defined by $\theta(\xx)=1$ if and only if the first syllable of $\xx$ is inverse. To identify if $\xx$ has any sign changes, we define
$$\delta(\xx):=
\begin{cases}
    1&\text{ if all syllables of }\xx\text{ are inverse},\\
    -1&\text{ if all syllables of }\xx\text{ are direct},\\
    0&\text{ otherwise}.
\end{cases}$$

Hammocks are partially ordered sets introduced by Brenner \cite{brenner1986combinatorial} to study factorizations of maps between finite dimensional indecomposable modules. The simplest version of hammocks introduced by Schr\"oer \cite[\S~3]{schroer2000infinite} in the context of string algebras are linear orders. We recall the definition of this version of hammocks below, which is the main object of study in this paper.
\begin{defn}\label{hammock defn}
The left and right hammock sets of the string $\xx_0$ are defined as
$$H_l(\xx_0):=\{\xx\in\St\mid\xx=\uu\xx_0\text{ for some string }\uu\},\ H_r(\xx_0):=\{\xx\in\St\mid\xx=\xx_0\uu\text{ for some string }\uu\}.$$
The left hammock $H_l(\xx_0)$ can be equipped with a linear order $<_l$, where for $\xx,\yy\in\St$ we have $\xx<_l\yy$ if one of the following holds:
\begin{itemize}
    \item $\yy=\uu\alpha\xx$ for some string $\uu$ and $\alpha\in Q_1^-$;
    \item $\xx=\vv\beta\yy$ for some string $\vv$ and $\beta\in Q_1$;
    \item $\xx=\vv\beta\ww$ and $\yy=\uu\alpha\ww$ for some $\alpha\in Q_1^-$, $\beta\in Q_1$ and strings $\uu,\vv,\ww$.
\end{itemize}
The ordering $<_r$ on $H_r(\xx_0)$ is defined as $\xx<_r\yy$ if and only if $\xx^{-1}<_l\yy^{-1}$ in $(H_l(\xx_0^{-1}),<_l)$.
\end{defn}

We will only study the left hammock in this paper--the dual results will hold for the right hammock.

For $\xx,\yy\in H_l(\xx_0)$, denote by $\xx\sqcap_l\yy$ the maximal common left substring of $\xx$ and $\yy$. If $\xx=\ww(\xx\sqcap_l\yy)$ with $|\ww|>0$ then define $\theta(\xx\mid\yy):=\theta(\ww)$ and $\delta(\xx\mid\yy):=\delta(\ww)$.

Almost all strings in the left hammock have an immediate successor as well as an immediate predecessor.
\begin{prop}\cite[\S~2.5]{SchHam98}
The linear order $(H_l(\xx_0),<_l)$ is a bounded discrete linear order. Its minimum element, denoted $\mm_{-1}(\xx_0)$, is the longest string in $H_l(\xx_0)$ satisfying either $\delta(\mm_{-1}(\xx_0)\mid\xx_0)=-1$ or $\mm_{-1}(\xx_0)=\xx_0$, whereas its maximum element is the longest string, denoted $\MM_1(\xx_0)$, satisfying either $\delta(\MM_1(\xx_0)\mid\xx_0)=1$ or $\MM_1(\xx_0)=\xx_0$.
\end{prop}

For $\xx\in H_l(\xx_0)$, the notations $l(\xx)$ and $\lb(\xx)$ were introduced in \cite[\S~2.4]{GKS20} by comparing the length of $\xx$ with that of its immediate successor and predecessor. If the immediate successor of $\xx$ is longer than $\xx$ then there exists an inverse syllable $\alpha$ such that $\alpha\xx$ is a string, and the immediate successor of $\xx$ is the string $l(\xx):=\ww\alpha\xx$, where $\ww$ is the longest string satisfying either $|\ww|=0$ or $\delta(\ww)=-1$ such that $\ww\alpha\xx$ is a string. On the other hand, if the immediate predecessor of $\xx$ is longer than $\xx$ then there exists a direct syllable $\beta$ such that $\beta\xx$ is a string, and the immediate predecessor of $\xx$ is the string $\lb(\xx):=\ww'\beta\xx$, where $\ww'$ is the longest string satisfying either $|\ww'|=0$ or $\delta(\ww')=1$ such that $\ww'\beta\xx$ is a string.

The next result shows that intervals in hammocks contain a unique ``pivotal'' string.
\begin{prop}\label{unique string with minimal length in an interval}
Given a non-empty interval $I$ in $(H_l(\xx_0),<_l)$, there is a unique string $\uu$ in $I$ with minimal length. Moreover, $I\subseteq H_l(\uu)$.
\end{prop}
\begin{proof}
Since $\{|\xx|:\xx\in I\}$ is a non-empty subset of $\N$, it has a minimum, say $m$. If possible, let $\uu_1<_l\uu_2$ be strings in $I$ such that $|\uu_1|=|\uu_2|=m$. Then $\uu_1<_l\uu_1\sqcap_l\uu_2<_l\uu_2$ and $|\uu_1\sqcap_l\uu_2|<m$. Since $I$ is an interval, $\uu_1\sqcap_l\uu_2\in I$, a contradiction to the minimality of $m$, thus showing that $I$ contains a unique string $\uu$ with $|\uu|=m$.

For $\xx\in I$, the string $\xx\sqcap_l\uu$ lies between $\uu$ and $\xx$, and hence in $I$. Therefore, $m\leq|\xx\sqcap_l\uu|\leq|\uu|=m$. Since $\uu$ is the unique string with minimal length in $I$, we conclude that $\xx\sqcap_l\uu=\uu$, i.e., $\uu\sqsubseteq_l\xx$.
\end{proof}

Recall from \cite[\S~2.1]{GKS20} that a left $\N$-string is a sequence of syllables $\cdots\alpha_3\alpha_2\alpha_1$ such that each $\alpha_n\cdots\alpha_2\alpha_1$ is a string. Call $\alpha_i$ the $i^\text{th}$ syllable of $\xx$. Denote the set of left $\N$-strings by $\StN$. A left $\N$-string of the form $^{\infty}\bb\uu$ for some cyclic string $\bb$ and some finite string $\uu$ is called an \emph{almost periodic string}. 
\begin{defn}
Say a sequence $(\xx_n)_{n\geq1}$ of strings in $H_l(\xx_0)$ is \emph{convergent} if there is $\yy\in\StN$ such that  
\begin{enumerate}
    \item $|\xx_n|\to\infty$ as $n\to\infty$;
    \item there is a sequence $\{n_k\mid k\in\N^+\}$ such that the $k^{th}$ syllables of $\yy$ and $\xx_n$ are identical for $n\geq n_k$.
\end{enumerate}
Clearly the limit of a convergent sequence is unique, and we write $\lim_{n \to \infty} \xx_n :=\yy$.
\end{defn}
If $\widehat H_l(\xx_0)$ is the extension of $H_l(\xx_0)$ by all the left $\N$-strings containing $\xx_0$ as a left substring, it is readily noted that the ordering $<_l$ can be extended to a linear order on $\widehat H_l(\xx_0)$.
\begin{prop}\label{completion of hammock}
The linear order $(\HHlx,<_l)$ is the completion of $(H_l(\xx_0),<_l)$.
\end{prop}

\begin{proof}
Suppose $\zz$ is a left $\N$-string in $\HHlx$. Set $X:=\{\xx\in H_l(\xx_0)\mid\xx<_l\zz\}$ and $Y:=\{\yy\in H_l(\xx_0)\mid\zz<_l\yy\}$. Then clearly $\mm_{-1}(\xx_0)\in X, \MM_1(\xx_0)\in Y$ and $H_l(\xx_0)=X\sqcup Y$. Moreover, $\xx<_l\yy$ for each $\xx\in X$ and $\yy\in Y$. Thus $(X,Y)$ is a Dedekind cut in $H_l(\xx_0)$. We show that $(X,Y)$ is actually a gap in $H_l(\xx_0)$.

Since the set $\{\vv\in\St:\delta(\vv)\neq0\}$ is finite, there are infinitely many inverse as well as direct syllables in $\zz$. Hence $X_1:=\{\xx\in H_l^i(\xx_0)\mid\xx\sqsubset_l\zz,\theta(\zz\mid\xx)=1\}$ and $Y_1:=\{\yy\in H_l^i(\xx_0)\mid\yy\sqsubset_l\zz,\theta(\zz\mid\yy)=-1\}$ are infinite subsets of $X$ and $Y$ respectively.

Let $\xx\in X$ and $\ww:=\xx\sqcap_l\zz$. Since $\xx<_l\zz$, we have $\xx\leq_l\ww<_l\zz$. Hence $\theta(\zz\mid\ww)=1$ which implies $\ww\in X_1$. This shows that $X_1$ is an infinite cofinal subset of $X$, which implies that $X$ does not have a maximum. A dual argument shows that $Y_1$ is a coinitial subset of $Y$, and thus $Y$ does not have a minimum. This completes the proof that $(X,Y)$ is a gap in $H_l(\xx_0)$.

Conversely, suppose $(X,Y)$ is a gap in $H_l(\xx_0)$. Since $X\neq\emptyset$ and $X$ is unbounded above, \cite[Theorem~3.36]{rosenstein1982linear} guarantees the existence of a countably infinite monotone increasing sequence $(\xx_n)_{n\in\omega}$ in $X$ that is cofinal in $X$. Dually we can argue the existence of a countably infinite monotone decreasing sequence $(\yy_n)_{n\in\omega}$ in $Y$ that is coinitial in $Y$. Then $[\xx_n,\yy_n]\supseteq[\xx_{n+1},\yy_{n+1}]$ and $\bigcap_{n\in\omega}[\xx_n,\yy_n]=\emptyset$.

For each $n\in\omega$, Proposition \ref{unique string with minimal length in an interval} guarantees the existence of the unique minimal length string $\zz_n\in[\xx_n,\yy_n]$ such that $\zz_n\sqsubseteq_l\zz_{n+1}$. Since $\bigcap_{n\in\omega}[\xx_n,\yy_n]=\emptyset$, for each $n_k\in\omega$ there is a least $n_{k+1}>n_k$ such that $\zz_{n_k}\notin[\xx_{n_{k+1}},\yy_{n_{k+1}}]$. Hence $\zz_{n_k}\sqsubset_l\zz_{n_{k+1}}$. Thus $|\zz_n|\to\infty$ as $n\to\infty$, which together with $\zz_n\sqsubseteq_l\zz_{n+1}$ for each $n\in\omega$ ensures that $(\zz_n)_{n\in\omega}$ is a convergent sequence with a left $\N$-string, say $\zz$, as its limit. Clearly $\xx_n<_l\zz<_l\yy_n$ for each $n\in\omega$.

Let $\xx\in X$. Since $\bigcap_{n\in\omega}[\xx_n,\yy_n]=\emptyset$, there is some $k\in\omega$ such that $\xx\notin[\xx_k,\yy_k]$. Thus $\xx<_l\xx_k<_l\zz$. Dually we can show that $\zz<_l\yy$ for each $\yy\in Y$.

If $\zz'<_l\zz''$ are two distinct left $\N$-strings satisfying $\xx<_l\zz'<_l\yy$ and $\xx<_l\zz''<_l\yy$ for each $\xx\in X$ and $\yy\in Y$, then $\zz'<_l\zz'\sqcap_l\zz''<_l\zz''$, and hence $\zz'\sqcap_l\zz''\in H_l(\xx_0)\setminus(X\cup Y)$, a contradiction. Thus associated to each gap $(X,Y)$ in $H_l(\xx_0)$ there is a unique left $\N$-string $\zz$ satisfying $\xx<_l\zz<_l\yy$ for each $\xx\in X$ and $\yy\in Y$. If $(X_1,Y_1)$ and $(X_2,Y_2)$ are distinct gaps then it is routine to verify that the left $\N$-strings associated to these gaps are distinct.
\end{proof}

The hammock $H_l(\xx_0)$ can be expressed as $H_l(\xx_0)=H_l^{-1}(\xx_0)\dotplus H_l^1(\xx_0)$, where 
$$H_l^i(\xx_0):=\{\xx\in H_l(\xx_0)\mid\text{ either }\xx=\xx_0\text{ or }\theta(\xx\mid\xx_0)=i\}$$
is a bounded discrete linear suborder of $H_l(\xx_0)$ with minimum element $\mm_i(\xx_0)$ and maximum element $\MM_i(\xx_0)$. It is easily noted that $\MM_{-1}(\xx_0)=\mm_1(\xx_0)=\xx_0$.

Smaller left hammocks can be embedded in bigger hammocks as intervals.
\begin{rem}\label{subhammock is an interval}
If $\yy(\neq\xx_0)\in H_l^i(\xx_0)$ then $H_l(\yy)$ is an interval in $H_l^i(\xx_0)$.
\end{rem}
The concept of $H$-equivalence was introduced in \cite[\S~5]{sardar2021variations} to identify when two left hammocks are isomorphic. Say that two strings $\xx$ and $\yy$ are $H$-equivalent, denoted $\xx\equiv_H\yy$, if for every string $\uu$, $\uu\xx\in\St$ if and only if $\uu\yy\in\St$. Indeed, if $\xx\equiv_H\yy$ then $(H_l(\xx),<_l)\cong (H_l(\yy),<_l)$. As a consequence of \cite[Proposition~5.3]{sardar2021variations}, a criterion for testing $H$-equivalence, we note a useful observation that we will use without mention.

\begin{rem}
If $\xx,\yy$ are strings with $\delta(\yy)=0$ such that $\yy\xx$ is a string then $\yy\equiv_H\yy\xx$. As a consequence, if $\zz$ is a string such that $\zz\yy$ is a string then $\zz\yy\xx$ is a string.
\end{rem}

Some finiteness results are the key to the proof of the main theorem which states that $(H_l^i(\xx_0),<_l )\in\dLOfdb{1}{1}$. Recall the definition of a prime band from \cite[Definition~3.1.1]{GKS20}. It was shown in \cite[Theorem~3.1.6]{GKS20} that there are only finitely many prime bands in $\mathsf{Ba}(\Lambda)$.

For $\sB\in\QBa$, call a string $\yy$ to be \emph{band-free with respect to $\sB$} if there does not exist $\bb\in\CycB$ and strings $\yy_1,\yy_2$ such that $\yy=\yy_2\bb\yy_1$. Recall from \cite[Proposition~3.1.7]{GKS20} that there are only finitely many band-free strings in any string algebra. Call a string $\yy=\zz\xx_0\in H_l^i(\xx_0)$ \emph{band-free relative to} $(\xx_0,i)$ if $\zz$ is band-free. The following is an immediate consequence of \cite[Proposition~3.1.7]{GKS20}. 

\begin{cor}\label{finitely many band free strings rel to x0,i}
There are only finitely many band-free strings relative to $(\xx_0,i)$.
\end{cor}

We end this section by mentioning a basic result about bands, which shows that a band $\bb$ has exactly $|\bb|$ distinct cyclic permutations. If $\xx=\alpha_n\cdots\alpha_2\alpha_1$ is a finite power of cyclic permutation of a band, call $1\leq k\leq n$ a \emph{period} of $\xx$ if $k=n$ or $\xx=\alpha_k\cdots\alpha_1\alpha_n\cdots\alpha_{k+1}$.
\begin{prop}\label{KMP}
\cite[Lemma~1]{knuth1977fast}
Let $\xx$ be a finite power of a cyclic permutation of a band and $p$ and $q$ be periods of $\xx$ such that $p+q\leq|\xx|+gcd(p,q)$. Then $gcd(p,q)$ is a period of $\xx$.
\end{prop}
\begin{cor}\label{cyclic perm are diff; cor}
If $\bb\in\Cyc$ is such that $\bb=\alpha_n...\alpha_1$ then $\alpha_r\cdots\alpha_1\alpha_n\cdots\alpha_{r+1}\neq\alpha_n\cdots\alpha_1$ for any $1\leq r<n$.
\end{cor}
\begin{proof}
If not, then $r$ is a period of $\bb$. Since $|\bb|$ is also a period of $\bb$, we have that $gcd(|\bb|,r)=:t$ is a period of $\bb$ by Proposition \ref{KMP}. Therefore, we get that $\alpha_n\cdots\alpha_1=(\alpha_t\cdots\alpha_1)^{|\bb|/t}$, a contradiction to the fact that $\bb$ is a primitive cyclic string.
\end{proof}

\section{A finite poset}\label{sec: a finite poset}
The property which distinguishes a non-domestic string algebra from a domestic string algebra is the existence of a meta-band in its bridge quiver \cite[Proposition~3.4.2]{GKS20}. A generalised meta band defined below captures the complete essence of a building block of a string algebra.

Recall from \cite[Definition~3.2.1]{GKS20} that the vertices of the (weak) bridge quiver are prime bands and the arrows are (weak) bridges.

\begin{defn}
A \emph{generalised meta band} (\emph{GMB}, for short), denoted $\mathsf{B}$, is a strongly connected component of the bridge quiver.
\end{defn}

Call a GMB $\sB$ \emph{domestic} if $\sB$ has only one vertex and \emph{non-domestic} otherwise. Note that a string algebra $\Lambda$ is non-domestic if and only if there is a non-domestic GMB in its bridge quiver (cf. \cite[Proposition~3.4.2]{GKS20}).

Recall the definition of generation of strings from paths in the bridge quiver from \cite[\S~3.3]{GKS20}. Call a string \emph{$\sB$-cycle} if it lies in $\Cyc$ and is generated by a path in $\sB$. Denote the set of all $\sB$-cycles by $\CycB$. Call a string \emph{$\sB$-extendable} if it is a substring of a power of a $\sB$-cycle. Denote the set of all $\sB$-extendable strings by $\ExtB$. Finally, set $\BaB:=\Ba\cap\CycB$.

It is useful to note the following trivial property of $\sB$-extendable strings, which we will always use without mention.
\begin{rem}
Given $\xx_1,\xx_2\in\ExtB$, there exists a string $\uu$ such that $\xx_2\uu\xx_1\in\ExtB$.
\end{rem}

Define a relation $\preceq$ on the set $\Ba$ of bands by declaring $\bb_1\preceq\bb_2$ if there is a string $\uu$ such that $\bb_2\uu\bb_1$ is a string. This relation is clearly reflexive and transitive.

\begin{prop}
The relation $\preceq$
defined above is anti-symmetric if and only if the string algebra $\Lambda$ is domestic. 
\end{prop}
\begin{proof}
Suppose the relation $\preceq$ is anti-symmetric. If possible, let the string algebra be non-domestic. By \cite[Proposition~3.4.2]{GKS20}, $\Lambda$ contains a meta-band. There are two cases.

If the length of the meta-band exceeds $1$ then consider two distinct prime bands $\bb_1$ and $\bb_2$ in that meta-band. The definition of a meta-band ensures the existence of strings $\uu$ and $\vv$ such that $\bb_2\uu\bb_1$ and $\bb_1\vv\bb_2$ are strings. This violates that $\preceq$ is anti-symmetric.

On the other hand, if the meta-band is a non-trivial bridge $\bb\xrightarrow{\uu}\bb$ then, by \cite[Proposition~3.4.1]{GKS20}, we have that $\bb$ is a vertex of a meta-band containing at least two prime bands. The rest of the argument is similar to that in the previous paragraph.

Conversely, suppose $\preceq$ is not anti-symmetric. Then there are distinct bands $\bb_1$ and $\bb_2$ and strings $\uu$ and $\vv$ such that $\bb_2\uu\bb_1$ and $\bb_1\vv\bb_2$ are strings. Now the strings $\vv\bb_2\uu\bb_1^k$ for $k\geq1$ contain an infinite family of cyclic permutations of distinct bands, proving that $\Lambda$ is non-domestic.
\end{proof}

Say that $\bb_1\approx\bb_2$ if $\bb_1\preceq\bb_2$ and $\bb_2\preceq\bb_1$, and set $\Q^{\mathrm{Ba}}:=\Ba/\approx$. Note that $\bb_1\approx\bb_2$ if and only if there is a GMB $\sB$ such that $\bb_1,\bb_2\in\BaB$. Hence we will denote the elements of $\QBa$ using $\sB$, possibly with decoration. Borrowing the adjectives for a GMB, if $\sB\in\QBa$ and $|\sB|$=1, then say that $\sB$ is \emph{domestic}, otherwise say that it is \emph{non-domestic}.

By appropriate manipulation of cyclic permutations, it is trivial to note the following. Let $\bb_1,\bb_2,\bb'_1,\bb'_2\in\Cyc$ such that $\bb'_1$ and $\bb'_2$ are cyclic permutations of $\bb_1$ and $\bb_2$ respectively. If $\bb_2\uu\bb_1$ is a string for some string $\uu$ then $\bb'_2\vv\bb'_1$ is a string for some string $\vv$. Therefore we can extend the relation $\preceq$ on the set $\Cyc$ such that for any $\bb_1,\bb_2\in\Ba$, we have $\bb_1\preceq\bb_2$ if and only if $\bb'_1\preceq\bb'_2$ where $\bb'_1$ and $\bb'_2$ are cyclic permutations of $\bb_1$ and $\bb_2$ respectively.

\begin{prop}\label{composite approx}
If $\bb\in\Ba$ is composite then there is $\bb_1\in\Ba$ with $|\bb_1|<|\bb|$ such that $\bb\approx\bb_1$.
\end{prop}
\begin{proof}
If $\bb\in\Ba$ is composite then there is a cyclic permutation $\bb'$ of $\bb$ which can be written as $\bb'=\bb'_k\cdots\bb'_1$ for some $k>1$ and cyclic permutations $\bb'_j$ of $\bb_j\in\Ba$. It is then straightforward to note that $|\bb'_1|<|\bb'|=|\bb|$ and $\bb_1\approx\bb'_1\approx\bb'\approx\bb$.
\end{proof}
This simple result has an immediate consequence.
\begin{cor}
If $\sB\in\QBa$ then $\sB$ contains a prime band.
\end{cor}
\begin{proof}
Let $\bb\in\BaB$. If $\bb$ is prime then we are done. Otherwise, by Proposition \ref{composite approx}, there exists $\bb_1\in\Ba$ such that $|\bb_1|<|\bb|$ and $\bb_1\approx\bb$. If $\bb_1$ is a prime band then we are done. Otherwise, we repeat the process on $\bb_1$ to get $\bb_2\in\Ba$ such that $|\bb_2|<|\bb_1|$ and $\bb_2\approx\bb_1$. Thus we get a sequence of bands $\bb_1,\bb_2,\cdots$ such that $\bb\approx\bb_1\approx\bb_2\approx\cdots$ and $|\bb|>|\bb_1|>|\bb_2|>\cdots$. Since $|\bb|$ is finite, this process has to terminate after finitely many steps, thereby giving us a prime band in $\sB$.
\end{proof}

Since \cite[Theorem~3.1.6]{GKS20} gives that there are only finitely many prime bands, the above corollary yields the following finiteness result.
\begin{prop}\label{QBA finite poset}
The poset $(\Q^{\mathrm{Ba}},\preceq)$ is a finite poset.
\end{prop}

For a fixed string $\xx_0$ and parity $i\in\{1,-1\}$, say that a band $\bb$ is \emph{reachable} from $(\xx_0,i)$ if there is a string $\uu$ such that $\bb\uu\xx_0\in H_l^i(\xx_0)$. If $\bb_1\approx\bb_2$ then $\bb_1$ is reachable from $(\xx_0,i)$ if and only if $\bb_2$ is reachable from $(\xx_0,i)$. Hence the subset $\Q_i^{\mathrm{Ba}}(\xx_0)$ of $\Q^{\mathrm{Ba}}$ of elements reachable from $(\xx_0,i)$ is also finite. Say that $\sB\in\Q^{\mathrm{Ba}}$ is \emph{minimal for $(\xx_0,i)$} if it is a minimal element of $(\Q_i^{\mathrm{Ba}}(\xx_0),\preceq)$. Since every finite poset contains a minimal element, the existence of a minimal $\sB$ for the pair $(\xx_0,i)$ is guaranteed.

\begin{exmp}\label{exmp: QBa in running example}
Consider the string algebra $\Gamma_0$ in Figure \ref{fig:running example}. There are four elements in $\QBa_1(a_0)$, namely $\sB_1=\{b_1B_4b_3B_2\}$, $\sB_2$ containing bands $d_1D_2$ and $d_3D_4$, $\sB_3$ containing bands $e_3E_2E_1$, $g_4G_3g_2G_1$ and $k_1K_2$, and $\sB_4=\{m_1M_2\}$. Here $\sB_1$ and $\sB_4$ are domestic; whereas $\sB_2$ and $\sB_3$ are non-domestic. We have $\sB_1\prec\sB_2$ and $\sB_3\prec\sB_4$ as the only order relations in $(\QBa_1(a_0),\prec)$. Only $\sB_1$ and $\sB_3$ are minimal for $(a_0,1)$.
\end{exmp}

\begin{figure}[h]
    \centering
\[\begin{tikzcd}[column sep=small]
                     &     &                                          &                          &                                           & v_{20} \arrow[ld, "p"']                                &                                         & v_{19} \arrow[ll, "k_1", bend left] \arrow[ll, "k_2"', bend right] &                                                               &                                                               &                                                                  &        \\
                     &     &                                          &                          & v_{12} \arrow[ld, "e_1"']                 &                                                        &                                         &                                                                    & v_{18} \arrow[lu, "h_2"] \arrow[rd, "h_1"]                    &                                                               &                                                                  &        \\
                     &     &                                          & v_{11} \arrow[ld, "a_2"] &                                           & v_{13} \arrow[ll, "e_3"] \arrow[lu, "e_2"'] &                                         & v_{14} \arrow[ll, "f"]                                             &                                                               & v_{15} \arrow[ll, "g_1"'] \arrow[d, "g_2"]        &                                                                  &        \\
v_0 \arrow[r, "a_0"] & v_1 & v_2 \arrow[rd, "a_3"'] \arrow[l, "a_1"'] &                          &                                           &                                                        &                                         &                                                                    & v_{17} \arrow[r, "g_3"'] \arrow[lu, "g_4"]                    & v_{16} \arrow[r, "q"']                                        & v_{21} \arrow[r, "m_1", bend left] \arrow[r, "m_2"', bend right] & v_{22} \\
                     &     &                                          & v_3 \arrow[rd, "a_4"']   &                                           & v_5 \arrow[ll, "a_5"']                                 & v_6 \arrow[l, "b_2"'] \arrow[rd, "b_3"] &                                                                    &                                                               &                                                               &                                                                  &        \\
                     &     &                                          &                          & v_4 \arrow[rrr, "b_4"'] \arrow[ru, "b_1"] &                                                        &                                         & v_7 \arrow[r, "c"]                                                 & v_8 \arrow[r, "d_1", bend left] \arrow[r, "d_2"', bend right] & v_9 \arrow[r, "d_3", bend left] \arrow[r, "d_4"', bend right] & v_{10}                                                           &       
\end{tikzcd}\]
    \caption{\centering $\Gamma_0$ with $\rho=\{a_3a_2,a_4a_5,b_4a_4,a_5b_1,cb_4,d_2c,d_4d_1,d_3d_2,a_2e_3,e_3f,fg_4,g_2h_1,k_2h_2,qg_3,m_2q,pk_2,e_1p\}$}
    \label{fig:running example}
\end{figure}
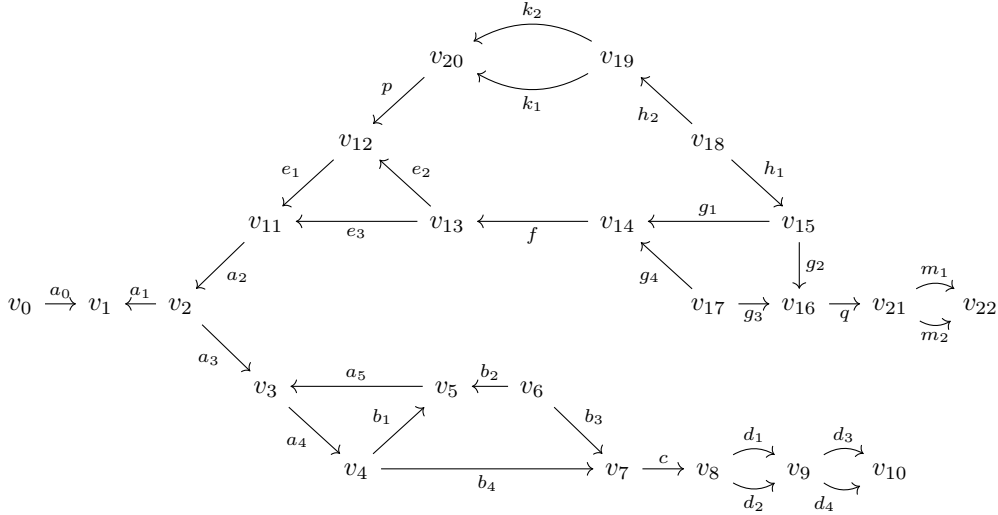
\section{Some finiteness results}\label{sec:finiteness results}
We introduce some sets of strings in a hammock that are close to an element of $\QBa_i(\xx_0)$ and show under a suitable condition that some subsets of those sets are finite.

Given $\sB\in \QBa$ and $j\in\{-1,1\}$, the following set captures the strings which ``touch'' $\sB$ with sign $j$.
$$\StB{j}:=\{\xx\in\St\mid\bb\xx\text{ is a string for some } \bb\in\CycB \text{ such that } \theta(\bb)=j\}.$$

Further set $$\StB{}:=\StB{1}\cup\StB{-1},\ \StB{\pm1}:=\StB{1}\cap\StB{-1}.$$ If $\sB\in\QBa_i(\xx_0)$ then set $$\STB{j}:=\StB{j}\cap H_l^i(\xx_0),\ \STB{}:=\StB{}\cap H_l^i(\xx_0)\text{ and } \STB{\pm1}:=\StB{\pm1}\cap H_l^i(\xx_0).$$
Now we close the above sets of strings under left substrings. For $j\in\{1,-1\}$, define $$\OSt{j}:=\{\xx\in\St\mid\uu\xx\in\StB{}\text{ for some string }\uu\text{ with }\theta(\uu)=j\},$$  $$\OSt{}:=\OSt{1}\cup\OSt{-1} \text{ and } \OSt{\pm1}:=\OSt{1}\cap\OSt{-1}.$$
Further if $\sB\in\QBa_i(\xx_0)$ then set $$\OST1:=\{\xx\in H_l^i(\xx_0)\mid\uu\xx\in\StB{}\cup\{\MM_i(\xx_0)\}\text{ for some string }\uu\text{ with }\theta(\uu)=1\}\cup\{\MM_i(\xx_0)\},$$$$\OST{-1}:=\{\xx\in H_l^i(\xx_0)\mid\uu\xx\in\StB{}\cup\{\mm_i(\xx_0)\}\text{ for some string }\uu\text{ with }\theta(\uu)=-1\}\cup\{\mm_i(\xx_0)\},$$$$\ \OST{}:=\OST1\cup\OST{-1}\text{ and }\OST{\pm1}:=\OST1\cap\OST{-1}.$$

When we use the notations $\STB{j}$ and $\OST{j}$ then we implicitly assume that $\sB\in\QBa_i(\xx_0)$.

It is trivial to note that $\StB{j}\subseteq\OSt{j}$ for any $j\in\{-1,1\}$. The following remarks are straightforward yet useful.

\begin{rem}\label{rem: Strings in hammock WStr}
For any $\xx\in H_l^i(\xx_0)$, $\xx\in\OST{1}$ if and only if $\xx\in\OSt{1}$ or $\xx\sqsubseteq_l\MM_i(\xx_0)$. Dually, for any $\xx\in H_l^i(\xx_0)$, $\xx\in\OST{-1}$ if and only if $\xx\in\OSt{-1}$ or $\xx\sqsubseteq_l\mm_i(\xx_0)$. 
\end{rem}
\begin{rem}\label{rem: WStr is closed by substring}
If $\xx\in\OSt{}$ and $\yy\sqsubset_l\xx$ then $\yy\in\OSt{\theta(\xx\mid\yy)}$.
\end{rem}
\begin{rem}\label{rem: Hammock WStr is closed by substring}
If $\xx\in\OST{}$, $\yy\sqsubset_l\xx$ and $\yy\in H_l^i(\xx_0)$ then $\yy\in\OST{\theta(\xx\mid\yy)}$.
\end{rem}

\begin{prop}\label{x<y implies x in ST1B}
If $\xx\in\OST{1}$, $\yy\in\OST{-1}$ and $\xx<_l\yy$ then $\xx\in\OSt1$ and $\yy\in\OSt{-1}$.
\end{prop}
\begin{proof}
If possible, let $\xx\notin\OSt{1}$. Then Remark \ref{rem: Strings in hammock WStr} implies $\xx\sqsubseteq_l\MM_i(\xx_0)$.

If $\xx_0=\MM_i(\xx_0)$ then $i=-1$ and $\xx=\xx_0=\MM_i(\xx_0)$, which contradicts the existence of $\yy>_l\xx$.

Therefore assume that $\xx_0\sqsubset_l\MM_i(\xx_0)$. This together with $\xx<_l\yy$ implies that $\yy=\ww\vv\xx$ for some strings $\ww,\vv$ with $\delta(\vv)=1$. Since $\yy\in\OST{-1}$ and $\xx\sqsubset_l\yy$, Remark \ref{rem: Hammock WStr is closed by substring} implies that $\xx\in\OST{\theta(\yy\mid\xx)}=\OST1$. Since $\xx_0=\mm_i(\xx_0)\sqsubset_l\yy$, we have $\yy\in\OSt{}$ by Remark \ref{rem: Strings in hammock WStr}. Now $\xx\sqsubset_l\yy$ together with Remark \ref{rem: WStr is closed by substring} implies that $\xx\in\OSt{\theta(\yy\mid\xx)}=\OSt1$, a contradiction to our assumption. Therefore $\xx\in\OSt{1}$. Similarly, we can show that $\yy\in\OSt{-1}$.
\end{proof}

In general, $\OST{}$ could be very large compared to $\STB{}$, but it is possible to control this difference when $\sB$ is minimal for $(\xx_0,i)$.

\begin{prop}\label{bandfreewrtB}
For $\sB\in\QBa$ and $j\in\{-1,1\}$, if $\xx\in\OSt{j}\setminus\StB{j}$ then $\xx$ is band-free with respect to $\sB$.
\end{prop}
\begin{proof}
If possible, let $\xx=\xx_2\bb\xx_1$ for some strings $\xx_1,\xx_2$ and $\bb\in\CycB$. Since $\xx\in\OSt{j}$, there exist a string $\uu$ and $\bb_1\in\CycB$ with $\theta(\bb_1\uu)=j$ such that $\bb_1\uu\xx_2\bb\xx_1$ is a string. Now $\bb,\bb_1\in\CycB$ implies that there is a string $\vv$ such that $\bb\vv\bb_1$ is a string. Since $\delta(\bb)=\delta(\bb_1)=0$, we have that $\bb\vv\bb_1\uu\xx_2\bb$ is a string implying that $\xx_2\bb\vv\bb_1\uu$ is a power of a $\sB$-cycle. This gives $\xx=\xx_2\bb\xx_1\in\StB{j}$, which is a contradiction.
\end{proof}
\begin{cor}\label{BandfreeOST-OSt}
If $\sB\in\QBa$ is minimal for $(\xx_0,i)$, $j\in\{-1,1\}$ and $\xx\in \OST{j}\setminus\STB{j}$ then $\xx$ is band-free relative to $(\xx_0,i)$.
\end{cor}
\begin{proof}
Let $\xx\in \OST{j}\setminus\STB{j}$. In view of Remark \ref{rem: Strings in hammock WStr}, there are three cases. Since $\delta(\MM_i(\xx_0)\mid\xx_0)\neq0$, if $\xx\sqsubseteq_l\MM_i(\xx_0)$ then $\xx$ is band-free relative to $(\xx_0,i)$. A dual argument works when $\xx\sqsubseteq_l\mm_i(\xx_0)$.

Finally, if $\xx\in\OSt{j}$ then there exist a string $\uu$ and $\bb_1\in\CycB$ with $\theta(\bb_1\uu)=j$ such that $\bb_1\uu\xx$ is a string. If possible, let $\xx=\xx_2\bb\xx_1$ for some strings $\xx_1,\xx_2$ and $\bb\in\Cyc$ such that $\xx_0\sqsubseteq_l\xx_1$. Then $\bb_1\uu\xx_2\bb\xx_1$ is a string, which gives $\bb\preceq\bb_1$. Since $\sB$ is minimal for $(\xx_0,i)$, we get $\bb\approx\bb_1$, a contradiction to Proposition \ref{bandfreewrtB}.
\end{proof}

Recall from Corollary \ref{finitely many band free strings rel to x0,i} that there are finitely many strings in $H_l^i(\xx_0)$ which are band-free relative to $(\xx_0,i)$. A simple set theoretic manipulation yields $$\OST{}\setminus\STB{}\subseteq(\OST{1}\setminus\STB{1})\cup(\OST{-1}\setminus\STB{-1}).$$ Therefore we get the following consequence of Corollary \ref{BandfreeOST-OSt}.

\begin{cor}\label{OST-STB is finite}If $\sB\in\QBa$ is minimal for $(\xx_0,i)$ then the set $\OST{}\setminus\STB{}$ is finite.
\end{cor}

\begin{exmp}
Continuing with Example \ref{exmp: QBa in running example}, recall that $\sB_1$ and $\sB_3$ are minimal for $(a_0,1)$. We have $\OsT{a_0}{1}{\sB_1}\setminus\stb{a_0}{1}{\sB_1}=\{a_0,A_1a_0,a_3A_1a_0\}$ and $\OsT{a_0}{1}{\sB_3}\setminus\stb{a_0}{1}{\sB_3}=\{a_0,A_1a_0\}$.
\end{exmp}

We prove yet one more conditional finiteness result.
\begin{prop}
Suppose $\sB\in\QBa$ is domestic and minimal for $(\xx_0,i)$. If $\xx\in\OST{\pm1}$ then $\xx$ is band-free relative to $(\xx_0,i)$.
\end{prop}
\begin{proof}
Since $\MM_i(\xx_0)$ and $\mm_i(\xx_0)$ are band-free relative to $(\xx_0,i)$, in view of Remark \ref{rem: Strings in hammock WStr}, it is enough to assume that $\xx\in\StB{\pm1}\cap H_l^i(\xx_0)$.

If possible, let $\xx=\xx_2\bb\xx_1$ for $\bb\in\Cyc$ and strings $\xx_1,\xx_2$ such that $\xx_0\sqsubseteq_l\xx_1$. Since $\xx\in\OSt{\pm1}$, there exist strings $\uu$ and $\vv$ with $\theta(\bb_1\uu)=-\theta(\bb_1\vv)$ such that $\bb_1\uu\xx_2\bb\xx_1$ and $\bb_1\vv\xx_2\bb\xx_1$ are strings, where $\bb_1$ is the unique element in $\sB$ since $\sB$ is domestic. This gives $\bb\preceq\bb_1$, which further implies that $\bb\approx\bb_1$ since $\sB$ is minimal for $(\xx_0,i)$. Moreover, $\sB$ is domestic implies that $\bb$ is a cyclic permutation of $\bb_1$. Then $\,^\infty\bb_1\uu\xx_2\bb=\,^\infty\bb=\,^\infty\bb_2\vv\xx_2\bb$, which implies $\theta(\bb_1\uu)=\theta(\bb_1\vv)$, a contradiction.
\end{proof}

Again by the finiteness of the set of band-free strings, we have the following corollary.
\begin{cor}\label{STB+-1 is finite}
If $\sB\in\QBa$ is domestic and minimal for $(\xx_0,i)$ then the set $\OST{\pm1}$ is finite.
\end{cor}
\begin{exmp}
Continuing with Example \ref{exmp: QBa in running example}, recall that $\sB_1$ is domestic and minimal for $(a_0,1)$. The string $a_3A_1a_0$ is the only element in $\overline{\mathsf{St}}_{\pm1}(a_0,1;\sB_1)$. On the other hand, $\sB_3$ is non-domestic and minimal for $(a_0,1)$, and we have $E_2E_1(e_3E_2E_1)^nA_2A_1a_0\in\overline{\mathsf{St}}_{\pm1}(a_0,1;\sB_3)$ for every $n\in\N$.
\end{exmp}

\begin{prop}\label{OSTpm is bounded}
Let $\sB\in\QBa$ be minimal for $(\xx_0,i)$. Then $\OST{\pm1}$ is bounded as a suborder of $(H_l^i(\xx_0),<_l)$.
\end{prop}
\begin{proof}
First note that $\xx_0\in\{\MM_i(\xx_0),\mm_i(\xx_0)\}$. Without loss assume that $\mm_i(\xx_0)<_l\MM_i(\xx_0)$. Then $\xx_0=\MM_i(\xx_0)$ if and only if $i=-1$.

Without loss assume that $i=-1$. Clearly $\xx_0=\MM_i(\xx_0)\in\OST{1}$. On the other hand, since $\sB$ is minimal for $(\xx_0,i)$, there is $\bb\in\BaB$ and a string $\uu$ such that $\bb\uu\xx_0\in H_l^i(\xx_0)$, which gives $\xx_0\in\OST{-1}$. Hence $\xx_0\in\OST{\pm1}$. On the other hand, we have $\theta(\bb\uu\xx_0\mid\mm_i(\xx_0))=1$. Hence there is a left substring of $\mm_i(\xx_0)$ that lies in $\OST{\pm1}$. Since $\delta(\mm_i(\xx_0)\mid\xx_0)=-1$, the longest such left substring will be the least element of $\OST{\pm1}$.
\end{proof}

\section{The condensation operator $c_\sB$}
\label{sec: Condensation}
Recall the concept of condensation from \S~\ref{sec:linear order}. In this section, we define a specific condensation operator $c_\sB$ on a hammock which helps in breaking it into smaller hammocks.

Note that $\xx_0\in \{\mm_i(\xx_0),\MM_i(\xx_0)\}\subseteq\OST{}$ and $\xx_0$ appears as a left substring of every string in $H_l^i(\xx_0)$. Therefore every string in $H_l^i(\xx_0)$ has a left substring in $\OST{}$. Now we use this observation to define the localization/condensation of a string in a hammock with respect to $\sB$.

\begin{defn}
If $\sB\in\QBa_i(\xx_0)$ then define the \emph{$\sB$-condensation} map $$c_\sB:H_l^i(\xx_0)\to\OST{}$$ by associating to each $\xx\in H_l^i(\xx_0)$ its longest left substring in $\OST{}$.
\end{defn}
\begin{rem}\label{cbI=IcapOST}
Note that if $\xx,\yy\in\OST{}$ then $\{\xx,\yy\}\subseteq c_\sB([\xx,\yy])=[\xx,\yy]\cap\OST{}.$ As a consequence, the map $c_\sB$ is surjective. Also for any $\xx\in H_l^i(\xx_0)$, we have $c_\sB(\xx)=\xx$ if and only if $\xx\in\OST{}$.
\end{rem}

Now define a function $\varphi_\sB:H_l^i(\xx_0)\to\{-1,0,1\}$ by $$\varphi_\sB(\xx):=\begin{cases}0&\mbox{ if }c_\sB(\xx)\in\OST{\pm1},\\1 &\mbox{ if }c_\sB(\xx)\in\OST1\setminus\OST{-1},\\-1 &\mbox{ if }c_\sB(\xx)\in\OST{-1}\setminus\OST1.\end{cases}$$

\begin{rem}\label{varphinonzero}
For each $\xx\in H_l^i(\xx_0)\setminus\OST{}$, we have $c_\sB(\xx)\sqsubset_l\xx$. If $\varphi_\sB(\xx)=0$ and $\alpha c_\sB(\xx)\sqsubseteq_l\xx$ for some syllable $\alpha$ then $\alpha c_\sB(\xx)\in\OST{}$ since $c_\sB(\xx)\in\OST{\theta(\alpha)}$. This is a contradiction to the definition of $c_\sB(\xx)$, and hence $\varphi_\sB(\xx)\neq 0$.
\end{rem}

Take the convention $H_l^0(\yy):=\{\yy\}$.

If $\yy\in\OST{}$ then $\yy\in H_l^{-\varphi_\sB(\yy)}(\yy)$. On the other hand, if $\xx\in H_l^i(\xx_0)\setminus\OST{}$ and $j:=\theta(\xx\mid c_\sB(\xx))$ then the definition of $c_\sB(\xx)$ ensures that $c_\sB(\xx)\notin\OST j$. Thus $\varphi_\sB(\xx)=\varphi_\sB(c_\sB(\xx))=-j$. We document this observation in the following result.

\begin{prop}\label{hammockcover}
If $\xx\in H_l^i(\xx_0)$ then $\xx\in H_l^{-\varphi_\sB(\xx)}(c_\sB(\xx))$.
\end{prop}

The function $\varphi_\sB$ is defined in such a way that the following statement is true. This will be the key to showing that the algorithm to compute the order type of a hammock terminates after finitely many steps.
\begin{rem}\label{hammockcoverrepr}
For each $\yy\in\OST{}$, we have $H_l^{-\varphi_\sB(\yy)}(\yy)\cap \OST{}=\{\yy\}$.
\end{rem}

\begin{prop}\label{hammockcoverclassesorder}
If $\xx,\yy\in\OST{}$ and $\xx<_l\yy$ then for each $\xx'\in H_l^{-\varphi_\sB(\xx)}(\xx)$ and $\yy'\in H_l^{-\varphi_\sB(\yy)}(\yy)$ we have $\xx'<_l\yy'$.
\end{prop}

\begin{proof}

If $\xx\sqsubset_l\yy$ then $\xx\sqsubset_l\yy\sqsubseteq_l\yy'$ for each $\yy'\in H_l^{-\varphi_\sB(\yy)}(\yy)$. Hence $\theta(\yy'\mid\xx)=\theta(\yy\mid\xx)=1$. Moreover, since $\xx\sqsubset_l\yy\in\OST{}$, we conclude that $\varphi_\sB(\xx)\neq-1$. If $\varphi_\sB(\xx)=0$ then the conclusion holds. On the other hand, if $\varphi_\sB(\xx)=1$ then $\theta(\xx'\mid\xx)=-1$ for each $\xx'\in H_l^{-\varphi_\sB(\xx)}(\xx)\setminus\{\xx\}$. Hence $\theta(\yy'\mid\xx')=\theta(\yy'\mid\xx)=1$, and hence the conclusion.

A dual argument can be given when $\xx\sqsupset_l\yy$.

Finally when $\xx$ and $\yy$ are incomparable then $\xx\sqcap_l\yy\in\OSt{\pm1}$. The arguments in the above two paragraphs then give that $\xx'<_l\xx\sqcap_l\yy<_l\yy'$  for each $\xx'\in H_l^{-\varphi_\sB(\xx)}(\xx)$ and $\yy'\in H_l^{-\varphi_\sB(\yy)}(\yy)$, and thus the conclusion follows.
\end{proof}
As a consequence, we get that certain hammocks are disjoint.
\begin{cor}\label{certain hammocks are disjoint}
If $\xx,\yy\in\OST{}$ and $\xx<_l\yy$ then $H_l^{-\varphi_\sB(\xx)}(\xx)\cap H_l^{-\varphi_\sB(\yy)}(\yy)=\emptyset$.
\end{cor}

The following is the main result of this section which serves as an ingredient for the main theorem of this paper (Theorem \ref{main}). Loosely speaking, this result states that any hammock can be broken down into smaller hammocks when we localize/condense the hammock away from $\sB\in\QBa_i(\xx_0)$. This result gives a recursive algorithm to compute the order type of a hammock.

\begin{lem}\label{hammockordersum}
Suppose $\sB\in\QBa_i(\xx_0)$. Then $$(H_l^i(\xx_0),<_l)\cong\sum_{\xx\in c_\sB(H_l^i(\xx_0))}(H_l^{-\varphi_\sB(\xx)}(\xx),<_l).$$
\end{lem}

\begin{proof}
Recall from Remark \ref{cbI=IcapOST} that $c_\sB(H_l^i(\xx_0))=\OST{}$. For any $\xx\in H_l^i(\xx_0)$ and $j\in\{1,0,-1\}$,  Remark \ref{subhammock is an interval} gives that $H_l^j(\xx)$ is an interval in $H_l^i(\xx_0)$. Hence $$ H_l^i(\xx_0)\supseteq\bigcup_{\xx\in c_\sB(H_l^i(\xx_0))}H_l^{-\varphi_\sB(\xx)}(\xx).$$ The inclusion in the other direction is provided by Proposition \ref{hammockcover} while Corollary \ref{certain hammocks are disjoint} ensures that the union on the right-hand side is disjoint. Finally, Proposition \ref{hammockcoverclassesorder} ensures that the above bijection is indeed an order isomorphism.
\end{proof}

\section{Neighbours of strings in $\sB$-condensation}\label{sec: Neighbours}
This section is devoted to defining operators $\lB$ and $\lbB$ on $\StB{}$, which when restricted to $\STB{}$ help us to find the immediate neighbours of strings in it. En route, we define two subsets $\BalB$ and $\BalbB$ of the set of prime bands in $\sB$ and see that the limit of the sequence of such iterated immediate successors (resp. predecessors) are almost periodic strings of the form $\,^\infty\bb\uu\xx_0$, where $\bb\in\BalB$ (resp. $\BalbB$).
 
The concept of an \emph{exit syllable} of a band and the \emph{exit} of a bridge was introduced in \cite[\S~3]{GKS20}. Slightly modifying the former, we introduce an \emph{exit} of a band below.
\begin{defn}
Given a band $\bb$, say that a pair $(\beta,\bb')$ is an \emph{exit} of $\bb$ if $\beta$ is a syllable and $\bb'$ is a cyclic permutation of $\bb$ such that $\beta\bb'$ is a string but $\beta\bb'\not\sqsubseteq_l\bb'^2$.
\end{defn}
It is trivial to note that if $(\beta,\bb')$ is an exit of a band $\bb$ then $\beta$ is an exit syllable of $\bb$. There are some exits of a $\sB$-band for a non-domestic $\sB\in\QBa$; the signs of the corresponding exit syllables are important in the computation of the order type of hammocks.
\begin{defn}
If $\sB\in\QBa$ and $\bb\in\BaB$, say that an exit $(\beta,\bb')$ of $\bb$ is a \emph{non-domestic exit} if $\beta\bb'\in\ExtB$.
\end{defn}

\begin{rem}\label{nondomexitexist}
For non-domestic $\sB\in\QBa$ and $\bb\in\BaB$ there is $\bb'\in\BaB$ such that $\bb\neq\bb'$. Let $\uu$ be a string such that $\bb'\uu\bb\in\ExtB$. Then $\,^\infty\bb'\uu\bb$ and $\,^\infty\bb$ fork. Hence $\bb$ has a non-domestic exit.
\end{rem}

\begin{defn}\label{defn: la, lb bands}
Denote by $\BalB$ the set of all $\sB$-bands having no non-domestic exit $(\beta,\bb')$ with $\beta\in Q_1$. Dually, denote by $\BalbB$ the set of all $\sB$-bands having no non-domestic exits $(\beta,\bb')$ with $\beta\in Q_1^-$.
\end{defn}

\begin{exmp}
Continuing from Example \ref{exmp: QBa in running example}, we have $\mathsf{Ba}_l(\sB_3)=\{e_3e_2E_1,g_4G_3g_2G_1\}$ and $\mathsf{Ba}_{\lb}(\sB_3)=\{k_1K_2\}$.
\end{exmp}

In view of Remark \ref{nondomexitexist}, it is trivial to note that if $\sB$ is non-domestic then $\BalB\cap\BalbB=\emptyset$. We show in Corollary \ref{BalB non empty} and Corollary \ref{BalB prime} that the sets $\BalB$ and $\BalbB$ are non-empty and finite.

The following proposition is key to defining the operator $\lB$.
\begin{prop}
\label{prop:6.1}
If $\xx\in\StB1$ then there exists $\yy\in \StB1$ such that $\xx\sqsubset_l \yy\sqsubseteq_l l(\xx) $.
\end{prop}
\begin{proof}
If possible, assume that for each $\xx\sqsubset_l\yy\sqsubseteq_ll(\xx)$ we have $\yy \notin \StB1$. Let $\bb\in\CycB$ such that $\bb\xx$ is a string and $\theta(\bb)=1$. Let $\zz:=\bb\xx\sqcap_ll(\xx)$. As $\theta(\bb)=1$, we get $\xx\sqsubset_l\zz\sqsubseteq_ll(\xx)$. Thus by our assumption $\zz\notin\StB1$. This implies that $\zz\neq\bb\xx$. Since $\xx\sqsubset_l\zz\sqsubset_l\bb\xx$, for an appropriate cyclic permutation $\bb'$ of $\bb$, $\bb'\zz$ is a string. Moreover, $\zz\notin\StB1$ implies that $\theta(\bb')=-1$. Therefore $\alpha\zz\sqsubseteq_l\bb\xx$, where $\alpha\in Q_1$ is the first syllable of $\bb'$. Since $\xx\sqsubset_l\zz\sqsubseteq_ll(\xx)$ we get $\alpha\zz\sqsubseteq_ll(\xx)$, which contradicts that $\zz=\bb\xx\sqcap_ll(\xx)$.
\end{proof}

\begin{defn}
Define $\lB:\StB1\to\StB1$ by choosing $\lB(\xx)$ to be the maximal (possibly equal) left substring of $l(\xx)$ such that $\lB(\xx)\in\StB1$.
\end{defn}

\begin{rem}\label{rem:6.4}
For any $\xx\in\StB1$, we have $\lB(\xx)\notin\StB{-1}$.
\end{rem}

For $\xx\in\StB1$ we inductively define the powers of the function $\lB$ by $\lB^0(\xx):=\xx$ and $\lB^{n+1}(\xx):=\lB(\lB^n(\xx))$ for $n\in\N$. Since $\lB^n(\xx)\sqsubset_l\lB^{n+1}(\xx)$ for each $n$ we get that $\lim_{n\to\infty}\lB^n(\xx)$ is a left $\N$-string. Denote this limit by $\brac{1}{\lB}(\xx)$.

The following remark notes that if $|\xx|>0$ then $\lB(\xx)$ depends only on the last syllable of $\xx$. As a consequence, the image of the function $\lB$ restricted to $\STB1$ lies in $\STB1$.
\begin{rem}\label{rem:6.3}
If $\alpha\xx,\alpha\yy\in\StB1$ for some $\alpha \in Q_1\cup Q_1^-$, then  $\lB^n(\alpha\xx)=\uu\alpha\xx$ if and only if $\lB^n(\alpha\yy)=\uu\alpha\yy$. Furthermore, if $\lB^n(\alpha)$ exists then $\lB^n(1_{(t(\alpha),\varepsilon(\alpha))})$ exists and $\lB^n(\alpha)=\lB^n(1_{(t(\alpha),\varepsilon(\alpha))})\alpha$.
\end{rem}
\begin{prop}\label{prop:6.5}
For $\xx\in\StB1$, we have $\brac{1}{\lB}(\xx)=\,^\infty\bb\uu\xx$ for some band $\bb$ and string $\uu$.
\end{prop}
\begin{proof}
Define a function $f:\N^+\to Q_1\cup Q_1^-$ such that $f(k)$ is the last syllable of $\lB^k(\xx)$. As $Q_1\cup Q_1^-$ is finite, there exist $m,n\in\N^+$ such that $f(m)=f(m+n)$. In view of the fact that $\lB^k(\ww)\sqsubset_l\lB^{k+1}(\ww)$ for any $\ww\in\StB1$, let $\yy\xx:=\lB^n(\xx)$ and $\zz\yy\xx:=\lB^{m+n}(\xx)$, where $\zz$ is a string with $|\zz|>0$. As $\lB^n(\xx)$ and $\lB^{m+n}(\xx)$ have the same last syllable, Remark $\ref{rem:6.3}$ together with induction yields that $\lB^{m+kn}(\xx) =\zz^k\yy\xx$ for every $k\in\N$. Since $\zz^k$ is a string for every $k\in\N$, $\zz$ is a finite power of cyclic permutation of a band, say $\bb$. Since $\brac{1}{\lB}(\xx)=\,^\infty\zz\yy\xx$, we get $\brac{1}{\lB}(\xx)=\,^\infty\bb\yy'\xx$ for some string $\yy'$.
\end{proof}

\begin{exmp}\label{exmp: 1lB}
Recall from Example \ref{exmp: QBa in running example} that $e_3E_2E_1$ is a band that lies in $\sB_3$. For $A_2A_1a_0\in\mathsf{St}(\sB_3)$, a routine computation yields $\brac{1}{\ell_{\sB_3}}(A_2A_1a_0)=\,^\infty(e_3E_2E_1)A_2A_1a_0$.
\end{exmp}
The conclusion of Proposition \ref{prop:6.5} is similar to the hypothesis of \cite[Proposition~3.4.5]{GKS20}, whose proof used the concept of $\la$-strings. However, a statement about $\la$-strings \cite[Remark~3.4.4]{GKS20} that was used in the proof is erroneous as demonstrated by Example \ref{ex: 3.4.4 is wrong}. Nevertheless, it does not render \cite[Proposition~3.4.5]{GKS20} false, as it can still be proven using techniques similar to those in the proof of Proposition \ref{prop:6.14}.

\begin{exmp}\label{ex: 3.4.4 is wrong}
Consider the string algebra $\Gamma$ from Figure \ref{3.4.4 is wrong}. For appropriate $j\in \{ 1, -1\}$, we have $$\brac{1}{l}(1_{(v,j)})=\,^\infty(cbaEbafcbD).$$ Here $Ebafcb$, $DcbaEb$ and $fcbDcb$ are $l$-strings with the same first syllable and same length.
\end{exmp}
\begin{figure}[h]
    \begin{tikzcd}
v_1 \arrow[r, "a"] \arrow[rd, "e"'] & v_2 \arrow[r, "d"] \arrow[d, "b"] & v \arrow[ll, "f"', bend right=49] \\
                                    & v_3 \arrow[ru, "c"']              &                                 
\end{tikzcd}
    \caption{$\Gamma$ with $\rho = \{ ce, da, ef, fd, cbaf, fcba \}$}
    \label{3.4.4 is wrong}
\end{figure}
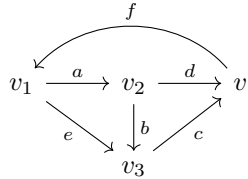
Motivated by the concept of $\la$-strings introduced in \cite[\S~3.4]{GKS20}, now we define $\lB$-strings to prove similar results where $\la$ is replaced with $\lB$.

\begin{defn}
A string $\uu$ is an \emph{$\lB$-string} if $\delta(\uu) = 0$ and $\uu\sqsubseteq\langle 1,\lB\rangle (1_{(v,i)})$ for some $1_{(v,i)}\in\StB1$.
\end{defn}

\begin{rem}\label{rem:6.7}
If $\lB^n(\xx)=\uu\yy\xx$, where $\yy$ and $\uu$ are strings with $|\uu|>0$, then $\theta(\uu)=1$ if and only if there exists $0\leq k<n$ such that $\lB^k(\xx)=\yy\xx$.
\end{rem}

\begin{prop}
Let $\xx\in\StB1$ and $\lB^n(\xx)=\uu\xx$ for some string $\uu$ and $n\in\N$. Then there exist $\bb\in\CycB$, $m\in\N^+$ and $\alpha\in Q_1^-$ such that $\alpha\uu\sqsubseteq_l\bb^m$ and $\bb^m\xx$ is a string.
\end{prop}
\begin{proof}
For each $n\in\N$, let $\uu_n\xx:=\lB^n(\xx)$. Let $\alpha_n\in Q_1^-$ satisfy $\alpha_n\uu_n\xx\sqsubseteq_l\lB^{n+1}(\xx)$. We will prove the result by induction on $n$.

For $n=0$, we have $\lB^n(\xx) = \xx$. Since $\xx\in\StB1$, there is $\bb\in\CycB$ with $\theta(\bb)=1$ such that $\bb\xx$ is a string. Taking $\alpha$ to be the first syllable of $\bb$ proves the statement.

For $n>0$, by induction hypothesis, there exists $\bb\in\CycB$ and $m\in\N^+$ such that $\alpha_{n-1}\uu_{n-1}\sqsubseteq_l \bb^m$ and $\bb^m\xx$ is a string. Let  $\uu_n := \yy\alpha_{n-1}\uu_{n-1}$ such that if $|\yy|>0$ then $\delta(\yy) = -1$. There are two cases.
\begin{itemize}
    \item If $|\yy|>0$, let $\bb_1,\bb_2\in\CycB$ be such that $\alpha_{n-1}\uu_{n-1}\sqsubseteq_r\bb_1$, $\theta(\bb_2)=1$ and $\bb_2\yy\alpha_{n-1}\uu_{n-1}\xx$ is a string. Since $\delta(\yy) = -1$ and $\alpha_{n-1}\in Q_1^-$ we get that $\bb_2\yy\bb_1$ is a string. Since $\bb_1,\bb_2\in\CycB$ we get that $\bb_1\ww\bb_2$ is a string for some string $\ww$. Thus $\ww\bb_2\yy\bb_1\in\CycB$ and $\alpha_n\uu_n\sqsubseteq\ww\bb_2\yy\bb_1$. Let $\bb'$ be a cyclic permutation of $\ww\bb_2\yy\bb_1$ such that $\alpha_n\uu_n=\alpha_n\yy\alpha_{n-1}\uu_{n-1} \sqsubseteq_l\bb'$. Since $\delta(\yy) =-1$ and $\theta(\alpha_n)=1$ we get $\delta(\alpha_n\uu_n)=\delta(\alpha_n\yy\alpha_{n-1}\uu_n)=0$. Therefore $\alpha_n\uu_n \sqsubseteq_l\bb'$ and $\bb'\xx$ is a string.
    
    \item If $|\yy| = 0$ then $\uu_n\sqsubseteq_l\bb$. Let $\ww\uu_n:=\bb^m$ for some string $\ww$. Let $\alpha$ be the first syllable of $\bb\ww$. By Remark \ref{rem:6.4}, we have $\uu_n\xx\notin\StB{-1}$. Therefore $\alpha\in Q_1^-$ and $\bb^{m+1}$ satisfy the conditions of the conclusion.
\end{itemize}
This completes the proof.
\end{proof}

\begin{cor}
\label{cor:6.12}
If $\xx\in\StB1$ and $\uu\yy\xx\sqsubset_l\braclB(\xx)$ for some strings $\uu$ and $\yy$ then there exists $\bb\in\CycB$ such that  $\bb\yy\xx$ is a string and $\uu\sqsubseteq_l\bb^m$ for some $m\in\N^+$.
\end{cor}

\begin{cor}\label{cor:6.10}
If $\uu$ is an $\lB$-string then $\uu\in\ExtB$.
\end{cor}

\begin{prop}
\label{prop:6.11}
If $\xx$ and $\alpha\xx$ are $\lB$-strings for some $\alpha\in Q_1\cup Q_{-1}$ then  $\xx\in \StB{-1}$ if and only if $\alpha\in Q_1$.
\end{prop}
\begin{proof}
Suppose $\xx\in\StB{-1}$. Since $\alpha\xx$ is an $\lB$-string, there exist a string $\yy,\ n\in\N^+,\ v\in Q_0$ and $j\in\{1,-1\}$ such that $\alpha \xx\yy \sqsubseteq_l \lB^{n+1}(1_{(v,j)})$ and $\lB^{n-1}(1_{(v,j)}) \sqsubset_l \xx\yy\sqsubseteq_l\lB^n(1_{(v,j)})$. Let $\uu:=\lB^{n-1}(1_{(v,j)})$. Hence $\xx\yy\sqsubseteq_l\lB(\uu)$. Since $\xx\yy\in\StB{-1}$, there exists $\bb\in\CycB$ with $\theta(\bb)=-1$ such that $\bb\xx\yy$ is a string. Let $\beta$ be the first syllable of $\bb$. Then $\beta\in Q_1$ gives $\beta\xx\yy\sqsubseteq_l \lB(\uu)$. Let $\zz:=\lB(\uu)\sqcap_l\bb\xx\yy$. Since $\uu\sqsubset_l\beta\xx\yy\sqsubseteq_l\zz\sqsubset_l\bb\xx\yy$ and $\uu<_l\lB(\uu)<_l\bb\xx\yy$, we get that $\xx\yy\sqsubset_l\gamma\zz \sqsubseteq_l\bb\xx\yy$ for some $\gamma\in Q_1^-$. This shows that $\zz\in \StB1$. Since $\beta\xx\yy\sqsubseteq_l \zz$ we get $\beta\xx\yy \sqsubseteq_l\lB(\uu)$. Thus $\beta\xx\yy \sqsubseteq_l \lB^{n+1}(1_{(v,j)})$, which gives $\alpha =\beta\in Q_1$.

Conversely if $\alpha\xx$ is an $\lB$-string for $\alpha\in Q_1$, Corollary \ref{cor:6.10} yields $\bb\in\CycB$ such that $\alpha\xx\sqsubseteq_l\bb$. Thus $\xx\in\StB{-1}$.
\end{proof}
\begin{cor}
\label{rem:6.9}
For strings $\xx,\uu,\vv$, if $\xx$, $\uu \xx$ and $\vv \xx$ are $\lB$-strings  then $\uu \xx$ and $\vv\xx$ do not fork.
\end{cor}
\begin{proof}
Suppose, if possible, $\uu\xx$ and $\vv\xx$ fork. Let $\zz\xx:=\uu\xx\sqcap_l\vv\xx$ with $\theta(\vv\xx\mid\zz\xx)=1$ and $\theta(\uu\xx\mid\zz\xx)=-1$. Since $\xx\sqsubseteq\zz\xx\sqsubset\uu\xx$, we have that $\zz\xx$ is an $\lB$-string. By Proposition \ref{prop:6.11}, we have that $\zz\xx\in\StB{\pm1}$, a contradiction to the combination of Remarks \ref{rem:6.4} and \ref{rem:6.7}.
\end{proof}

\begin{rem}
\label{rem:6.10}
If $\xx$ is an $\lB$-string then there exist $\bb\in\BaB$ and a string $\yy$ such that $\bb\yy\uu$ is a string.
\end{rem}

\begin{prop}\label{prop:6.14}
If every cyclic permutation of a band $\bb$ is an $\lB$-string, then $\bb$ is a prime band.
\end{prop}
\begin{proof}
Suppose, if possible, $\bb$ is a composite band. Then there exist $n>1$, $a_1,\cdots,a_n\in\N^+$, $\bb_1,...,\bb_n\in\Cyc$ satisfying $\bb_n\neq\bb_1$ as well as $\bb_j\neq\bb_{j+1}$ for any $j\in\{1,\cdots,n-1\}$, and a cyclic permutation $\bb'$ of $\bb$ such that $\bb'=\bb_n^{a_n}\cdots\bb_2^{a_2}\bb_1^{a_1}$. The hypothesis implies that $\bb'$ is an $\lB$-string. Note that for $v=s(\bb')$ and  appropriate $j\in\{ 1, -1\}$, $\bb_1,\bb_2,\cdots,\bb_n\in(H_l(1_{(v, j)}),<_l)$. By Corollary \ref{cor:6.10}, $\bb_1,\bb_2,\cdots,\bb_n\in\CycB$.

\noindent{}\textbf{Claim.} If $j\neq k$ and $\bb_j\bb_k$ is an $\lB$-string then $\bb_j<_l\bb_k$ in $(H_l(1_{(v, j)}),<_l)$. 

\noindent{}\textit{Proof of the claim.} Let $\yy:=\bb_k\sqcap_l\bb_j$. If $\delta(\yy)=0$ then $\yy$ is an $\lB$-string. In view of Corollary \ref{cyclic perm are diff; cor}, the strings $\bb_{j-1}^{a_{j-1}} \cdots\bb_1^{a_1}\bb_n^{a_n}\cdots\bb_j^{a_j}$ and $\bb_{k-1}^{a_{k-1}}\cdots\bb_1^{a_1}\bb_n^{a_n} \cdots\bb_k^{a_k}$ fork, which contradicts Corollary \ref{rem:6.9}. Hence $\delta(\yy)\neq 0$ and, in particular, as $\delta(\bb_i)=\delta(\bb_j)=0$ but $\delta(\yy)\neq0$, we get that $\bb_i$ and $\bb_j$ fork. 

Assume, if possible, that we have $\bb_k <_l \bb_j$. Let $\xx:=\bb_j\bb_k\sqcap_l\bb_k^2$, so that $\xx$ is an $\lB$-string. Then $\alpha\xx\sqsubseteq_l\bb_j\bb_k$ and $\beta\xx\sqsubseteq_l\bb_k^2$ for some $\alpha\in Q_1^-$ and $\beta\in Q_1$. As $\bb_k^2\in\ExtB$ and $\beta\xx\sqsubseteq_l\bb_k^2$, we have $\xx\in\StB{-1}$. Further since $\alpha\xx$ is an $\lB$-string, Proposition \ref{prop:6.11} gives that $\alpha\in Q_1$, a contradiction. This completes the proof of the claim.\hfill$\blacksquare$

Since $\bb_{j+1}\bb_j$ is an $\lB$-string, it follows from the claim that $\bb_{j+1}<_l\bb_j $ for every $j\in\{1,2,\cdots ,n-1\}$. Using transitivity of $<_l$, we get $\bb_n <_l \bb_1$. However $\bb_1\bb_n$ being a substring of a cyclic permutation of $\bb$ is also an $\lB$-string, and hence the claim gives that $\bb_1<_l\bb_n$, which is a contradiction. Therefore $n=1$ and $a_1=1$, which shows that $\bb$ is prime.
\end{proof}

\begin{cor}\label{cor:6.15}
If $\xx \in \StB1$ then $\braclB(\xx) = \,^\infty\bb\uu\xx$ for some prime $\sB$-band $\bb$.
\end{cor}
\begin{proof}
By Proposition \ref{prop:6.5}, $\braclB(\xx) = \,^\infty\bb\uu\xx$ for some band $\bb$. By Remark \ref{rem:6.3}, we have $\braclB(1_{(t(\xx), \varepsilon(\xx))})=\,^\infty\bb\uu$. Thus every cyclic permutation of $\bb$ is an $\lB$-string. But then Proposition \ref{prop:6.14} gives that $\bb$ is a prime band. Finally, Corollary \ref{cor:6.10} yields $\bb\in\BaB$.
\end{proof}

In fact, the band appearing in Corollary \ref{cor:6.15} is more than just a prime band. We now show that it lies in $\BalB$.

\begin{prop}\label{prop:6.16}
If $\xx\in\StB1$ and $\braclB(\xx)=\,^\infty\bb\uu\xx$ for some band $\bb$ and a string $\uu$ then $\bb\in\BalB$.
\end{prop}
\begin{proof}
The existence of $\bb$ and $\uu$ is guaranteed by Proposition \ref{prop:6.5}, whereas Corollary \ref{cor:6.15} gives that $\bb\in\BaB$. Suppose, if possible, $\bb\notin\BalB$. Then there exists an exit $(\beta,\bb')$ of $\bb$ such that $\theta(\beta)=-1$. Rewrite $\,^\infty\bb\uu\xx$ as $\,^\infty\bb'\vv\xx$ for some string $\vv$. Since $\theta(\bb')=1$, Remark \ref{rem:6.7} yields $n\in\N$ such that $\lB^n(\xx)=\bb'\vv\xx$. Since $\beta\bb'\vv\xx\in\ExtB\subseteq\StB{}$, we have $\bb'\vv\xx\in\StB{-1}$, a contradiction to Remark \ref{rem:6.4}.
\end{proof}

\begin{exmp}
Continuing from Example \ref{exmp: 1lB}, indeed $e_3E_2E_1\in\mathsf{Ba}_l(\sB_3)$ and $e_3E_2E_1$ is a prime band.
\end{exmp}
Combining Propositions \ref{prop:6.5}, \ref{prop:6.16} and Corollary \ref{cor:6.15}, we have the following result.
\begin{cor}\label{BSt1toBalb}
If $\xx\in\StB1$ then $\braclB(\xx)=\,^\infty\bb\uu\xx$ for some string $\uu$ and $\bb\in\BalB$.
\end{cor}

Since $\StB{1}\neq\emptyset$, the above result guarantees the existence of a band in $\BalB$.
\begin{cor}\label{BalB non empty}
If $\sB\in\QBa$ then $\BalB\neq\emptyset$.
\end{cor}

In fact, we can guarantee that each $\bb\in\BalB$ occurs in the conclusion of Proposition \ref{prop:6.16} for some $\xx$.
\begin{prop}\label{prop:6.16a}
If $\bb\in\BalB$ then there exist $\xx\in\StB1$ and a string $\uu$ such that $\braclB(\xx)=\,^\infty\bb\uu\xx$.
\end{prop}
\begin{proof}
Since $\bb\in\BaB$, we have $\theta(\bb) = 1$ and $\bb \in \StB1$. Thus by Proposition \ref{prop:6.16}, there exist $\bb_1\in\BalB$ and a string $\uu$ such that $\braclB(\bb)=\,^\infty\bb_1\uu\bb$. Suppose, if possible, $\,^\infty\bb_1\uu\bb\neq \,^\infty\bb$. Since $\bb \in \BalB$, we get that $\theta(\,^\infty\bb_1\uu\bb \mid \,^\infty\bb)=1$. Let $\zz\bb:=\,^\infty\bb_1\uu\bb\sqcap_l\,^\infty\bb$. By Remark \ref{rem:6.7}, there exists $n \geq 0$ such that $\lB^n(\bb)=\zz\bb$. Since $\theta(\,^\infty\bb_1\uu\bb)=\theta(\bb)=1$, we get that $|\zz|>0$. Hence $n>0$ and, in particular, $\zz\bb\in\StB{-1}$, a contradiction to Remark \ref{rem:6.4}.
\end{proof}

The finiteness of the set $\BalB$ can be concluded from the next result which is obtained by combining Propositions \ref{prop:6.16} and \ref{prop:6.16a}.
\begin{cor}\label{BalB prime}
If $\sB\in\QBa$ and $\bb\in\BalB$ then $\bb$ is prime.
\end{cor}

\begin{rem}\label{Balalbbandsintersection}
If $\sB\in\QBa$ is non-domestic then $\BalB\cap\BalbB=\emptyset$ in view of Remark \ref{nondomexitexist}. On the other hand, if $\sB$ is domestic then $\BalB=\BalbB=\sB$.
\end{rem}

\section{Extending the $\sB$-neighbour operators}
\label{sec: Extending neighbours}
In section \ref{sec:finiteness results}, we defined sets $\OST{j}$ for $j\in\{-1,1\}$ containing strings which eventually reach $\sB\in\QBa_i(\xx_0)$. We would like to extend the function $\lB$ on this bigger set, $\OST{1}$, and obtain a result similar to Corollary \ref{BSt1toBalb}.

\begin{prop}\label{existenceofwlb}
If $\xx \in\OST{1}\setminus\{\MM_i(\xx_0)\} $ then there exists $\zz\in\OST{1}$ such that $\xx\sqsubset_l\zz\sqsubseteq_ll(\xx)$.
\end{prop}
\begin{proof}
If $\xx\sqsubset_l\MM_i(\xx_0)$ then the result holds trivially. So assume $\xx\not\sqsubseteq_l\MM_i(\xx_0)$. Then Remark \ref{rem: Strings in hammock WStr} yields $\xx\in\OSt{1}$. If possible, suppose that $\yy\notin\OSt{1}$ for each $\xx \sqsubset_l\yy\sqsubseteq_ll(\xx)$.

Since $\xx\in\OSt{1}$, there are $\bb\in\CycB$ and $\uu\in\St$ such that $\bb\uu\xx\in\St$ and $\theta(\bb\uu)=1$. Let $\zz:=\bb^2\uu\xx\sqcap_ll(\xx)$. Clearly $\zz\sqsubset_l\bb^2\uu\xx$. As $\theta(\bb\uu) = 1$, we have $\xx\sqsubset_l\zz\sqsubseteq_ll(\xx)$. By our assumption, $\zz\notin\OSt1$. Hence $\theta(\bb^2\uu\xx\mid\zz)=-1$. Let $\alpha\in Q_1$ be such that $\alpha\zz\sqsubseteq_l\bb^2\uu\xx$. Since $\xx\sqsubset_l\zz\sqsubseteq_ll(\xx)$, we get $\alpha\zz\sqsubseteq_ll(\xx)$, which contradicts that $\zz=\bb\uu\xx\sqcap_ll(\xx)$.
\end{proof}

It immediately follows from Proposition \ref{existenceofwlb} that $\xx\sqsubset_lc_\sB(l(\xx))$, which motivates the following definition.
\begin{defn}
Define $\LB:\OST{1}\setminus\{\MM_i(\xx_0)\}\to\OST{1}$ by $\LB(\xx):=c_\sB(\la(\xx))$, i.e, by choosing $\LB(\xx)$ to be the maximal (possibly equal) left substring of $l(\xx)$ such that $\LB(\xx)\in\OST{1}$.
\end{defn}

\begin{prop}\label{LB Rem 1}
If $\xx\in\OST1\setminus\{\MM_i(\xx_0)\}$ then $\xx<_l\LB(\xx)$, $\varphi_\sB(\LB(\xx))=1$ and $H_l^{-1}(\LB(\xx))=(\xx,\LB(\xx)]$.
\end{prop}

\begin{proof}
By the definition of $\LB(\xx)$, it is clear that $\xx<_l\LB(\xx)$, $\LB(\xx)\notin\OST{-1}$, and hence $\varphi_\sB(\LB(\xx))=1$.

Let $\zz\in(\xx,\LB(\xx))$. If $\xx\not\sqsubseteq_l\zz$ then $\zz\sqcap_l\xx=\zz\sqcap_l\LB(\xx)$. This implies that both $\xx$ and $\LB(\xx)$ lie on the same side of $\zz$ in the hammock, a contradiction. Hence $\xx\sqsubset_l\zz$. Since $\theta(\zz\mid\xx)=1$, we get $a\in Q_1$ such that $A\xx\sqsubseteq_l\zz$. By definition, $A\xx\sqsubseteq_l\LB(\xx)$. Hence $A\xx\sqsubseteq_l\LB(\xx)\sqcap_l\zz$ and $\theta(\LB(\xx)\mid\zz)=1$. If $\LB(\xx)=\xx'A\xx$ then either $|\xx'|=0$ or $\delta(\xx')=-1$. Therefore we conclude that $\LB(\xx)\sqsubset_l\zz$, which together with $\theta(\zz\mid\LB(\xx))=-1$ gives $\zz\in H_l^{-1}(\LB(\xx))$ as required.
\end{proof}

Combining the above with Remark \ref{hammockcoverrepr}, we conclude that $\LB(\xx)$ is the immediate successor of $\xx$ in $\OST{}$.
\begin{cor}\label{neighbourcondenscommute}
If $\xx\in\OST1\setminus\{\MM_i(\xx_0)\}$ then $(\xx,\LB(\xx))\cap\OST{}=\emptyset$.
\end{cor}

Now we show that the function $\LB$ is indeed an extension of the function $\lB$ on a larger domain.
\begin{prop}\label{LB is ext of lB}
If $\xx\in\STB1$ then $\lB(\xx)=\LB(\xx)$.
\end{prop}
\begin{proof}
Let $\lB(\xx)=:\zz'\xx$ and $\LB(\xx)=:\zz\xx$. Since $\lB(\xx)\in\STB{1}\subseteq\OST1$, we have $\lB(\xx)\sqsubseteq_l\zz\xx\sqsubseteq_ll(\xx)$. If $\zz=\zz'$ then there is nothing to prove. Therefore assume that $\zz'\sqsubset_l\zz$. By Remark \ref{rem:6.3}, $1_{(t(\xx), \varepsilon(\xx))}\in\StB1$ and $\lB(1_{(t(\xx), \varepsilon(\xx))})=\zz'$. Then Corollary \ref{cor:6.10} yields $\bb\in\CycB$ such that $\zz'\sqsubseteq_l\bb^n$ for some $n\in\N^+$.

Since $\zz'\bb$ is a string, $\theta(\zz')=1$ and $\theta(\zz\mid\zz')=-1$, we get that $\zz\bb$ is a string. Since $\zz'\sqsubset_l\zz \sqsubseteq_ll(1_{(t(\xx),\varepsilon(\xx))})$ we get $\delta(\zz)=0$. As $\zz\xx\in\OST1$, there is $\bb'\in\CycB$ and a string $\uu$ such that $\bb'\uu\zz\xx\in\St$ and $\theta(\bb'\uu) = 1$. Furthermore, $\delta(\zz) =0$ gives that $\bb'\uu\zz\bb$ is a string. Thus $\zz\bb\uu\zz\bb'\uu$ is a power of a $\mathsf{B}$-cycle, $\theta(\zz\bb\yy\bb'\uu) = 1$ and $(\zz\bb\yy\bb'\uu)\zz$ is a string. Hence $\zz\in\StB1$. Since $\delta(\zz)=0$, we have $\zz\xx\in\StB1$. Since $\zz\xx\sqsubseteq_l\la(\xx)$, we get a contradiction to the definition of $\lB(\xx)$.
\end{proof}

For $\xx\in\OST1$ we inductively define the powers of the function $\LB$ by $\LB^0(\xx):=\xx$ and $\LB^{n+1}(\xx):=\LB(\LB^n(\xx))$ for $n\in\N$, if $\LB^n(\xx)\neq\MM_i(\xx_0)$. Note that $\LB^n(\xx)$ exists for each $n$ if and only if $\xx\in\OSt{1}$. Whenever this happens, using $\LB^n(\xx)\sqsubset_l\LB^{n+1}(\xx)$ we get that $\lim_{n\to\infty}\LB^n(\xx)$ is a left $\N$-string; denote this limit by $\brac{1}{\LB}(\xx)$.

\begin{rem}\label{LB Rem 2}
If $\LB^n(\xx)=\uu\yy\xx$, where $\yy$ and $\uu$ are strings with $|\uu|>0$, then $\theta(\uu)=1$ if and only if there exists $0\leq k<n$ such that $\LB^k(\xx)=\yy\xx$.
\end{rem}

The proof of the following result is along similar lines as the proof of Proposition \ref{prop:6.5}.

\begin{prop}
If $\xx\in\OST1\cap\OSt{1}$ then $\brac{1}{\LB}(\xx)=\,^\infty\bb\uu\xx$ for some band $\bb$ and some string $\uu$.
\end{prop}

However the band $\bb$ obtained above might not be a $\sB$-band as is evident from the following example.

\begin{exmp}
Continuing from Example \ref{exmp: QBa in running example}, for $a_0\in\overline{\mathsf{St}}_1(a_0,1;\sB_2)\cap\overline{\mathsf{St}}_1(\sB_2)$, we have $\brac{1}{l_{\sB_2}}(a_0)=\,^\infty(b_1B_4b_3B_2)b_1a_4a_3A_1a_0$, but $b_1B_4b_3B_2\notin\sB_2$.
\end{exmp}

This issue is resolved if $\sB$ is minimal for $(\xx_0,i)$.
\begin{prop}\label{OST to infty bux}
Suppose $\sB$ is minimal for $(\xx_0,i)$. If $\xx\in\OST1\cap\OSt{1}$ then $\brac{1}{\LB}(\xx)=\,^\infty\bb\uu\xx$ for some $\bb\in\BalB$ and some string $\uu$.
\end{prop}
\begin{proof}
By Corollary \ref{OST-STB is finite}, we have that $\LB^n(\xx)\in\STB{1}$ for some $n\in\N$. Therefore $\brac{1}{\LB}(\xx)=\brac{1}{\LB}(\LB^n(\xx))=\brac{1}{\lB}(\LB^n(\xx))$, where the last equality follows from Proposition \ref{LB is ext of lB}. The conclusion is then immediate from Corollary \ref{BSt1toBalb}.
\end{proof}
We end this section with a useful result that will be used to prove the density of some special strings called $\sB$-centers in Proposition \ref{Density of Cent}.
\begin{prop}\label{1Lb<1LbB}
Suppose $\sB\in\QBa_i(\xx_0)$ is non-domestic and minimal for $(\xx_0,i)$.
If $\xx\in\OST1$ and $\yy\in\OST{-1}$ such that $\xx<_l\yy$ then $\brac{1}{\LB}(\xx)<_l\brac{1}{\LbB}(\yy)$.
\end{prop}
\begin{proof}
Proposition \ref{x<y implies x in ST1B} gives $\xx\in\OSt1$ and $\yy\in\OSt{-1}$.
In view of Proposition \ref{OST to infty bux}, let $\brac{1}{\LB}(\xx)=:\,^\infty\bb_1\uu_1\xx$ and $\brac{1}{\LbB}(\yy)=:\,^\infty\bb_2\uu_2\yy$, where $\bb_1\in\BalB$, $\bb_2\in\BalbB$ and $\uu_1,\uu_2$ are strings. Since $\BalB\cap\BalbB=\emptyset$, we get that $\bb_1$ is not a cyclic permutation of $\bb_2$, which implies $\,^\infty\bb_1\uu_1\xx\neq\,^\infty\bb_2\uu_2\yy$. Let $\ww:=\,^\infty\bb_1\uu_1\xx\sqcap_l\,^\infty\bb_2\uu_2\yy$ so that $\ww\in\OSt{\pm1}$. If $\theta(\,^\infty\bb_1\uu_1\xx\mid\ww)=1$ then Remark \ref{LB Rem 2} yields $n\in\N^+$ such that $\LB^n(\xx)=\ww$, a contradiction to Proposition \ref{LB Rem 1}. Hence $\theta(\,^\infty\bb_1\uu_1\xx\mid\ww)=-1$, which completes the proof.
\end{proof}

\section{$\sB$-centers}\label{sec: Centers}
Given $\sB\in\QBa$, the presence of some special strings in $\STB{\pm1}$, which we shall call \emph{$\sB$-centers}, characterizes non-domesticity of $\sB$. The suborder of such strings is a dense linear order, and it is responsible for the shuffle structure (see clause (4) of Definition \ref{defn: LOFD}) in the hammocks for non-domestic string algebras. The absence of $\sB$-centers in domestic string algebras (Proposition \ref{Absence of B center in domestic}) thus prohibits domestic string algebras to have a shuffle structure.

To define $\sB$-centers, we need a notion called \emph{$\sB$-equivalence} which guarantees that the maximal scattered intervals in $\OST{}$ around $\sB$-centers are canonically isomorphic.

\begin{defn}\label{defn: Bstrand equiv}
Let $\sB\in\QBa_i(\xx_0)$. Two strings $\xx$ and $\yy$ in $\STB{\pm1}$ are said to be \emph{$\sB$-equivalent}, denoted $\xx\equiv_\sB\yy$, if there exist distinct syllables $\alpha$ and $\beta$ such that $\alpha\xx$, $\beta\xx$, $\alpha\yy$ and $\beta\yy$ are strings.
\end{defn}
It is trivial to note that $\equiv_\sB$ is an equivalence relation on $\STB{\pm1}$. The following remark notes that there are finitely many $\equiv_\sB$-classes.

\begin{rem}\label{rem: fin many B equiv classes}
Associated to each string $\xx$ of a $\sB$-equivalence class there is a unique pair $(\alpha,\beta)\in Q_1\times Q_1^-$ for which $\alpha\xx$ and $\beta\xx$ are strings. Therefore the assignment of each $\sB$-equivalence class to its corresponding pair in $Q_1\times Q_1^-$ is injective. Since $Q_1\times Q_1^-$ is finite in every string algebra, we have that there are finitely many $\sB$-equivalence classes.
\end{rem}

\begin{exmp}\label{exmp: Bequiv classes}
Recall from Example \ref{exmp: QBa in running example} that $\sB_3\in\QBa_1(a_0)$ is non-domestic. There are three $\sB_3$-equivalence classes in $\overline{\mathsf{St}}_{\pm1}(a_0,1;\sB)$ with strings $E_1A_2A_1a_0,\ G_1FE_2E_1A_2A_1a_0,\ k_1h_2H_1G_1FE_2E_1A_2A_1a_0$ as their representatives.
\end{exmp}

The notion of $\sB$-equivalence is strictly weaker than $H$-equivalence as demonstrated in Example \ref{Example B equiv weaker than H equiv}. However, these notions coincide in the case of gentle string algebras.

\begin{exmp}\label{Example B equiv weaker than H equiv}
Consider the string algebra $\Gamma'$ from Figure \ref{fig: Hequiv Bequiv}. Choose $j\in\{1,-1\}$ such that $f1_{(v_5,j)}$ is a string. There is only one non-domestic $\sB\in\QBa_{-1}(1_{(v_5,j)})$. Consider the strings $f,feDf\in\mathsf{St}_{\pm1}(1_{(v_5,j)},-1;\sB)$. Since $cf,cfeDf,Df,DfeDf$ are strings, we have $f\equiv_{\sB}feDf$. On the other hand, $f\not\equiv_HfeDf$ since $acf$ is a string but $acfeDf$ is not.
\end{exmp}
\begin{figure}[h]
    \centering
    \begin{tikzcd}
                    & v_2 &                                      & v_4 \arrow[ld, "d"'] \arrow[rd, "e"] &                                                              \\
v_1 \arrow[ru, "a"] &     & v_3 \arrow[ll, "c"] \arrow[lu, "b"'] &                                      & v_5 \arrow[ll, "f"] \arrow[llll, "g", bend left, shift left]
\end{tikzcd}
    \caption{$\Gamma'$ with $\rho=\{bf,cd,ge,ag,acfe\}$}
    \label{fig: Hequiv Bequiv}
\end{figure}
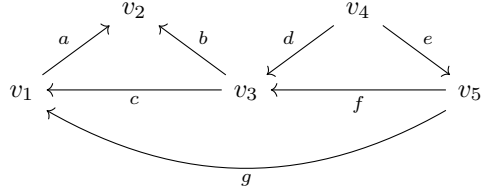

\begin{prop}\label{prop: Bcenter equiv conditions}
Let $\sB\in\QBa_i(\xx_0)$. The following statements are equivalent for a string $\xx\in\STB{\pm1}$.
\begin{enumerate}
    \item There exists $\bb'\in\CycB$ such that $\bb'\xx\in\STB{\pm1}$ and $\bb'\xx\equiv_\sB\xx$.
    
    \item There exist strings $\xx_1,\xx_2$ and $\bb'\in\CycB$ such that $\xx=\xx_2\xx_1$, $\xx_2\bb'\xx_1\in\STB{\pm1}$, $\xx_2\bb'\xx_1\equiv_\sB\xx$ and $\xx_2\bb'\in\ExtB$.
    
    \item There exist strings $\zz,\uu$ such that $\delta(\zz)=0$, $\zz\in\ExtB$, 
    $\zz\uu\in\STB{\pm1}$ and $\zz\uu\equiv_\sB\xx$.
    
    \item There exists a syllable $\gamma\in\StB{\pm1}$ such that $1_{(t(\xx),\varepsilon(\xx))}\gamma$ is a string, $\alpha\gamma,\beta\gamma\in\ExtB$ for distinct syllables $\alpha,\beta$ such that $\alpha\xx,\beta\xx\in\St$.
\end{enumerate}
\end{prop}
\begin{proof}
(1)$\implies$(2): This is immediate as we can take $\xx_1=\xx$ and $\xx_2=1_{(v,j)}$ for appropriate $(v,j)$ such that $1_{(v,j)}\xx_1$ is a string.

(2)$\implies$(3): Take $\zz=\xx_2\bb'$ and $\uu=\xx_1$. Since $\bb'$ is a cyclic permutation of a band and $\bb'\sqsubseteq_l\zz$, we have $\delta(\zz)=0$. Moreover, $\zz\uu=\xx_2\bb'\xx_1\equiv_\sB\xx$.

(3)$\implies$(4): Since $\delta(\zz)=0$, we have $|\zz|>0$. Take $\gamma$ to be the last syllable of $\zz$. Without loss, assume $\theta(\alpha)=1=-\theta(\beta)$. Since $\zz\uu\in\STB{\pm1}$, there exist $\bb'_1,\bb'_2\in\CycB$ with $\theta(\bb'_1)=1=-\theta(\bb'_2)$ such that $\bb'_1\zz\uu$ and $\bb'_2\zz\uu$ are strings. Note that the first syllables of $\bb'_1$ and $\bb'_2$ are $\alpha$ and $\beta$ respectively. Also since $\zz\in\ExtB$, there exists $\bb'\in\CycB$ such that $\bb'^n=\ww\zz$ for some string $\ww$ and $n\in\N$. Let $\vv$ be a string such that $\bb'\vv\bb'_1$ is a string. Since $\delta(\zz)=0$ and $\zz\bb'$ and $\bb'_1\zz$ are strings, we conclude that $\zz\bb'\vv\bb'_1$ is a power of a cyclic permutation of a $\sB$-band. Since the last syllable of $\zz$ is $\gamma$ and the first syllable of $\bb'_1$ is $\alpha$, it follows that $\alpha\gamma\in\ExtB$. Similarly we can show that $\beta\gamma\in\ExtB$.

(4)$\implies$(1): Without loss of generality, let $\theta(\gamma)=\theta(\alpha)=-\theta(\beta)$. There are two cases.

\textbf{Case 1:} Either $|\xx|=0$ or $\theta(\gamma')=\theta(\gamma)$, where $\gamma'$ is the last syllable of $\xx$.\\
Since $\beta\gamma\in\ExtB$, there is $\bb'\in\CycB$ and $n\in\N^+$ such that $\gamma\ww\beta=\bb'^n$ for some string $\ww$. Further, since $\alpha\gamma\in\ExtB$, there exists a string $\ww'$ such that $\alpha\gamma\ww'\bb'$ is a string. Since the first syllable of $\bb'$ is $\beta$ and $\theta(\bb')=-\theta(\gamma)$, we get that $\gamma\ww'\bb'$ is a power of a $\sB$-cycle, say $\bb''$, such that $\alpha\bb''\xx$ and $\beta\bb''\xx$ are strings.

\textbf{Case 2:} $\theta(\gamma')=-\theta(\gamma)$, where $\gamma'$ is the last syllable of $\xx$.\\
Since $\alpha\gamma\in\ExtB$, there exists $\bb'\in\CycB$ and $n\in\N^+$ such that $\gamma\ww\alpha=\bb'^n$ for some string $\ww$. Since the first syllable of $\bb'$ is $\alpha$, and thus $\theta(\bb')=-\theta(\gamma')$, we get that $\bb'\xx$ is a string. Further since the last syllable of $\bb'$ is $\gamma$, and $\theta(\gamma)=-\theta(\beta)$ we get that $\beta\bb'$ is a string. Since $\delta(\bb')=0$, it follows that $\beta\bb'\xx$ is a string too, thus completing the proof.
\end{proof}
\begin{defn}
Say that $\xx\in\STB{\pm1}$ is a \emph{$\sB$-center} if it satisfies one of the equivalent conditions of Proposition \ref{prop: Bcenter equiv conditions}. Denote the set of all $\sB$-centers by $\Cent$.
\end{defn}
The following example shows that $\Cent$ can be a proper subset of $\STB{\pm1}$.

\begin{exmp}\label{exmp: Centers proper subset STBpm1}
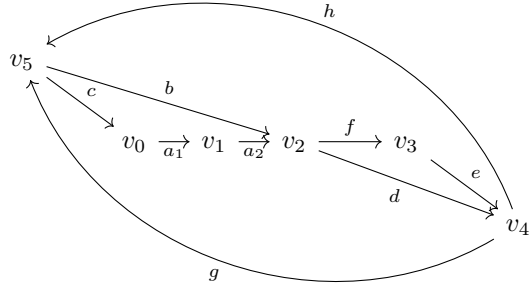
\begin{figure}[h]
\vspace{-10mm}
\[\begin{tikzcd}[row sep=small,column sep=small]
v_5 \arrow[rrdd, "c"] \arrow[rrrrdd, "b"] &  &               &               &                                    &  &                  &  &                                                                                      \\
                                          &  &               &               &                                    &  &                  &  &                                                                                      \\
                                          &  & v_0 \arrow[r,"a_1"'] & v_1 \arrow[r,"a_2"'] & v_2 \arrow[rr,"f"] \arrow[rrrrdd, "d"'] &  & v_3 \arrow[rrdd,"e"] &  &                                                                                      \\
                                          &  &               &               &                                    &  &                  &  &                                                                                      \\
                                          &  &               &               &                                    &  &                  &  & v_4 \arrow[lllllllluuuu, "g", bend left=50] \arrow[lllllllluuuu, "h"', bend right=50]
\end{tikzcd}\]
\vspace{-10mm}
    \caption{$\Gamma''$ with $\rho=\{fa_2,db,ge,hd,bg,ch,da_2a_1c,efb\}$}
    \label{fig:Cent is proper in STBpm1}
\end{figure}
Consider the string algebra $\Gamma''$ from Figure \ref{fig:Cent is proper in STBpm1}. There is a unique non-domestic $\sB\in\QBa_{-1}(a_1)$ with $a_2a_1\in\mathsf{St}_{\pm1}(a_1,-1;\sB)\setminus\mathsf{Cent}(a_1,-1;\sB)$. 
\end{exmp}
\begin{rem}\label{B equiv means same ext}
In view of Remark \ref{rem:6.3}, if $\xx\equiv_\sB\yy$ and $\brac{1}{\LB}(\xx)=\,^\infty\bb\uu\xx$ for some $\bb\in\BalB$ and string $\uu$ then $\brac{1}{\LB}(\yy)=\,^\infty\bb\uu\yy$.
\end{rem}
Given $\xx\in\Cent$, we say that $\xx$ is the center of the interval $\mathcal{I}_{(\xx_0,i;\sB)}(\xx):=(\brac{1}{\LbB}(\xx),\brac{1}{\LB}(\xx))$ in the hammock $(H_l^i(\xx_0),<_l)$. The following result shows that two intervals of the above form are canonically isomorphic if and only if their centers are $\sB$-equivalent.

\begin{prop}\label{shuffle structure of the middle}
Suppose $\xx,\yy\in\Cent$. If $\xx\equiv_\sB\yy$ then for any string $\uu$, $\uu\xx\in\mathcal{I}_{(\xx_0,i;\sB)}(\xx)$ if and only if $\uu\yy\in\mathcal{I}_{(\xx_0,i;\sB)}(\yy)$.
\end{prop}
\begin{proof}
Since $\xx\in\StB{1}$, Corollary \ref{BSt1toBalb} and Proposition \ref{LB is ext of lB} together give $\brac{1}{\LB}(\xx)=\,^\infty\bb\uu_1\xx$ for some string $\uu_1$ and $\bb\in\BalB$. Then Remark \ref{B equiv means same ext} gives $\brac{1}{\LB}(\yy)=\,^\infty\bb\uu_1\yy$. Suppose $\uu\xx \in \mathcal{I}_{(\xx_0,i;\sB)}(\xx)$. Without loss we can assume that $|\uu|>0$ and $\theta(\uu)=1$ so that $\uu\xx\in(\xx,\,^\infty\bb\uu_1\xx)$. If $\uu\yy$ is a string then $\uu\xx\in(\xx,\,^\infty\bb\uu_1\xx)$ immediately implies $\uu\yy\in(\yy,\,^\infty\bb\uu_1\yy)$. Hence it remains to show that $\uu\yy$ is a string. There are two possibilities.

Since $\,^\infty\bb\uu_1\yy$ is a left $\N$-string, if $\uu\sqsubset_l\,^\infty\bb\uu_1$ then $\uu\yy$ is a string.

If $\uu$ and $\,^\infty\bb\uu_1$ fork then let $\zz:=\uu\sqcap_l\,^\infty\bb\uu_1$. Since $\theta(\,^\infty\bb\uu_1)=\theta(\uu)=1$, we have $|\zz|>0$. Thus by the above paragraph, $\zz\yy$ is a string. As $\uu<_l\,^\infty\bb\uu_1$, we get that $\alpha\zz\sqsubseteq_l\uu$ for some $\alpha \in Q_1$. Since $\theta(\zz)=1$, $\zz\yy$ is a string and $\delta(\alpha\zz) = 0$, we get that $\uu\yy$ is a string.
\end{proof}

Now we show the existence of $\sB$-centers when $\sB$ is non-domestic using Remark \ref{nondomexitexist}.
\begin{prop}\label{Existence of B cent in nondomestic}
Let $\xx\in H_l^i(\xx_0)$, $\bb\in\BaB$ and $(\beta,\bb')$ be a non-domestic exit of $\bb$ with $\theta(\beta)=1$. If $\beta\bb'\xx\in H_l^i(\xx_0)$ then $\bb'\xx\in\Cent$.
\end{prop}
\begin{proof}
Since $\theta(\beta)=1$ we have $\theta((\bb')^2\xx\mid\bb'\xx)=-1$. Since $\beta\bb'\in\ExtB$, there exists $\bb''\in\CycB$ such that $\beta\bb'\sqsubseteq_l\bb''$. So for an appropriate cyclic permutation $\bb'''$ of $\bb''$, we have that $\bb'''\bb'\xx$ is a string with $\theta(\bb''')=\theta(\beta)=1$ thus showing $\bb'\xx\in\STB{1}$. Also $\bb'\bb'\xx$ is a string with $\theta(\bb')=-1$ which gives $\bb'\xx\in\STB{-1}$. Finally, since $\delta(\bb')=0$ we  have $\bb'\bb'\xx\equiv_\sB\bb'\xx$. In view of Proposition \ref{prop: Bcenter equiv conditions}(1), we conclude that $\bb'\xx\in\Cent$.
\end{proof}
Contrary to the above result, there are no $\sB$-centers for any domestic $\sB\in\QBa$.
\begin{prop}\label{Absence of B center in domestic}
If $\sB\in\QBa_i(\xx_0)$ is domestic then $\Cent=\emptyset$.
\end{prop}
\begin{proof}
If possible, let $\yy\in\Cent$. Then Proposition \ref{prop: Bcenter equiv conditions}(1) yields $\bb'\in\CycB$ such that $\bb'\yy\in\STB{\pm1}$ and $\bb'\yy\equiv_\sB\yy$. Since $\bb'\yy\in\STB{\pm1}$ there exists $\bb''\in\CycB$ with $\theta(\bb')=-\theta(\bb'')$ such that $\bb''\bb'\yy$ is a string. Since $\bb',\bb''\in\ExtB$, $\bb'\uu\bb''$ is a string for some string $\uu$. Since $\sB$ is domestic and $\uu\bb''\bb'$ is a mixed cyclic string, we get $\uu\bb''\bb'=\bb_1^n$ for some $n\in\N^+$ and a cyclic permutation $\bb_1$ of $\bb'$. In view of Corollary \ref{cyclic perm are diff; cor}, we get $\bb'=\bb_1$, which gives a contradiction to $\theta(\bb')=-\theta(\bb'')$.
\end{proof}

We end this section with the following proposition, which is the key to investigating the order type of the set of $\sB$-centers.
\begin{prop}\label{Density of Cent}
Suppose $\sB\in\QBa_i(\xx_0)$ is non-domestic and minimal for $(\xx_0,i)$, and $\xx<_l\yy$ for some $\xx\in\OST1$ and $\yy\in\OST{-1}$. Then for any $\zz\in\Cent$ there exists $\zz'\in\Cent$ such that $\zz'\equiv_\sB\zz$ and $\brac{1}{\LB}(\xx)<_l\zz'<_l\brac{1}{\LbB}(\yy).$
\end{prop}
\begin{proof}
Proposition \ref{1Lb<1LbB} gives $\brac{1}{\LB}(\xx)<_l\brac{1}{\LbB}(\yy)$. Let $\ww:=\brac{1}{\LB}(\xx)\sqcap_l\brac{1}{\LbB}(\yy)$. Proposition \ref{prop: Bcenter equiv conditions}(3) gives the existence of strings $\zz_1,\zz_2$ such that $\delta(\zz_2)=0$, $\zz_2\in\ExtB$, $\zz_2\zz_1\in\STB{\pm1}$ and $\zz_2\zz_1\equiv_\sB\zz$.

In view of Proposition \ref{OST to infty bux}, let $\brac{1}{\LB}(\xx)=\,^\infty\bb\uu_1\xx$ for some $\bb\in\BalB$ and some string $\uu_1$. Since $\theta(\,^\infty\bb\uu_1\xx\mid\ww)=-1$, there exists $n\in\N^+$ such that $\theta(\bb^n\uu_1\xx\mid\ww)=-1$. As $\bb\in\BalB$, consider an exit $(\beta,\bb')$ of $\bb$. There exists a string $\uu\sqsubset_l\bb$ such that $\beta\bb'\uu\bb$ is a string. Since $\beta\bb',\zz_2\in\ExtB$, there exists a string $\vv$ such that $\zz_2\vv\beta\bb'$ is a string. Now $\delta(\bb')=0$ implies $\zz':=\zz_2\vv\beta\bb'\uu\bb^n\uu_1\xx$ is a string, $\zz'\in\STB{\pm1}$ and $\zz'\equiv_\sB\zz$. Since $\theta(\bb^n\uu_1\xx\mid\ww)=-1$, we get $\theta(\zz'\mid\ww)=-1$. Also $\theta(\beta)=1$ implies $\theta(\,^\infty\bb\uu_1\xx\mid\zz')=-1$. Therefore we have $\brac{1}{\LB}(\xx)<_l\zz'<_l\ww<_l\brac{1}{\LbB}(\yy)$ to complete the proof.
\end{proof}

\section{Computation of the order type of hammocks}\label{sec: main thm}
So far we have collected most of the ingredients to prove the main result (Theorem \ref{main}), whose proof we finish in this section. Furthermore, we prove a partial converse followed by a discussion about the potential impossibility of the converse in its full generality.

Since there are finitely many strings that are band-free relative to $(\xx_0,i)$, recall from Corollary \ref{BandfreeOST-OSt} that if $\sB\in\QBa$ is minimal for $(\xx_0,i)$ then there are finitely many strings in $\OST{j}\setminus\STB{j}$ for each $j\in\{-1,1\}$. A simple set theoretic manipulation yields $$\OST{\pm1}\setminus\STB{\pm1}\subseteq(\OST{1}\setminus\STB{1})\cup(\OST{-1}\setminus\STB{-1}),$$ and thus $\OST{\pm1}\setminus\STB{\pm1}$ is finite. The following proposition shows that the set $\STB{\pm1}\setminus\Cent$ is also finite when $\sB$ is minimal for $(\xx_0,i)$.

\begin{prop}\label{rem: STBpm1-Cent is Bbandfree}
If $\sB\in\QBa$ is minimal for $(\xx_0,i)$ and $\xx\in\STB{\pm1}\setminus\Cent$ then $\xx$ is band-free relative to $(\xx_0,i)$.
\end{prop}
\begin{proof}
If $\xx\in\STB{\pm1}\setminus\Cent$ such that $\xx=\xx_2\bb\xx_1$ for some $\bb\in\CycB$, then we have $\xx\equiv_\sB\xx_2\bb^2\xx_1$, which implies $\xx\in\Cent$ by Proposition \ref{prop: Bcenter equiv conditions}(2), a contradiction. This proves that $\xx$ is band-free with respect to $\sB$. As a consequence, if $\sB\in\QBa$ is minimal for $(\xx_0,i)$ then $\xx$ is band-free relative to $(\xx_0,i)$.
\end{proof}

Since $$\OST{\pm1}\setminus\Cent=(\OST{\pm1}\setminus\STB{\pm1})\cup(\STB{\pm1}\setminus\Cent),$$
we get that there are finitely many strings in $\OST{\pm1}\setminus\Cent$. This observation helps us to break the linear order $(\OST{},<_l)$ into finitely many ``irreducible'' intervals.

\begin{defn}\label{defn: B Beam}
Let $\sB\in\QBa$ be minimal for $(\xx_0,i)$. Call an interval $[\xx,\yy]$ in $H_l^i(\xx_0)$ a \emph{$\sB$-beam} if $\xx,\yy\in\OST{\pm1}\setminus\Cent$, $\xx<_l\yy$ and  $(\xx,\yy)\cap\OST{\pm1}\subseteq\Cent$.
\end{defn}
Since the set $\OST{\pm1}\setminus\Cent$ is finite when $\sB\in\QBa$ is minimal for $(\xx_0,i)$, we get that there are finitely many $\sB$-beams. Let $n_\sB$ denote the number of $\sB$-beams.

If $\sB\in\QBa$ is minimal for $(\xx_0,i)$ and $\yy_0<_l\yy_1<_l\cdots<_l\yy_{n_\sB}$ is the complete list of elements in $\OST{\pm1}\setminus\Cent$ then 
\begin{equation}\label{breaking of hammock into beams}
(H_l^i(\xx_0),<_l)=[\mm_i(\xx_0),\yy_0]\dotplus[\yy_0,\yy_1]\dotplus\cdots\dotplus[\yy_{n_\sB-1},\yy_{n_{\sB}}]\dotplus[\yy_{n_{\sB}},\MM_i(\xx_0)],
\end{equation}
where each $[\yy_j,\yy_{j+1}]$ is a $\sB$-beam.

In view of Remark \ref{cbI=IcapOST}, we have
\begin{equation}\label{breaking of OST into beams}
(c_\sB(H_l^i(\xx_0)),<_l)=c_\sB([\mm_i(\xx_0),\yy_0])\dotplus c_\sB([\yy_0,\yy_1])\dotplus\cdots\dotplus c_\sB([\yy_{n_\sB-1},\yy_{n_{\sB}}])\dotplus c_\sB([\yy_{n_{\sB}},\MM_i(\xx_0)]).
\end{equation}

\begin{exmp}\label{exmp:beams for running example}
Continuing Example \ref{exmp: QBa in running example}, since $\sB_1$ is minimal for $(a_0,1)$, Equation \eqref{breaking of hammock into beams} takes the form 
\begin{equation}\label{computation: hammock}
H_l^1(a_0)=[a_0,a_0]\dotplus[a_0,a_3A_1a_0]\dotplus[a_3A_1a_0,A_1a_0]\dotplus[A_1a_0,H_1G_1FE_2E_1A_2A_1a_0].
\end{equation}
and Equation \eqref{breaking of OST into beams} takes the form
\begin{equation}\label{computation: condensation OST}
    c_{\sB_1}(H_l^1(a_0))=c_{\sB_1}([a_0,a_3A_1a_0])\dotplus c_{\sB_1}([a_3A_1a_0,A_1a_0])\dotplus c_{\sB_1}([A_1a_0,H_1G_1FE_2E_1A_2A_1a_0]).
\end{equation}
\end{exmp}

\begin{prop}\label{[mix0,y1] is finite}
The sets $c_\sB([\mm_i(\xx_0),\yy_0])$ and $c_\sB([\yy_{n_\sB},\MM_i(\xx_0)])$ are finite.
\end{prop}
\begin{proof}
Without loss assume that $i=-1$. Recall from the proof of Proposition \ref{OSTpm is bounded} that $\yy_{n_\sB}=\xx_0=\MM_i(\xx_0)$, and hence $|[\yy_{n_\sB},\MM_i(\xx_0)]|=1$. On the other hand, the same proof describes $\yy_0$ as the longest left substring of $\mm_i(\xx_0)$ that lies in $\OST{\pm1}$. Hence $\zz\in c_\sB([\mm_i(\xx_0),\yy_0])$ if and only if $\yy_0\sqsubseteq_l\zz\sqsubseteq_l\mm_i(\xx_0)$. Since there are only finitely left substrings of $\mm_i(\xx_0)$, the proof is complete.
\end{proof}

Recall from Remark \ref{rem: fin many B equiv classes} that the set $\STB{\pm1}/{\equiv_\sB}$ is finite. Let $k_\sB:=|\Cent/{\equiv_\sB}|$. Propositions \ref{Existence of B cent in nondomestic} and \ref{Absence of B center in domestic} together imply that $k_\sB=0$ if and only if $\sB$ is domestic.

The next result is a consequence of Proposition \ref{Density of Cent}, which computes the order type of the suborder of $\sB$-centers inside a $\sB$-beam.
\begin{cor}\label{Centers are dense in a beam}
If $\sB\in\QBa$ is minimal for $(\xx_0,i)$ and $[\xx,\yy]$ be a $\sB$-beam then $$(\Cent\cap[\xx,\yy],<_l)\cong\Xi(\underbrace{1,1,\cdots,1}_{k_\sB\text{ times}}).$$
\end{cor}
\begin{proof}
Since $H_l^i(\xx_0)$ is countable, the set $\Cent\cap[\xx,\yy]$ is also countable. If $\sB$ is domestic then Proposition \ref{Absence of B center in domestic} implies that both sides are empty linear orders. On the other hand, if $\sB$ is non-domestic then Proposition \ref{Existence of B cent in nondomestic} implies that $\Cent\neq\emptyset$. Therefore it suffices to prove that each $\sB$-equivalence class of $\sB$-centers intersects the beam $[\xx,\yy]$ in a non-empty, unbounded, and dense fashion. Let $\ww\in\Cent$.

\begin{itemize}
    \item[Non-empty:] Since $\xx\in\OST{1}$ and $\yy\in\OST{-1}$, Proposition \ref{Density of Cent} applied on $\xx<_l\yy$ yields $\zz\in\Cent\cap(\xx,\yy)$ such that $\zz\equiv_\sB\ww$.
    \item[Unbounded:] Let $\zz\in\Cent\cap(\xx,\yy)$. Since $\zz\in\STB{\pm1}$, Proposition \ref{Density of Cent} applied on $\xx<_l\zz$ and $\zz<_l\yy$ guarantees the existence of $\zz_1\in\Cent\cap(\zz,\yy)$ and $\zz_2\in\Cent\cap(\xx,\zz)$ respectively such that $\zz_1\equiv_\sB\zz_2\equiv_\sB\ww$.
    \item[Dense:] If $\zz_1,\zz_2\in\Cent\cap(\xx,\yy)$ then Proposition \ref{Density of Cent} applied on $\zz_1<_l\zz_2$ yields $\zz_3\in\Cent\cap(\zz_1,\zz_2)$ such that $\zz_3\equiv_\sB\ww$.
\end{itemize}
\end{proof}

\begin{prop}\label{structure of beams}
Let $\sB\in\QBa$ be minimal for $(\xx_0,i)$. If $\xx\in\OST{1}$ then there exists $\yy\in\OST{\pm1}$ and $n\in\N$ such that $\LB^n(\yy)=\xx$.
\end{prop}
\begin{proof}
In view of Proposition \ref{OST to infty bux}, let $\brac{1}{\LB}(\xx)=:\,^\infty\bb\uu\xx$ for some string $\uu$ and $\bb\in\BalB$. Consider the shortest string $\yy\in\OST{1}$ such that $\brac{1}{\LB}(\yy)=\,^\infty\bb\uu\xx$. We claim that $\yy\in\OST{-1}$.

If $\yy\in\OSt{-1}$ then we are done. Otherwise, in view of Remark \ref{rem: Strings in hammock WStr}, we need to show that $\yy\sqsubseteq_l\mm_i(\xx_0)$.

If possible, let $\xx_0\sqsubseteq_l\zz\sqsubset_l\yy$ be a string such that $\theta(\yy\mid\zz)=1$. Without loss, take $\zz$ to be the longest such string. Then $\yy=\ww\alpha\zz$, where $\alpha\in Q_1^-$ and $\ww$ satisfies either $|\ww|=0$ or $\delta(\ww)=-1$. Clearly $\zz\in\OST{1}$ and $\yy\sqsubseteq_l\LB(\zz)$.

If $\yy\sqsubset_l\LB(\zz)$ then there exists $\beta\in Q_1$ such that $\LB(\zz)\sqsupseteq_l\beta\yy\in\OST1$. Thus $\yy\in\OST{-1}$. Since $\delta(\LB(\zz))=0$, in view of Remark \ref{rem: Strings in hammock WStr}, we get $\LB(\zz)\in\OSt1$ implying that $\yy\in\OSt{-1}$, a contradiction to our assumption. On the other hand, if $\yy=\LB(\zz)$ then $\brac{1}{\LB}(\yy)=\brac{1}{\LB}(\zz)$, a contradiction to the minimality of $|\yy|$.

Therefore there does not exist any string $\zz$ with $\xx_0\sqsubseteq_l\zz\sqsubset_l\yy$ and $\theta(\yy\mid\zz)=1$, which gives $\yy\sqsubseteq_l\mm_i(\xx_0)$.

Finally, since $\yy\sqsubseteq_l\xx$ and $\theta(\brac{1}{\LB}(\yy)\mid\xx)=1$, we have $\LB^n(\yy)=\xx$ by Remark \ref{rem:6.3}. This completes the proof.
\end{proof}

Proposition \ref{structure of beams} and its dual can be used to define a further condensation operator $$C_B:\OST{}\to\OST{\pm1}.$$ The next result shows that $C_\sB$ is compatible with the partition given by Equation \eqref{breaking of OST into beams}.

\begin{prop}\label{OT(beam)}
If $\sB\in\QBa$ is minimal for $(\xx_0,i)$, $[\xx,\yy]$ is a $\sB$-beam and $\zz\in c_\sB([\xx,\yy])$ then $C_\sB(\zz)\in[\xx,\yy]$. 
\end{prop}

\begin{proof}
Without loss assume that $\zz\in\OST1$. Then Proposition \ref{structure of beams} gives some $n\in\N$ such that $\LB^n(C_\sB(\zz))=\zz$. Let $\ww:=C_\sB(\zz)$ for brevity so that $\ww\leq_l\zz$. We claim that $\ww\in[\xx,\yy]$.

Now $\varphi_\sB(\zz)=0$ if and only if $\ww=\zz$. On the other hand, if $\varphi_\sB(\zz)=1$, then assume the claim fails, i.e., assume $\ww<_l\xx<_l\zz$. Then clearly $\ww\sqsubset_l\zz$. Let $\vv:=\ww\sqcap_l\xx$. There are two cases.

If $\vv\sqsubset_l\ww$ then $\ww\sqsubset_l\zz$ gives $\theta(\zz\mid\xx)=\theta(\zz\mid\vv)=-1$, a contradiction to $\xx<_l\zz$.

On the other hand, if $\vv=\ww$ then we have $\ww\sqsubset_l\xx$ and $\ww\sqsubset_l\zz$. Thus $\ww\sqsubset_l\xx\sqcap_l\zz$. If $\xx\sqcap_l\zz=\xx$ then Remark \ref{LB Rem 2} yields $\LB^k(\ww)=\xx$ for some $k\in\N^+$, a contradiction to Proposition \ref{LB Rem 1} as $\xx\in\OST{-1}$. Thus $\xx\sqcap_l\zz\sqsubset_l\xx$. If $\xx\sqcap_l\zz=\zz$ then $\zz\sqsubset_l\xx$ together with $\theta(\xx\mid\zz)=-1$ implies that $\zz\in\OST{-1}$ by Remark \ref{rem: Hammock WStr is closed by substring}, a contradiction to $\varphi_\sB(\zz)=1$. On the other hand, if $\xx\sqcap_l\zz\sqsubset_l\zz$ then $\ww\sqsubset_l\xx\sqcap_l\zz$ together with $\theta(\zz\mid\xx\sqcap_l\zz)=1$ implies $\LB^k(\ww)=\xx\sqcap_l\zz\in\OST{\pm1}$ for some $k\in\N^+$ by Remark \ref{LB Rem 2}, a contradiction to Proposition \ref{LB Rem 1}.

The definition of a $\sB$-beam gives $\ww\in\OST{\pm1}\cap[\xx,\yy]=(\Cent\cup\{\xx,\yy\})\cap[\xx,\yy]$.

Dually, we can show that if $\zz\in\OST{-1}$ we get $\ww\in\OST{\pm1}\cap[\xx,\yy]=(\Cent\cup\{\xx,\yy\})\cap[\xx,\yy]$ such that $\LbB^n(\ww)=\zz$ for some $n\in\N$.
\end{proof}

Using Equation \eqref{breaking of OST into beams} together with Propositions \ref{[mix0,y1] is finite}, \ref{Centers are dense in a beam}, \ref{OT(beam)}, it is possible to compute the order type of the $\sB$-condensation of a $\sB$-beam, and hence that of $\OST{}$.

\begin{cor}\label{wstrordertype}
Suppose $\sB\in\QBa$ is minimal for $(\xx_0,i)$. If $[\xx,\yy]$ is a $\sB$-beam then 
$$(c_\sB([\xx,\yy]),<_l)\cong
    \omega+\Xi(\underbrace{\zeta,\zeta,\cdots,\zeta}_{k_\sB\text{ times}})+\omega^*.$$ As a consequence, $$(c_\sB(H_l^i(\xx_0)),<_l)\cong(\omega+\Xi(\underbrace{\zeta,\zeta,\cdots,\zeta}_{k_\sB\text{ times}})+\omega^*)\cdot\mathbf{n}_\sB.$$
\end{cor}

Recall the definition of $\mathcal I_{(\xx_0,i;\sB)}(\xx)$ for $\xx\in\Cent$. If $\sB$ is minimal for $(\xx_0,i)$, we can extend this definition to all $\xx\in\OST{\pm1}$ as follows:
$$\mathcal I_{(\xx_0,i;\sB)}(\yy_k):=\begin{cases}[\mm_i(\xx_0),\MM_i(\xx_0)]&\text{if }0=k=n_\sB,\\ [\mm_i(\xx_0),\braclB(\yy_0))&\text{if }0=k<n_\sB,\\ (\braclbB(\yy_0),\MM_i(\xx_0)]&\text{if }0<k=n_\sB,\\(\braclbB(\yy_k),\braclB(\yy_k))&\text{if }0<k<n_\sB.\end{cases}$$

Let $\mathbf{c}_\sB:=C_\sB\circ c_\sB$. It is straightforward to verify that $\mathbf{c}_\sB^{-1}(\xx)=\mathcal I_{(\xx_0,i;\sB)}(\xx)$ for each $\xx\in\OST{\pm1}$. Thus 
\begin{equation}\label{ordersumC_B}
H_l^i(\xx_0)\cong\sum_{\xx\in\OST{\pm1}}\mathbf{c}_\sB^{-1}(\xx)\cong\sum_{\xx\in\OST{\pm1}}\mathcal I_{(\xx_0,i;\sB)}(\xx).
\end{equation}

Now we have all the tools necessary for proving the main result of this paper.
\begin{thm}\label{main}
Given a string $\xx_0$ and a parity $i\in\{1,-1\}$, we have $(H_l^i(\xx_0),<_l)\in\dLOfdb{1}{1}$.
\end{thm}
\begin{proof}
The proof is by induction on the size of $\QBa_i(\xx_0)$, which is a finite poset by Proposition \ref{QBA finite poset}.

\textbf{Base Step.} If $\QBa_i(\xx_0)=\emptyset$ then the strings in $H_l^i(\xx_0)$ are band-free relative to $(\xx_0,i)$, implying that $H_l^i(\xx_0)$ is finite in view of Corollary \ref{finitely many band free strings rel to x0,i}. Thus $H_l^i(\xx_0)\in\dLOfdb{1}{1}$.

\textbf{Inductive Step.} Assume that $\QBa_i(\xx_0)\neq\emptyset$ and for any $\xx\in\St$ and $j\in\{1,-1\}$ with $|\QBa_j(\xx)|<|\QBa_i(\xx_0)|$, we have $H_l^j(\xx)\in\dLOfdb{1}{1}$.

Since $\QBa_i(\xx_0)\neq\emptyset$, choose $\sB\in\QBa$ that is minimal for $(\xx_0,i)$. Then Lemma \ref{hammockordersum} gives $$(H_l^i(\xx_0),<_l)\cong\sum_{\xx\in c_\sB(H_l^i(\xx_0))}(H_l^{-\varphi_\sB(\xx)}(\xx),<_l).$$

If $\xx\in c_\sB(H_l^i(\xx_0))\subseteq H_l^i(\xx_0)$ then $\QBa_{-\varphi_\sB(\xx)}(\xx)\subseteq\QBa_i(\xx_0)$. Moreover, for any such $\xx$, we have $\sB\notin\QBa_{-\varphi_\sB(\xx)}(\xx)$ thanks to Remark \ref{hammockcoverrepr}, and hence $\QBa_{-\varphi_\sB(\xx)}(\xx)\subsetneq\QBa_i(\xx_0)$. Therefore by the induction hypothesis, we get $H_l^{-\varphi_\sB(\xx)}(\xx)\in\dLOfdb{1}{1}$ for each $\xx\in c_\sB(H_l^i(\xx_0))$.

In view of Equation \eqref{breaking of OST into beams}, Proposition \ref{hammockcoverclassesorder} and the fact that $H_l^{-\varphi_\sB(\yy_p)}(\yy_p)=H_l^0(\yy_p)=\{\yy_p\}$ for each $0\leq p\leq n_\sB$, we can write $$H_l^i(\xx_0)=\widetilde{L}_0\dotplus\widetilde{L}_1\dotplus\cdots\dotplus\widetilde{L}_{n_{\sB}+1},$$ where $$\widetilde{L}_k:=\begin{cases}\sum_{\xx\in[\mm_i(\xx_0),\yy_0]}(H_l^{-\varphi_\sB(\xx)}(\xx),<_l) &\text{if }k=0,\\\sum_{\xx\in[\yy_{k-1},\yy_k]}(H_l^{-\varphi_\sB(\xx)}(\xx),<_l)&\text{if }1\leq k\leq n_\sB,\\\sum_{\xx\in[\yy_{n_\sB},\MM_i(\xx_0)]}(H_l^{-\varphi_\sB(\xx)}(\xx),<_l)&\text{if } k=n_\sB+1.\end{cases}$$

Since a finite order sum of linear orders in $\dLOfdb{1}{1}$ lies in $\dLOfdb{1}{1}$, using the induction hypothesis and Proposition \ref{[mix0,y1] is finite}, we see that $\widetilde{L}_0,\widetilde{L}_{n_\sB+1}\in\dLOfdb{1}{1}$. We will use Lemma \ref{recdlofdb} to show that $\widetilde{L}_k\in\dLOfdb{1}{1}$ for $1\leq k\leq n_\sB$.

Proposition \ref{OT(beam)} showed $c_\sB([\yy_0,\yy_1])=[\yy_0,\yy_1]\cap\OST{}\cong(\omega+\Xi(\underbrace{\zeta,\zeta,\cdots,\zeta}_{k_\sB\text{ times}})+\omega^*)$. Its proof together with Equation \eqref{ordersumC_B} helps us to write $$\widetilde{L}_1=\overline L_1+\overline L_2+\overline L_3,$$
where
\begin{eqnarray*}
\overline L_1&:=&\mathcal I_{(\xx_0,i;\sB)}(\yy_0)\cap[\yy_0,\yy_1]=\sum_{n\in\omega}H_l^{-\varphi_\sB(\LB^n(\yy_0))}(\LB^n(\yy_0)),\\\overline{L}_2&:=&\sum_{\xx\in\Cent\cap[\yy_0,\yy_1]}\mathcal I_{(\xx_0,i;\sB)}(\xx)\\&=&\sum_{\xx\in\Cent\cap[\yy_0,\yy_1]}\left(\sum_{n\in\omega^*,\,n\neq0}H_l^{-\varphi_\sB(\LbB^n(\xx))}(\LbB^n(\xx))+\sum_{n\in\omega}H_l^{-\varphi_\sB(\LB^n(\xx))}(\LB^n(\xx))\right),\\\overline L_3&:=&\mathcal I_{(\xx_0,i;\sB)}(\yy_1)\cap[\yy_0,\yy_1]=\sum_{n\in\omega^*}H_l^{-\varphi_\sB(\LbB^n(\yy_1))}(\LbB^n(\yy_1)).\end{eqnarray*}

Let $\{\xx_1,\cdots,\xx_{k_\sB}\}$ be a set of representatives of distinct $\sB$-equivalence classes of $\Cent\cap[\yy_0,\yy_1]$. Proposition \ref{shuffle structure of the middle} states that the order type of the interval $\mathcal I_{(\xx_0,i;\sB)}(\xx_j)$ in $H_l^i(\xx_0)$ is independent of the choice of the representative $\xx_j$ for each $j\in\{1,\cdots,k_\sB\}$. Thus if $L_j$ denotes the order type of $\mathcal I_{(\xx_0,i;\sB)}(\xx_j)$, then $$\overline{L}_2\cong\Xi(L_1,\cdots,L_{k_\sB}).$$

For any string $\zz\in\ c_\sB([\yy_0,\yy_1])$ with $\varphi_\sB(\zz)\neq -1$, Proposition \ref{OST to infty bux} ensures the existence of a string $\uu$ and $\bb\in\BalB$ such that $\brac{1}{\LB}(\zz)=\,^\infty\bb\uu\zz$. Since $\theta(\bb)=1$, Remark \ref{LB Rem 2} yields $s'\in\N$ and $p\in\N^+$ such that $\LB^{s'}(\zz)=\bb\uu\zz$ and $\LB^{s'+p}(\zz)=\bb^2\uu\zz=\bb\LB^{s'}(\zz)$. Since $\delta(\bb)=0$, we have $\LB^{s'+p}(\zz)\equiv_H\LB^{s'}(\zz)$. Since $\bb\in\StB1$, Remark \ref{rem:6.3} and Proposition \ref{LB is ext of lB} together imply $\LB^{s'+p+k}(\zz)\equiv_H\LB^{s'+k}(\zz)$ for each $k\in\N$. Since $\varphi_\sB(\LB^q(\zz))=1$ for each $q\in\N^+$, we get $H_l^{-\varphi_\sB(\LB^{s'+p+k}(\zz))}(\LB^{s'+p+k}(\zz))\cong H_l^{-\varphi_\sB(\LB^{s'+k}(\zz))}(\LB^{s'+k}(\zz))$. A dual result can be shown for $\zz\in\ c_\sB([\yy_0,\yy_1])$ with $\varphi_\sB(\zz)\neq 1$. Thus we have shown that all the hypotheses of Lemma \ref{recdlofdb} are satisfied, and hence we get that $\widetilde{L}_1=\overline L_1+\overline L_2+\overline L_3\in\dLOfdb{1}{1}$. A similar argument shows that $\widetilde{L}_k\in\dLOfdb{1}{1}$ for each $1\leq k\leq n_\sB$, and this completes the proof.
\end{proof}

\begin{exmp}
We compute the order type of $H_l^1(a_0)$ from Example \ref{exmp: QBa in running example}. Continuing from Example \ref{exmp:beams for running example}, recall that $\sB_1$ is minimal and domestic for $(a_0,1)$. By domesticity of $\sB_1$, we have $k_{\sB_1}=0$, and thus by Corollary \ref{wstrordertype} we have
\begin{equation}
    c_{\sB_1}([a_0,a_3A_1a_0])=\sum_{k\in\omega}\{l^k_{\sB_1}(a_0)\}+\sum_{k\in\omega^*}\{\lb_{\sB_1}^k(a_3A_1a_0)\}\cong \omega+\omega^*.
\end{equation}
Using Lemma \ref{hammockordersum}, we obtain
\begin{equation}\label{computation: B1 beam}
[a_0,a_3A_1a_0]=H_l^0(a_0)+\sum_{k\in\omega,\ k\neq0}H_l^{-1}(l^k_{\sB_1}(a_0))+\sum_{k\in\omega^*,\ k\neq0} H_l^1(\lb_{\sB_1}^k(a_3A_1a_0))+H_l^0(a_3A_1a_0).
\end{equation}
Note that $l^{2k+r}_{\sB_1}(a_0)\equiv_Hl^{2+r}_{\sB_1}(a_0)$ for every $k\geq 0$ and $0\leq r\leq1$ with $(k,r)\neq(0,0)$. Moreover, $H_l^1(\lb_{\sB_1}^k(a_3A_1a_0))=\{\lb_{\sB_1}^k(a_3A_1a_0)\}$ for every $k\in\omega^*,\ k\neq0$, and $H_l^{-1}(l_{\sB_1}(a_0))=\{l_{\sB_1}(a_0)\}$.  Plugging these in Equation \eqref{computation: B1 beam}, we get
\begin{equation}
[a_0,a_3A_1a_0]\cong\mathbf1+(\mathbf1+H_l^{-1}(l^2_{\sB_1}(a_0)))\cdot\omega+\omega^*.
\end{equation}
To compute $H_l^{-1}(l^2_{\sB_1}(a_0))$, note that non-domestic $\sB_2$ is the unique element of $\QBa_{-1}(l^2_{\sB_1}(a_0))$, and thus minimal for $(l^2_{\sB_1}(a_0),-1)$. Since $k_{\sB_2}=1$, applying Lemma \ref{hammockordersum} to Corollary \ref{wstrordertype} with the help of the base case of the proof of Theorem \ref{main}, we get $H_l^{-1}(l^2_{\sB_1}(a_0))\cong\omega+\Xi(\zeta)+\omega^*$, so that Equation \eqref{computation: B1 beam} takes the form
\begin{equation}\label{computation: a}
[a_0,a_3A_1a_0]\cong\mathbf1+(\mathbf1+\omega+\Xi(\zeta)+\omega^*)\cdot\omega+\omega^*\cong(\omega+\Xi(\zeta)+\omega^*)\cdot\omega+\omega^*.
\end{equation}
Similarly we can obtain 
\begin{equation}\label{computation: b}
[a_3A_1a_0,A_1a_0]\cong(\omega+\Xi(\zeta)+\omega^*)\cdot\omega+\omega^*.
\end{equation}
Again applying Lemma \ref{hammockordersum} to the last term of the right-hand side of Equation \eqref{computation: condensation OST}, we obtain
\begin{equation}\label{computation: last finite interval}
\begin{split}
[A_1a_0,H_1G_1FE_2E_1A_2A_1a_0]=\{A_1a_0\}+\{A_2A_1a_0\}+\{E_1A_2A_1a_0\}+H_l^{-1}(E_2E_1A_2A_1a_0)\\+\{FE_2E_1A_2A_1a_0\}+H_l^{-1}(G_1FE_2E_1A_2A_1a_0)+H_l^{-1}(H_1G_1FE_2E_1A_2A_1a_0).
\end{split}
\end{equation}
To compute $H_l^{-1}(E_2E_1A_2A_1a_0)$, note that $\sB_3$ is non-domestic and minimal for $(E_2E_1A_2A_1a_0,-1)$. Recall from Example \ref{exmp: Bequiv classes} that $k_{\sB_3}=3$. It is easy to verify that
\begin{equation}\label{computation: I(Bcenter)=zeta}
\begin{split}
\mathcal{I}_{(E_2E_1A_2A_1a_0,-1;\sB_3)}(G_1FE_2E_1A_2A_1a_0)&\cong\mathcal{I}_{(E_2E_1A_2A_1a_0,-1;\sB_3)}(k_1h_2H_1G_1FE_2E_1A_2A_1a_0)\cong\zeta,\\
\mathcal{I}_{(E_2E_1A_2A_1a_0,-1;\sB_3)}(E_2E_1e_3E_2E_1A_2A_1a_0)&\cong(\omega^*+(\omega+\omega^*)\cdot\omega).
\end{split}
\end{equation}
Using Equations \eqref{computation: I(Bcenter)=zeta}, Corollary \ref{wstrordertype}, and Lemma \ref{hammockordersum}, we obtain
\begin{equation}\label{computation: 1}
\begin{split}
H_l^{-1}(E_2E_1A_2A_1a_0)&\cong\omega+\Xi(\zeta,\zeta,\omega^*+(\omega+\omega^*)\cdot\omega)+\omega^*,\\
H_l^{-1}(G_1FE_2E_1A_2A_1a_0)&\cong(\omega+\omega^*)\cdot\omega+\Xi(\zeta,\zeta,\omega^*+(\omega+\omega^*)\cdot\omega)+\omega^*,\\
H_l^{-1}(H_1G_1FE_2E_1A_2A_1a_0)&\cong\omega+\Xi(\zeta,\zeta,\omega^*+(\omega+\omega^*)\cdot\omega)+\omega^*.
\end{split}
\end{equation}
Plugging Equations \eqref{computation: 1} in Equation \eqref{computation: last finite interval} while using Equation \eqref{shuffle property finite} we get
\begin{equation}\label{computation: c}
\begin{split}
[A_1a_0,H_1G_1FE_2E_1A_2A_1a_0]&\cong\mathbf3+\omega+\Xi(\zeta,\zeta,\omega^*+(\omega+\omega^*)\cdot\omega)+\omega^*\\&\ \ \ +\mathbf1+(\omega+\omega^*)\cdot\omega+\Xi(\zeta,\zeta,\omega^*+(\omega+\omega^*)\cdot\omega)+\omega^*\\&\
 \ \ +\omega+\Xi(\zeta,\zeta,\omega^*+(\omega+\omega^*)\cdot\omega)+\omega^*\\&\cong\omega+\Xi(\zeta,\zeta,\omega^*+(\omega+\omega^*)\cdot\omega)+\omega^*.
\end{split}
\end{equation}
Plugging Equations \eqref{computation: a}, \eqref{computation: b} and \eqref{computation: c} in Equation \eqref{computation: hammock}, we obtain
\begin{equation}
H_l^1(a_0)\cong((\omega+\Xi(\zeta)+\omega^*)\cdot\omega+\omega^*)\cdot\mathbf2+\omega+\Xi(\zeta,\zeta,\omega^*+(\omega+\omega^*)\cdot\omega)+\omega^*.
\end{equation}
\end{exmp}

Since Theorem \ref{main} generalizes \cite[Corollary~10.14]{SardarKuberHamforDom} and the latter has a converse in \cite[Proposition~10.18]{SardarKuberHamforDom} for linear orders in $\dLOfpb{1}{1}$, it is natural to ask if the converse to the former is true. Proposition \ref{mainpartialconverse} proves a special case of the converse for which the next result is essential.
\begin{prop}
\label{prop: partition of dlofp}
    Suppose $L(\neq\mathbf 0)\in\dLOfpb{}{}$.
    \begin{itemize}
        \item If $L\in\dLOfpb{0}{1}$ there exist $L_1\in \dLOfpb{1}{1}$ and $L_2\in\dLOfpb{1}{1}\cup\{\mathbf{0}\}$ such that $L \cong L_1 \cdot \omega^* + L_2$;
        \item If $L\in\dLOfpb{1}{0}$ then there exist $L_2 \in \dLOfpb{1}{1}$ and $L_1\in\dLOfpb{1}{1}\cup\{\mathbf{0}\}$ such that $L \cong L_1 + L_2 \cdot \omega$;
        \item If $L \in \dLOfpb{0}{0}$, then there exist $L_1,L_3 \in \dLOfpb{1}{1}$ and $L_2\in\dLOfpb{1}{1}\cup\{\mathbf{0}\}$ such that $L \cong L_1 \cdot \omega^* + L_2 + L_3 \cdot \omega$.
    \end{itemize}
\end{prop}
\begin{proof}
We use the notations and results from \cite{AKS} to prove the first result; the proofs of the rest are similar.

Recall from \cite[Proposition~5.6]{AKS} that for any $L\in\LOfp$ there is $(T,s_T)\in\mathrm{3}\mathbf{ST}_\omega$ such that $\mathrm{LIN}(T,s_T)\cong L$. Suppose for $n\in\N^+$ the notation $0^n$ denotes $\underbrace{00\cdots0}_{n \mbox{ times}}$. It is easy to note that $L$ has a minimum if and only if whenever $0^n\in T$ for some $n\in\N ^+$ then $s_T(0^n)\neq-$.

Now if $L\in\dLOfpb{0}{1}$ then choose the least $N(T)\in\N^+$ such that $0^{N(T)}\in T$ and $s_T(0^{N(T)})=-$. We use induction on ${N(T)}$ to obtain required $L_1$ and $L_2$.

\noindent{\textbf{Base step}} (${N(T)}=1$): Let $w$ be the width of $(T,s_T)$. Then $L\cong L'_1\cdot\omega^*+L'_2$, where $L'_1:=\mathrm{LIN}(\widehat T_0,s_{\widehat T_0})$ and $L'_2:=\sum_{1\leq k<w}\mathrm{LIN}(T_k,s_{T_k})$. If $L'_2\neq\mathbf{0}$ then it has a maximum.

If $L'_1$ has a maximum then by discreteness of $L$, it also has a minimum. Furthermore, if $L'_2\neq\mathbf{0}$ then it also has a minimum. Thus irrespective of whether $L'_2=\mathbf 0$ or not, we can choose $L_1:=L'_1$ and $L_2:=L'_2$.

On the other hand, if $L'_2\neq\mathbf{0}$ and $L'_1$ does not have a maximum then let $x\in L'_1$ be any element. Since $L'_1\cdot\omega^*$ is discrete, $x$ has an immediate successor, say $y$. Thus we can write $L'_1=L'_{11}+L'_{12}$, where $x$ is the maximum of $L'_{11}$ and $y$ is the minimum of $L'_{12}$. Then $L\cong L'_1\cdot\omega^*+L_2\cong (L'_{12}+L'_{11})\cdot\omega^*+(L'_{12}+L'_2)$, so that $L_1:=L'_{12}+L'_{11}$ and $L_2:=L'_{12}+L'_2$ are as required.

\noindent{\textbf{Inductive step}} (${N(T)}>1$): Here $s_T(0^{N(T)-1})=+$. Note that $N(\mathsf{EXUDE}((T,s_T);0^{N(T)-1}))<N(T)$. Moreover, recall from \cite[Proposition~6.5]{AKS} that $\mathrm{LIN}(\mathsf{EXUDE}((T,s_T);0^{N(T)-1}))\cong L$. Thus the induction hypothesis applied to $\mathsf{EXUDE}((T,s_T);0^{N(T)-1})$ produces the required orders $L_1$ and $L_2$.
\end{proof}
\begin{prop}\label{mainpartialconverse}
    If $L_0\in\dLOfpb{1}{0},L_1\in\dLOfpb{0}{0}$ and $L_2\in\dLOfpb{0}{1}$ then there is a non-domestic string algebra $\Lambda$, a string $\xx_0$ for $\Lambda$ and a parity $i\in\{1,-1\}$ such that $(H_l^i(\xx_0),<_l)\cong L_0+\Xi(L_1)+L_2$.
\end{prop}
\begin{proof}

Proposition \ref{prop: partition of dlofp} yields $L_{00} , L_{01}, L_{10}, L_{11}, L_{12}, L_{20}, L_{21} \in \dLOfpb{1}{1}$ such that
$$L_0 = L_{00} + L_{01} \cdot \omega,\ L_1 = L_{11}\cdot \omega^* + L_{10} + L_{12} \cdot \omega,\text{ and }L_2=L_{21}\cdot\omega^*+L_{20}.$$

\begin{figure}[h]\label{fig: counter example}
\[\begin{tikzcd}
                   &                                    &                       &                                           & v_5 \arrow[r, "m_1"]                    & Q_{01}                   &                                       &                                        & Q_{10}                                  & Q_{11} \arrow[d, "x_2"]                \\
v_0 \arrow[d, "a"] &                                    &                       & v_4 \arrow[ru, "j_1"] \arrow[rr, "k_1"]   &                                         & v_6 \arrow[lu, "l_1"']   &                                       & v_{12} \arrow[r, "q"] \arrow[rd, "t"'] & v_{13} \arrow[u, "x_1"]                 & v_{14} \arrow[l, "r"'] \arrow[ld, "s"] \\
v_1 \arrow[d, "d"] & v_2 \arrow[l, "b"'] \arrow[r, "c"] & v_3 \arrow[ru, "f_1"] &                                           &                                         &                          & v_7 \arrow[lu, "n_1"] \arrow[ru, "p"] &                                        & v_{15} \arrow[ld, "v"'] \arrow[rd, "w"] &                                        \\
Q_{00}             & Q_{20} \arrow[u, "e"']             &                       & v_{10} \arrow[lu, "f_2"] \arrow[r, "h_2"] & v_9 \arrow[r, "k_2"] \arrow[rd, "j_2"'] & v_8 \arrow[ru, "n_2"]    &                                       & v_{16} \arrow[r, "z"']                 & v_{17} \arrow[d, "x_3"]                 & v_{18} \arrow[l, "y"]                  \\
                   &                                    &                       &                                           & Q_{21} \arrow[r, "m_2"']                & v_{11} \arrow[u, "l_2"'] &                                       &                                        & Q_{12}                                  &                                       
\end{tikzcd}\]
\caption{$\Gamma'''$ with $\rho = \{ db,ce,f_1f_2, j_1f_1, l_1n_1, n_1n_2, m_1j_1, k_2h_2, l_2m_2, n_2k_2, tp, x_1 r, sx_2, wt, vs, x_3 z \}$}
\label{fig:counter example}
\end{figure}
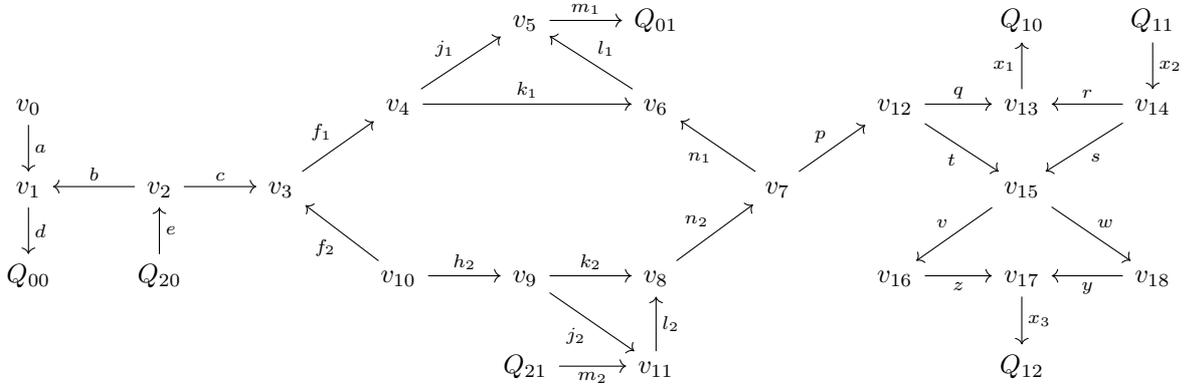

Consider the string algebra $\Gamma'''$ from Figure \ref{fig:counter example}, where \cite[Proposition~10.18]{SardarKuberHamforDom} allows us to choose quivers with relations $(Q_{00},\rho_{00})$, $(Q_{01},\rho_{01})$, $(Q_{20},\rho_{20})$, $(Q_{21},\rho_{21})$, $(Q_{10},\rho_{10})$, $(Q_{10},\rho_{10})$, $(Q_{11},\rho_{11})$ and $(Q_{12},\rho_{12})$ for domestic gentle algebras such that
\begin{equation}
\begin{split}
H_l^{-1}(a)\cong L_{00}&,\ H_l^{-1}(l_1)\cong L_{01},\\
H_l^{-1}(q)\cong L_{10},\ H_l^1(R)&\cong L_{11},\ H_l^{-1}(y)\cong L_{12},\\
H_l^1(B)\cong L_{20},&\text{ and } H_l^1(j_2)\cong L_{21}.
\end{split}
\end{equation}
Since the algebras presented by $(Q_{mn},\rho_{mn})$ are gentle and $\rho$ for $\Gamma'''$ consists only of paths of length $2$, we conclude that $\Gamma'''$ is a gentle algebra.

We will show for an appropriate $j\in\{-1,1\}$ that $H_l^{-1}(1_{(v_0,j)})\cong L_0+\Xi(L_1)+L_2$.

Since $\Gamma'''$ is gentle, for any string $\uu\in H_l^{-1}(1_{(v_0,j)})$ with $|\uu|>0$, we have $\uu\equiv_H\alpha$, where $\alpha$ is the last syllable of $\uu$. Consequently, $H_l(\uu)\cong H_l(\alpha)$.

We have that $\sB_1=\{l_1k_1J_1\}$ and $\sB_2=\{k_2J_2L_2\}$ are minimal for $(1_{(v_0,j)},-1)$. Choosing $\sB_1$ as minimal for $(1_{(v_0,j)},-1)$, we have
\begin{equation}\label{cec: hammock}
H_l^{-1}(1_{(v_0,j)})=[\mm_{-1}(1_{(v_0,j)}),a]\dotplus[a,(1_{(v_0,j)})]\cong H_l^{-1}(a)\dotplus[a,(1_{(v_0,j)})]\cong L_{00}\dotplus[a,(1_{(v_0,j)})].
\end{equation}
Since $\sB_1$ is domestic, using Corollary \ref{wstrordertype}, we have 
\begin{equation}
c_{\sB_1}([a,1_{(v_0,j)}])=\sum_{k\in\omega}\{l^k_{\sB_1}(a)\}+\sum_{k\in\omega^*}\{\lb^k_{\sB_1}(1_{v_0,j})\}\cong\omega+\omega^*.
\end{equation}
Applying Lemma \ref{hammockordersum} to the above equation,
\begin{equation}
[a,1_{(v_0,j)}]\cong\sum_{k\in\omega,k\neq0}H_l^{-1}(l^k_{\sB_1}(a))+\sum_{k\in\omega^*,k\neq0}H_l^1(\lb^k_{\sB_1}(1_{v_0,j})).
\end{equation}
Using appropriate $H$-equivalences, we obtain
\begin{equation}\label{cec: beam}
\begin{split}
[a,1_{(v_0,j)}]&\cong\{j_1f_1cBa\}+H_l^{-1}(l_1)\cdot\omega+H_l^1(k_1)\cdot\omega^*+H_l^1(c)+H_l^1(B)+\{a\}\\&\cong L_{01}\cdot\omega+H_l(N_1)\cdot\omega^*+H_l^1(c)+L_{20}+\mathbf1.
\end{split}
\end{equation}
To compute $H_l^1(c)$, note that $\sB_2$ is minimal for $(c,1)$. Using appropriate $H$-equivalences, we have 
\begin{equation}\label{cec:subbeam}
H_l^1(c)\cong \{cBa\}+H_l^{-1}(l_2)\cdot\omega+H_l^1(j_2)\cdot\omega^*\cong \mathbf1+H_l(n_2)\cdot\omega+L_{21}\cdot\omega^*.
\end{equation}
Note that $N_1\equiv_H n_1$ and therefore $H_l(N_1)\cong H_l(n_2)\cong H_l^{-1}(N_1)$, which we now compute.

Let $\sB\in\QBa_{-1}(1_{(v_0,j)})$ be the only non-domestic element with $\BalB=\{ywVZ\}$ and $\BalbB=\{qTsR\}$. Since $\sB\in\QBa$ is minimal for $(N_1,-1)$ and $k_{\sB}=1$ with $\zz:=sRqpN_1$ as a representative for the unique $\sB$-equivalence class, it is easy to verify that
\begin{equation}
\mathcal{I}_{(N_1,-1;\sB)}(\zz)\cong H_l^1(R)\cdot\omega^*+H_l^{-1}(q)+H_l^{-1}(y)\cdot\omega\cong L_{11}\cdot\omega^*+L_{10}+L_{12}\cdot\omega.
\end{equation}
Using Corollary \ref{wstrordertype} and Lemma \ref{hammockordersum}, we get
\begin{equation}\label{cec: subsubbeam}
\begin{split}
H_l(N_1)&\cong H_l^{-1}(q)+H_l^{-1}(y)\cdot\omega+\Xi(\mathcal{I}_{(N_1,-1;\sB)}(\zz))+H_l^1(r)\cdot\omega^*\\&\cong L_{10}+L_{12}\cdot\omega+\Xi(L_{11}\cdot\omega^*+L_{10}+L_{12}\cdot\omega)+L_{11}\cdot\omega^*.
\end{split}
\end{equation}
Using the isomorphism $H_l(N_1)\cong H_l(n_1)$ and Equation \ref{shuffle property infinite} while plugging Equations \eqref{cec: beam}, \eqref{cec:subbeam} and \eqref{cec: subsubbeam} in Equation \eqref{cec: hammock}, we get
\begin{equation}
\begin{split}
H_l(1_{(v_0,j)})&\cong L_{00}+L_{01}\cdot\omega+H_l(N_1)\cdot\omega^*+\mathbf1+H_l(n_2)\cdot\omega+L_{21}\cdot\omega^*+L_{20}+\mathbf1
\\&\cong L_{00}+L_{01}\cdot\omega+H_l(N_1)\cdot(\omega^*+\omega)+L_{21}\cdot\omega^*+L_{20}
\\&\cong L_0+(L_{10} + L_{12}\cdot \omega + \Xi( L_{11}\cdot \omega^* + L_{10} + L_{12}\cdot \omega) + L_{11}\cdot \omega^*)\cdot(\omega^*+\omega)+L_2
\\&\cong L_0+\Xi( L_{11}\cdot \omega^* + L_{10} + L_{12}\cdot \omega)+L_2
\\&\cong L_0+\Xi(L_1)+L_2.
\end{split}
\end{equation}
This completes the proof.
\end{proof}

\begin{que}\label{mainconversefail}
Does there exist a non-domestic string algebra $\Lambda$, a string $\xx_0$ for $\Lambda$ and a parity $i\in\{1,-1\}$ such that $(H_l^i(\xx_0),<_l)\cong\omega+\Xi((\omega+\omega^*)\cdot\omega^*+\omega,\omega^*+(\omega+\omega^*)\cdot\omega)+\omega^*$?
\end{que}
We believe that the answer to the above question is negative; however, currently, we do not have any methods to show this. In fact, it is an interesting problem to determine the subclass of $\dLOfdb{1}{1}$ which consists only of the order types of hammocks for string algebras.

\section{Locating left $\N$-strings in the completion of a hammock}\label{sec:completionhammock}
This last section is devoted to computing the order type of the completion of the $\sB$-condensation of a hammock and describing the location of some left $\N$-strings therein. As a consequence, we characterize (Proposition \ref{position of infinite strings in OST}) some almost periodic left $\N$-strings in terms of scattered subintervals of the condensations.

Given $\sB\in\QBa$ we identify two subsets of $\StN$ associated to $\sB$.
$$\EStB:=\{\xx\in\StN\mid\text{all but finitely many left substrings of }\xx\text{ are in }\StB{}\}.$$$$\EOSt:=\{\xx\in\StN\mid\text{all proper left substrings are in }\OSt{}\}.$$Further set $\ESTB:=\EStB\cap\HHlix$ and $\EOST{}:=\EOSt\cap\HHlix$.

The following remarks are straightforward.
\begin{rem}\label{EOSTsubsetESTB and equal iff minimal}
Note that $\EStB\subseteq\EOSt$, and therefore $\ESTB\subseteq\EOST$. Moreover, it follows from Corollary \ref{OST-STB is finite} that $\ESTB=\EOST$ if and only if $\sB$ is minimal for $(\xx_0,i)$.
\end{rem}
\begin{rem}\label{EST(x,j,B)subsetofEST(x0,i,B)}
If $\xx\in H_l^i(\xx_0)\setminus\{\xx_0\}$ and $\sB\in\QBa_j(\xx)$ then for any $j\in\{-1,1\}$, we have $\EST{\xx}{j}{\sB}\subseteq\ESTB$ and $\EOsT{\xx}{j}{\sB}\subseteq\EOST$.
\end{rem}

Now we show that every left $\N$-string in $\HHlix$ lies in $\ESTB$ for some $\sB\in\QBa_i(\xx_0)$.

\begin{prop}\label{eventually left N strings are in OST minimal}
Given $\xx\in\HHlix\setminus H_l^i(\xx_0)$, there exists $\zz\in H_l^i(\xx_0)$ with $\zz\sqsubset_l\xx$ and $\sB\in \QBa_i(\xx_0)$ minimal for $(\zz,\theta(\xx\mid\zz))$ such that $\xx\in\N\mbox-\mathsf{St}(\zz,\theta(\xx\mid\zz);\sB)$.
\end{prop}
\begin{proof}
Since $\xx_0\sqsubset_l\xx$ and $\xx$ is a left $\N$-string, there are infinitely many strings $\vv$ such that $\vv\xx_0$ is a string. Therefore by the observation in the base case of the proof of Theorem \ref{main}, we have $\QBa_i(\xx_0)\neq\emptyset$.

It is trivial to note that if $\vv\sqsubset_l\vv'$ then $\QBa_j(\vv')\subseteq\QBa_{\theta(\vv'\mid\vv)}(\vv)$ for each $j\in\{-1,1\}$. For each $n\in\N$, let $\xx_n\sqsubset_l\xx$ satisfy $|\xx_n|=|\xx_0|+n$. Therefore the sequence of sets $\mathcal{B}_n:=\QBa_{\theta(\xx\mid\xx_n)}(\xx_n)$ satisfies $\mathcal B_{n+1}\subseteq\mathcal B_n$ for every $n\in\N$. Since $\mathcal B_0$ is finite, thanks to Proposition \ref{QBA finite poset}, we get that there is $N\in\N$ such that $\mathcal{B}_n=\mathcal{B}_N$ for every $n\geq N$. Choose a minimal element $\sB$ of $\mathcal B_N$ with respect to $\preceq$. This implies that $\sB$ is minimal for $(\xx_N,\theta(\xx\mid\xx_N))$ and $\xx_n\in\overline{\mathsf{St}}(\xx_N,\theta(\xx\mid\xx_N);\sB)$ for every $n\geq N$. Thus $\xx\in\EOsT{\xx_N}{\theta(\xx\mid\xx_N)}{\sB}$. Finally, since $\sB$ is minimal for $(\xx_N,\theta(\xx\mid\xx_N))$, it follows from Remark \ref{EOSTsubsetESTB and equal iff minimal} that $\xx\in\EST{\xx_N}{\theta(\xx\mid\xx_N)}{\sB}$.
\end{proof}

Recall from Proposition \ref{completion of hammock} that $\gap(H_l^i(\xx_0))\cong\HHlix\setminus H_l^i(\xx_0)$. Proposition \ref{eventually left N strings are in OST minimal} together with Remark \ref{EST(x,j,B)subsetofEST(x0,i,B)} allows us to describe the latter set in two possible ways.
\begin{equation}\label{expression 1 of left N strings}
    \HHlix\setminus H_l^i(\xx_0)=\bigcup_{\substack{\xx\in H_l^i(\xx_0)\setminus\{\xx_0\},\ j\in\{-1,1\},\\\sB\text{ is minimal for $(\xx,j)$}\\j=i\text{ if }\xx=\xx_0}}\EST{\xx}{j}{\sB},
\end{equation}
\begin{equation}\label{expression 2 of left N strings}
    \HHlix\setminus H_l^i(\xx_0)=\bigsqcup_{\sB\in\QBa_i(\xx_0)}\ESTB.
\end{equation}
Equation \eqref{expression 1 of left N strings} will be used to locate the position of a left $\N$-string in the extended hammock $\HHlix$, whereas Equation \eqref{expression 2 of left N strings} is a generalization of \cite[Proposition~2]{ringalgcom95} which states that every left $\N$-string in a domestic string algebra is almost periodic.

As a consequence of Proposition \ref{eventually left N strings are in OST minimal}, we show that each gap in the hammock $H_l^i(\xx_0)$ corresponds to a gap in $\OsT{\xx}{j}{\sB}$ for some $\xx\in H_l^i(\xx_0)$, a parity $j\in\{-1,1\}$ and $\sB\in\QBa$ minimal for $(\xx,j)$.
\begin{prop}\label{gaps of hammocks are gaps of OST}
If $(X,Y)$ is a gap in $H_l^i(\xx_0)$ then there is a string $\xx\in H_l^i(\xx_0)$, a parity $j\in\{-1,1\}$ and $B\in\QBa$ minimal for $(\xx,j)$ such that $(X\cap\overline{\mathsf{St}}(\xx,j;\sB),Y\cap\overline{\mathsf{St}}(\xx,j;\sB))$ is a gap in $\overline{\mathsf{St}}(\xx,j;\sB)$.
\end{prop}
\begin{proof}
The gap $(X,Y)$ in $H_l^i(\xx_0)$ corresponds to a unique $\yy\in\HHlix\setminus H_l^i(\xx_0)$ by Proposition \ref{completion of hammock}. Further Proposition \ref{eventually left N strings are in OST minimal} yields $\xx\in H_l^i(\xx_0)$ such that $\xx\sqsubset_l\yy$ and $\sB$ minimal for $(\xx,\theta(\yy\mid\xx))$ and $\xx\in\EST{\xx}{\theta(\yy\mid\xx)}{\sB}$.

Let $\yy=:\zz\xx$ and $j:=\theta(\yy\mid\xx)$. Since the set $\{\vv\in\St:\delta(\vv)\neq0\}$ is finite, there are infinitely many inverse as well as direct syllables in $\zz$. It is easy to see that the set $\{\vv\xx\in\St:\vv\xx\sqsubset_l\yy,\theta(\yy\mid\vv\xx)=1\}$ is a cofinal subset of $X\cap\overline{\mathsf{St}}(\xx,j;\sB)$ having no maximum element and the set $\{\vv\xx\in\St:\vv\xx\sqsubset_l\yy,\theta(\yy\mid\vv\xx)=-1\}$ is a coinitial subset of $Y\cap\overline{\mathsf{St}}(\xx,j;\sB)$ having no minimum element. Therefore we conclude that $(X\cap\overline{\mathsf{St}}(\xx,j;\sB),Y\cap\overline{\mathsf{St}}(\xx,j;\sB))$ is a gap in $\overline{\mathsf{St}}(\xx,j;\sB)$.
\end{proof}

Conversely we show that each gap of $\OsT{\xx}{j}{\sB}$ corresponds to a gap in $H_l^i(\xx_0)$, where $\xx\in H_l^i(\xx_0)$, $j\in\{-1,1\}$ and $\sB\in\QBa_j(\xx)$, with the restriction that $j=i$ if $\xx=\xx_0$. Note that we do not require $\sB$ to be minimal for $(\xx,j)$.
\begin{prop}\label{gaps of OST are gaps of hammocks}
Let $\xx\in H_l^i(\xx_0)$, $j\in\{-1,1\}$ and $\sB\in \QBa_i(\xx_0)$. If $(X,Y)$ is a gap in $\overline{\mathsf{St}}(\xx,j;\sB)$ then there exists a unique $X'\supseteq X$ and a unique $Y'\supseteq Y$ such that $(X',Y')$ is a gap in $H_l^i(\xx_0)$.
\end{prop}
\begin{proof}
In view of Proposition \ref{completion of hammock}, it suffices to show that there is a unique left $\N$-string $\yy$ such that $X\subseteq\{\vv\in H_l^i(\xx_0)\mid\vv<_l\yy\}$ and $Y\subseteq\{\vv\in H_l^i(\xx_0)\mid\yy<_l\vv\}$.

The technique to get a left $\N$-string $\yy$ by ``filling up'' the gap $(X,Y)$ in $\OST{}$ is similar to the proof of the converse part of Proposition \ref{completion of hammock}, keeping in mind that $\overline{\mathsf{St}}(\xx,j;\sB)$ is closed under substrings in $H_l^j(\xx)$, thanks to Remark \ref{rem: Hammock WStr is closed by substring}. The construction of $\yy$ ensures that $\yy\in\EOsT{\xx}{j}{\sB}$.

If $\yy_1$ and $\yy_2$ are two distinct left $\N$-strings in $\EOsT{\xx}{j}{\sB}$ then the string $\yy_3:=\yy_1\sqcap_l\yy_2$ lies in $\overline{\mathsf{St}}(\xx,j;\sB)$ and between $\yy_1$ and $\yy_2$ in $(\HHlix,<_l)$. Therefore $\yy_1$ and $\yy_2$ cannot correspond to the same gap in $\overline{\mathsf{St}}(\xx,j;\sB)$, thus proving the uniqueness of $\yy$.
\end{proof}

Note that the left $\N$-string produced in Proposition \ref{gaps of OST are gaps of hammocks} corresponding to a gap in $\OST{}$ lies in $\EOST$. Conversely, given $\xx\in\EOST$, the technique used in Proposition \ref{gaps of hammocks are gaps of OST} produces a gap in $\OST{}$.
\begin{cor}\label{completion of OST}
Suppose $\sB\in\QBa_i(\xx_0)$. Then $\com(\OST{})\cong\EOST\sqcup\OST{}$. In particular, if $\sB$ is minimal for $(\xx_0,i)$ then $\com(\OST{})\cong\ESTB\sqcup\OST{}$ thanks to Remark \ref{EOSTsubsetESTB and equal iff minimal}.
\end{cor}
As a consequence of the above corollary, the map $c_\sB: H_l^i(\xx_0)\to\OST{}$ can be extended to a map $\HHlix\to\EOST\sqcup\OST{}$, which we again denote by $c_\sB$, where $c_\sB(\xx)$ is the longest left (possibly left $\N$-) substring of $\xx$ that lies in $\EOST\sqcup\OST{}$. As a consequence, $\com(\OST{})$ is a condensation of $\com(H_l^i(\xx_0))$ via the composition
$$\com(H_l^i(\xx_0))\xrightarrow{\cong}\HHlix\xrightarrow{c_\sB}\EOST\sqcup\OST{}\xrightarrow{\cong}\com(\OST{}).$$

Propositions \ref{gaps of hammocks are gaps of OST} and \ref{gaps of OST are gaps of hammocks} yield bijections
\begin{equation}\label{eqn: gaps of hammocks in terms of gaps of OST minimal}
\com(H_l^i(\xx_0))\cong\bigcup_{\substack{\xx\in H_l^i(\xx_0),\ j\in\{-1,1\},\\\sB\in\QBa_j(\xx),\\j=i\text{ if }\xx=\xx_0}}\com(\OsT{\xx}{j}{\sB})\cong\bigcup_{\substack{\xx\in H_l^i(\xx_0),\ j\in\{-1,1\},\\ \sB\text{ is minimal for }(\xx,j),\\j=i\text{ if }\xx=\xx_0}}\com(\OsT{\xx}{j}{\sB}).  
\end{equation}
Equation \eqref{eqn: gaps of hammocks in terms of gaps of OST minimal} shows that it is sufficient to study the order type of $\com(\OsT{\xx}{j}{\sB})$, where $\sB\in\QBa$ minimal for $(\xx,j)$ to understand the position of left $\N$-strings (or equivalently gaps in $H_l^i(\xx_0)$) in the extended hammock $\HHlix$. Henceforth we study the order type of $\com(\OST{})$, where $\sB\in\QBa$ is minimal for $(\xx_0,i)$, and subsequently the positions of the left $\N$-strings in it.

The following result is useful in determining the position of an almost periodic left $\N$-string of a certain form in the extended hammock $\HHlix$.
\begin{prop}\label{position of infinite strings in OST}
Suppose $\sB\in\QBa_i(\xx_0)$ and $\yy\in\ESTB$. Then the following statements are equivalent.
\begin{enumerate}
    \item There exists $\bb\in\BalB$ and a string $\uu$ such that  $\yy=\,^\infty\bb\uu\xx_0$.
    \item There exists $\xx\in H_l^i(\xx_0)$ with $\xx<_l\yy$ such that $c_\sB([\xx,\yy))\cong\omega$.
\end{enumerate}
\end{prop}
\begin{proof}
By Proposition \ref{eventually left N strings are in OST minimal}, there exists $\yy\sqsupset_l\ww\in H_l^i(\xx_0)$ such that $\sB$ is minimal for $(\ww,\theta(\yy\mid\ww))$ and $\yy\in\EST{\ww}{\theta(\yy\mid\ww)}{\sB}$.

$(1)\Rightarrow(2).$ Since $\yy=\,^\infty\bb\uu\xx_0$, there exists $N\in\N^+$ such that $\ww\sqsubset_l\bb^N\uu\xx_0\sqsubset_l\yy$. Since 
$\theta(\bb)=1$, Remark \ref{LB Rem 2} along with the fact that $\bb\in\BalB$ implies that for each $n\geq N$, $\LB^{k_n}(\bb^N\uu\xx_0)=\bb^n\uu\xx_0$ for some $k_n\in\N$. Recall the definition of $C_\sB$ stated before Proposition \ref{OT(beam)}. The preceding arguments in this proof imply $C_\sB(\bb^N\uu\xx_0)=C_\sB(\bb^n\uu\xx_0)$ for every $n\geq N$. Consequently, $\brac{1}{\LB}(C_\sB(\bb^N\uu\xx_0))=\,^\infty\bb\uu\xx_0$. Remark \ref{hammockcoverrepr} gives that the interval $c_\sB([C_\sB(\bb^N\uu\xx_0),\,^\infty\bb\uu\xx_0))\cong\omega$.

$(2)\Rightarrow(1).$ Since $\xx<_l\yy$, we have $\theta(\yy\mid\xx)=1$. Both $\ww$ and $\xx\sqcap_l\yy$ are left substrings of $\yy$. Choose $\zz\sqsubset_l\yy$ with $\theta(\yy\mid\zz)=1$ such that $\xx\sqcap_l\yy\sqsubset_l\zz$ and $\ww\sqsubset_l\zz$. Clearly, $\xx<_l\zz<_l\yy$. Since $\yy\in\EST{\ww}{\theta(\yy\mid\ww)}{\sB}$, we have $\zz\in\OsT{\ww}{\theta(\yy\mid\ww)}{\sB}$. Therefore the interval $c_\sB([\zz,\yy))$ being an infinite suborder of $c_\sB([\xx,\yy))$ is also isomorphic to $\omega$. Note that $\LB^n(\zz)<_l\yy$ for every $n\in\N$, and therefore $\brac{1}{\LB}(\zz)\leq_l\yy$. If $\brac{1}{\LB}(\zz)\neq\yy$ then the string $\zz':=\brac{1}{\LB}(\zz)\sqcap_l\yy$ satisfies $\brac{1}{\LB}(\zz)<_l\zz'<_l\yy$, which is a contradiction to $c_\sB([\zz,\yy))\cong\omega$. Therefore $\yy=\brac{1}{\LB}(\zz)$, which implies $\yy=\,^\infty\bb\vv\zz$ for some $\bb\in\BalB$ and a string $\vv$ by Proposition \ref{OST to infty bux}. Since $\zz\in H_l^i(\xx_0)$, we have $\yy=\,^\infty\bb\uu\xx_0$ for some string $\uu$.
\end{proof}

Now we compute $\com(\OST{})$ when $\sB\in\QBa$ is minimal for $(\xx_0,i)$ and understand the position of gaps and the form of the left $\N$-strings corresponding to them.

When $\sB$ is non-domestic and minimal for $(\xx_0,i)$, the order type of $\OST{}$ was computed in Corollary \ref{wstrordertype} to be $\mathcal O:=\omega+\zeta\cdot\eta+\omega^*$. Recall from Corollary \ref{completion nondomestic beam} that 
$$\com(\mathcal O)\cong\omega+\mathbf{1}+\left(\sum_{r\in\lambda}T_r\right)+\mathbf{1}+\omega^*,\text{ where }T_r:=\begin{cases}
    \mathbf 1+\zeta+\mathbf 1&\text{ if $r\in\eta$},\\
    \mathbf 1&\text{ otherwise}.
\end{cases}$$
Recall from the end of \S~\ref{Completion of linear orders} that we partitioned the set $\gap(\mathcal O)$ as $\gap(\mathcal O)=\gap^+(\mathcal O)\sqcup \gap^-(\mathcal O)\sqcup\gap^0(\mathcal O)$. Along similar lines, we partition the set $\ESTB$ into three classes as follows.
\begin{align*}
\esTB{l}&:=\{\xx\in\ESTB\mid\xx=\,^\infty\bb\uu\xx_0\text{ for some }\bb\in\BalB\text{ and some string }\uu\},\\
\esTB{\lb}&:=\{\xx\in\ESTB\mid\xx=\,^\infty\bb\uu\xx_0\text{ for some }\bb\in\BalbB\text{ and some string }\uu\},\\
\esTB{0}&:=\ESTB\setminus(\esTB{l}\cup\esTB{\lb}).
\end{align*}

The following is a straightforward consequence of Corollary \ref{completion nondomestic beam}, Proposition \ref{position of infinite strings in OST} and its dual, and describes the positions of all left $\N$-strings in $\com(\OST{})$.
\begin{prop}\label{correspondence between diff forms of left N strings and gaps}
Let $\sB$ be non-domestic and minimal for $(\xx_0,i)$. Then using the notations described above, the order isomorphism $\ESTB\cup\OST{}\cong\com(\OST{})\cong\com(\mathcal O)$ restricts to the following order isomorphisms:
$$\esTB{l}\cong\gap^+(\mathcal O),\ \esTB{\lb}\cong\gap^-(\mathcal O)\text{ and }\esTB{0}\cong\gap^0(\mathcal O).$$
\end{prop}

\begin{rem}\label{leftNstring in domestic}
When $\sB\in\QBa_i(\xx_0)$ is domestic and $\bb$ is the unique element of $\sB$, then $\yy\in\ESTB$ if and only if there exists a string $\uu$ such that $\yy=\,^\infty\bb\uu\xx_0$. 

Furthermore, if $\sB$ is minimal for $(\xx_0,i)$, Proposition \ref{OST-STB is finite} gives that the set $\ESTB$ is finite. Corollary \ref{completion of OST} gives that $\com(\OST{})=\ESTB\sqcup\OST{}$, and hence $\gap(\OST{})\cong\ESTB$. In order to understand the location of finitely many elements of $\ESTB$ in $\com(\OST{})$, recall from Corollary \ref{wstrordertype} that $\OST{}\cong(\omega+\omega^*)\cdot\mathbf{n_\sB}$. The completion of the latter is $(\omega+\mathbf 1+\omega^*)\cdot \mathbf n_\sB$ (Example \ref{C(OT(domestic OST))}), which contains only finitely many extra points than in $\OST{}$.

Finally, Proposition \ref{position of infinite strings in OST} and its dual applied to the explicit form of the order type of $\com(\OST{})$ gives that each element $\yy\in\ESTB$ is of the form $\brac{1}{\LB}(\zz)$ for some $\zz\in\OST{\pm1}$ as well as of the form $\brac{1}{\LbB}(\zz')$ for some $\zz'\in\OST{\pm1}$, which is in agreement with the fact that $\BalB=\BalbB=\{\bb\}$ since $\sB$ is domestic.
\end{rem}
\bibliographystyle{plain}
\bibliography{main}
\vspace{0.2in}
\noindent{}Vinit Sinha\\
Indian Institute of Technology Kanpur\\
Uttar Pradesh, India\\
Email: \texttt{vinitsinha20@iitk.ac.in}

\vspace{0.2in}
\noindent{}Corresponding Author: Amit Kuber\\
Indian Institute of Technology Kanpur\\
Uttar Pradesh, India\\
Email: \texttt{askuber@iitk.ac.in}\\
Phone: (+91) 512 259 6721\\
Fax: (+91) 512 259 7500

\vspace{0.2in}
\noindent{}Annoy Sengupta\\
Indian Institute of Technology Kanpur\\
Uttar Pradesh, India\\
Email: \texttt{annoysgp20@iitk.ac.in}

\vspace{0.2in}
\noindent{}Bhargav Kale\\
Indian Institute of Technology Kanpur\\
Uttar Pradesh, India\\
Email: \texttt{kalebhargav@gmail.com}
\end{document}